\documentclass[11pt]{article}

\usepackage[margin=1in]{geometry}
\usepackage{amsmath,amssymb,amsthm,mathtools,bm}

\usepackage{enumitem}
\usepackage{algorithm}
\usepackage{algpseudocode}
\usepackage{microtype}
\usepackage[hidelinks]{hyperref}
\newcommand{\keywords}[1]{%
  \par\noindent\textbf{Keywords:} #1
}
\usepackage{booktabs}
\usepackage{makecell}
\usepackage{graphicx}
\newcommand{\safeincludegraphics}[2][]{%
  \IfFileExists{#2}{\includegraphics[#1]{#2}}{%
    \fbox{\parbox[c][0.22\textheight][c]{0.92\linewidth}{\centering
      Figure file not included in the source bundle:\\[3pt]
      {\ttfamily\detokenize{#2}}}}%
  }%
}

\usepackage{tabularx}
\usepackage{array}

\newcolumntype{Y}{>{\raggedright\arraybackslash}X}
\newcolumntype{M}[1]{>{\raggedright\arraybackslash}m{#1}}

\usepackage{tikz}
\usetikzlibrary{positioning,arrows.meta,fit,calc}

\usepackage{subcaption}
\newtheorem{theorem}{Theorem}[section]
\newtheorem{proposition}[theorem]{Proposition}
\newtheorem{lemma}[theorem]{Lemma}
\newtheorem{corollary}[theorem]{Corollary}
\newtheorem{assumption}[theorem]{Assumption}
\theoremstyle{definition}

\newtheorem{example}[theorem]{Example}
\theoremstyle{remark}
\newtheorem{remark}[theorem]{Remark}

\newcommand{\ReMID}{\textsc{ReMID}}

\title{Fully First-Order Optimization for Implicit Multilevel Compositions via Recursive Perturbations}
\author{%
  Xiaoyu Wang\textsuperscript{1} and Shiqian Ma\textsuperscript{2}\\[0.3em]
  \textsuperscript{1}University of Chinese Academy of Sciences \quad 
  \textsuperscript{2}Johns Hopkins University\\[0.3em]
  \texttt{wangxiaoyu@ucas.ac.cn}\quad
  \texttt{sqma@jhu.edu}%
}

\date{}

\begin{document}
\maketitle

\begin{abstract}
We study optimization over a finite composition of implicit response maps
generated by smooth potentials with strongly monotone signed gradients.
This class includes chains mixing strongly convex minimization
and strongly convex--strongly concave saddle responses.  We propose a fully first-order method based on recursively
perturbed response solves. The central issue is stability under inexact response solves: a downstream
adjoint error changes the perturbation vector of the next upstream perturbed problem.
We establish coupled bounds for the forward response errors and backward errors in the scaled response differences
whose coefficients remain bounded as the perturbation parameter $\delta$
decreases.  A triangular Lyapunov construction combines these bounds with
the $O(\delta)$ bias of the exact perturbation estimator.  For the
warm-started method with finitely many inner iterations, this gives $\frac{1}{K}\sum_{k=0}^{K-1}\|\nabla\Phi(x^k)\|^2
\leq O(K^{-1})+O(\delta^2).$
Thus, for fixed problem data and finite depth, $\delta=O(\epsilon)$ yields an
$O(\epsilon^{-2})$ outer-iteration dependence.  For strongly convex
minimization chains, we also derive condition-number-dependent first-order
oracle complexity bounds.
\end{abstract}



\section{Introduction}
\label{sec:intro}

Multilevel optimization arises in machine learning, optimal control,
and sequential decision-making, where the solution of one optimization
or equilibrium problem parameterizes the next.
An upstream decision is evaluated through the resulting sequence of
responses, so its influence on terminal performance must be traced
through the entire chain.

We consider a finite composition of implicit response maps
$S_i:\mathbb{R}^{d_{i-1}}\to\mathbb{R}^{d_i}$:
\begin{equation}
z_0^\star(x):=x,
\qquad
z_i^\star(x):=S_i\bigl(z_{i-1}^\star(x)\bigr),
\qquad i=1,\ldots,M.
\label{eq:intro-implicit-chain}
\end{equation}
The resulting outer problem is
\begin{equation}
\min_{x\in\mathbb{R}^{d_0}}
\quad
\Phi(x):=f\bigl(x,z_M^\star(x)\bigr),
\label{eq:intro-outer-problem}
\end{equation}
where $f$ measures terminal performance and may also depend directly on
$x$. The response maps and their Jacobians are not assumed to be
available explicitly; each response must instead be obtained by solving
its defining problem. We focus on smooth, single-valued chains of the
form \eqref{eq:intro-implicit-chain}; general nested hierarchical
programs need not admit this representation.

Even when the response chain can be evaluated, differentiating it remains
a substantially harder task. At every level, standard implicit
differentiation requires second-order derivative actions together with
the solution of a linearized optimality system, and these operations
compound across the chain. This is particularly restrictive in
large-scale settings, where the response problems may be accessible only
through first-order solvers. Differentiating through the solver
iterations avoids separate linearized solves, but its computational and
memory costs grow with the unrolling horizon, and it yields sensitivities
of the truncated solver trajectories rather than of the exact response
maps. These limitations raise a fundamental question: can one optimize
the exact composite objective using only first-order evaluations and a
fixed finite number of response updates per outer iteration?

\subsection{Motivating applications}
\label{subsec:motivating-applications}

Response-chain structures arise in multilevel learning pipelines
\cite{somayajula2022multilevel}, end-to-end data augmentation and
task-adaptive pretraining
\cite{fang2022endtoend,zhang2025tapweight}, and multistage decision and
market models \cite{persak2024decision,xia2026sequentialmarkets}.
Although the complete formulations may be constrained, nonconvex, or
discrete, the following examples identify smooth structured instances
covered by our framework.

\paragraph{Sequentially regularized model adaptation.}
A model adapted across sites or tasks may be regularized toward the model
from the preceding stage
\cite{kirkpatrick2017overcoming,lenga2020continual}:
\begin{equation}
\label{eq:sequential-adaptation}
z_i^\star
=\operatorname*{arg\,min}_{z}
\left\{
\ell_i(z;\mathcal D_i)
+\frac{\omega_i}{2}
(z-z_{i-1}^\star)^\top W_i(z-z_{i-1}^\star)
+\frac{\mu_i}{2}\|z\|^2
\right\},
\qquad z_0^\star=x.
\end{equation}
The initial model $x$ is chosen according to validation performance after
the full adaptation sequence. If $\ell_i$ is twice continuously
differentiable and convex, $\omega_i\ge0$, $W_i\succeq0$, and $\mu_i>0$,
then each stage defines a unique smooth response. Our domain-adaptation
experiment uses this structure with fixed features and convex stage
losses.

\paragraph{Saddle-point response chains.}
A stage may instead return an equilibrium between competing decisions,
as in parametrized minimax models
\cite{zheng2026bilevel,shen2026penalty}.
Given the preceding response $z_{i-1}^\star$, let $z_i^\star=(y_i^\star,p_i^\star)$
satisfy
\begin{equation}
\label{eq:sequential-saddle-chain}
\min_{y_i}\max_{p_i}
\Psi_i(y_i,p_i;z_{i-1}^\star), \qquad
z_i^\star:=(y_i^\star,p_i^\star),
\qquad
z_0^\star=x,
\end{equation}
where $\Psi_i$ is uniformly strongly convex in $y_i$ and
strongly concave in $p_i$.
Passing $z_i^\star$ to the next stage produces a response chain.
Smooth, uniformly strongly convex--strongly concave potentials yield
unique differentiable responses within our framework.

\paragraph{Layer-wise thermal control.}
In laser powder bed fusion, each layer inherits a thermal state from its
predecessor, and process parameters are chosen using thermal models and
interlayer measurements
\cite{alam2021slmcontrol,riensche2022feedforward,kavas2023layer,
ren2025interlayer}.
A representative stage problem is
\begin{equation}
\label{eq:lpbf-sequential-control}
(c_i^\star,\theta_i^\star)
\in
\operatorname*{arg\,min}_{c_i,\theta_i}
\mathcal J_i(c_i,\theta_i;\theta_{i-1}^\star)
\quad\text{subject to}\quad
\mathcal A_i(c_i,\theta_i;\theta_{i-1}^\star)=0,
\end{equation}
where $c_i$ denotes the layer-wise control and $\mathcal A_i=0$ represents discretized heat-transfer dynamics.
A locally smooth state reduction and a uniformly strongly convex reduced
control problem define a smooth response to
$\theta_{i-1}^\star$, whose terminal state becomes the input to the next
layer. This example motivates the response-chain structure; our
experiments use structured simplified instances rather than a complete
manufacturing-process model.

For the structured instances considered here, first-order iterations
can approximate each response. Differentiating their composition,
however, typically involves second-order operations or unrolling
solver trajectories. We introduce \ReMID{} (Recursive Multilevel Implicit
Differentiation), which approximates reverse sensitivity through
recursively perturbed response solves using only first-order information.
We next describe the response structure that enables this approach
and the challenge of controlling errors from inexact recursive solves.

\subsection{Response class and analytical challenge}
\label{subsec:response-class-challenge}

The minimization and saddle responses above share a common
potential-based formulation:
\begin{equation}
\label{eq:intro-response-preview}
\nabla_1\Psi_i\bigl(S_i(u),u\bigr)=0.
\end{equation}
We assume that each level admits a fixed symmetric
orthogonal matrix $P_i$ such that
$z\mapsto P_i\nabla_1\Psi_i(z,u)$ is uniformly strongly monotone
in $u$. For minimization responses, $P_i=I$ recovers the standard
uniform strong convexity condition.
For saddle responses, $P_i=\operatorname{diag}(I,-I)$,
with signs corresponding to the minimization and maximization
variables, recovers uniform strong convexity--concavity.
Under the regularity assumptions of Section~\ref{sec:model},
this condition ensures well-posed responses and contractive
first-order updates.

\paragraph{Perturbations and reverse sensitivity.}
Reverse differentiation propagates the terminal gradient backward
through the response chain:
\begin{equation}
\label{eq:intro-adjoint-preview}
v_M^\star=\nabla_2f\bigl(x,z_M^\star(x)\bigr),
\qquad
v_{i-1}^\star
=DS_i\bigl(z_{i-1}^\star(x)\bigr)^\top v_i^\star,
\quad i=M,\ldots,1,
\end{equation}
with $\nabla\Phi(x)=\nabla_1f(x,z_M^\star(x))+v_0^\star$.
The potential structure enables these reverse actions to be recovered
through response perturbations.
For a general implicit equation $F(z,u)=0$, reverse differentiation
requires a transpose-inverse action involving $(D_zF)^{-\top}$,
whereas an additive response perturbation involves $(D_zF)^{-1}$.
These generally differ.
Here the state Jacobian of the response equation is the symmetric
Hessian $\nabla_{11}^2\Psi_i$, so its inverse and transpose inverse
coincide.
Specifically, for fixed $u$ and a downstream adjoint vector $w$,
the perturbed response
$\nabla_1\Psi_i(z_{i,\delta},u)+\delta w=0$ satisfies
\begin{equation}
\label{eq:intro-perturbation-preview}
\frac{\nabla_2\Psi_i(z_{i,\delta},u)
-\nabla_2\Psi_i(S_i(u),u)}{\delta}
=DS_i(u)^\top w+O(\delta)
\end{equation}
under the regularity assumptions.
Thus, perturbed response solves and first-order gradient evaluations
approximate the reverse action without Hessian--vector products
or separate adjoint linear-system solves.
Related single-response constructions are known
\cite{sow2022hessianfree,kwon2023fullyfirst,kornowski2024constrained,
zhao2026ffolayer}.
The difficulty here is the interaction between recursive
composition and inexact response computation.

\paragraph{Recursive finite-solve errors.}
At depth $M\ge2$, an inexact downstream adjoint supplies the
perturbation vector for the next upstream response problem.
Its error therefore changes the upstream response target, and the
resulting solve generates the next adjoint approximation.
Consequently, single-response estimates with a prescribed
perturbation vector do not directly close the multilevel error
recursion. Moreover, bounding the perturbed and unperturbed
response errors separately before division by $\delta$ can yield
terms of order $r/\delta$, where $r$ denotes an absolute
response-solve error
(Proposition~\ref{prop:two-level-finite-solve-obstruction}).

\ReMID{} refreshes the unperturbed responses in forward order and
then updates the perturbed responses in backward order.
At each level, the two branches share the same refreshed input,
stepsize, and inner budget, and remain warm started across outer
iterations. This pairing allows their updates to be subtracted
to obtain a recursion directly for the scaled response difference
$(z_{i,\delta}-z_i)/\delta$, avoiding bounds that divide separate
absolute solve errors by $\delta$.
The resulting tracking bounds have coefficients that remain
bounded as $\delta\downarrow0$, while accounting for both forward
response errors and backward errors in the recursively generated
adjoint approximations.
A weighted Lyapunov construction closes these coupled recursions.
\subsection{Contributions}
\label{subsec:contributions}
Our contributions are threefold.
\begin{itemize}[leftmargin=1.6em]
\item \textbf{Recursive first-order differentiation.}
We construct a recursive perturbation estimator for the potential-based
response class using only gradient evaluations, without Hessian--vector
products or separate linearized-system solves. The construction applies at arbitrary fixed depth and accommodates
both strongly convex minimization and strongly convex--strongly
concave saddle responses within the same chain. For exact response solves, we show that the estimator recovers
the exact adjoint as $\delta\downarrow0$, establish a uniform
$O(\delta)$ one-sided bias bound, and derive regularity estimates
for finite-solve tracking.

\item \textbf{Uniform tracking with fixed inner budgets.}
We establish tracking bounds for scaled paired-response
differences (Theorem~\ref{thm:auxiliary-tracking}) and combine
forward and backward error recursions through a weighted Lyapunov
argument (Theorem~\ref{thm:lyapunov-closure}). For any fixed admissible
upper radius $\delta_{\rm adm}$, the required inner budgets depend on the
problem data, chain depth, and solver contractions, but not on the chosen
$\delta\in(0,\delta_{\rm adm}]$. The result also accommodates paired
initialization without requiring $\delta$-dependent initial solve
accuracy.

\item \textbf{Stationarity and first-order oracle complexity.}
With fixed inner budget and paired initialization choices, we establish
\[
\frac1K\sum_{k=0}^{K-1}\|\nabla\Phi(x^k)\|^2
\leq O(K^{-1})+O(\delta^2)
\]
for the exact composite objective, with constants uniform in $\delta$ (Theorem~\ref{thm:global-stationarity}). For fixed problem data and depth,
choosing $\delta=c\epsilon$
for sufficiently small fixed $c>0$ yields an average squared
gradient norm at most $\epsilon^2$ in $O(\epsilon^{-2})$
outer iterations with fixed inner budgets.
In the homogeneous strongly convex minimization regime, we further obtain
\[
K=O\!\left(\kappa^{2M+1}\epsilon^{-2}\right),
\qquad
\mathcal N_{\rm FO}
=\widetilde O\!\left(M^2\kappa^{2M+2}\epsilon^{-2}\right),
\]
where $\kappa$ is the largest condition number among the
response problems and
$\mathcal N_{\rm FO}$ counts first-order oracle calls. The hidden constants may depend on the fixed depth $M$. At $M=1$, the oracle complexity reduces to
$\widetilde O(\kappa^4\epsilon^{-2})$, matching the
condition-number and accuracy dependence of the nonaccelerated
fully first-order bilevel bound of Chen et al.~\cite{chen2025f2ba}.
\end{itemize}

\subsection{Related work}
\label{subsec:related-work}

\paragraph{Bilevel differentiation and first-order alternatives.}
Bilevel programming has an extensive literature on formulations
and algorithms
\cite{dempe2002foundations,colson2007overview,luo1996mpec}.
Gradient methods typically differentiate either finite lower-level
optimization trajectories
\cite{domke2012generic,maclaurin2015reversible,franceschi2018bilevel}
or optimality systems
\cite{pedregosa2016hyperparameter,grazzi2020iteration,blondel2022implicit}.
Their convergence and complexity have been studied in
\cite{ghadimi2018approximation,ji2021bilevel,yang2021faster,
hong2023twotimescale,dagreou2022framework,dagreou2024lower}.
Hessian-free and fully first-order methods replace explicit
second-order computations with nearby-path, penalty, two-timescale,
or perturbation constructions
\cite{sow2022hessianfree,kwon2023fullyfirst,chen2025f2ba,
li2022fullysingle,liu2022bome,kornowski2024constrained}.
Related approaches use perturbed optimization solves for backward
differentiation \cite{zhao2026ffolayer} or zeroth-order queries
for response-sensitivity estimation \cite{jiang2026bizol}.
Our single-response mechanism belongs to this broader family.
The equality-constrained method of
Kornowski et al.~\cite{kornowski2024constrained}
uses solve accuracies that tighten with the perturbation radius.

\paragraph{Multilevel differentiation.}
Sato et al.~\cite{sato2021multilevel} differentiate finite
optimization trajectories and establish stationarity guarantees
for the resulting surrogate, together with asymptotic consistency
of optimal values and solutions as the inner iteration counts grow.
Shafiei et al.~\cite{shafiei2024trilevel} prove convergence and
rate results for convex hierarchical problems defined by nested
solution sets, using monotone operators and fixed-point methods.
Gu and Chen~\cite{gu2024differentiable} develop hierarchical
implicit differentiation with gradient-computation complexity
and convergence analyses; their construction uses inverse
linearized-system actions.

The \textsc{Betty} framework~\cite{choe2023betty} composes
approximate best-response Jacobians and supports finite-difference
approximations. Related recursive constructions appear in
\cite{xie2023classroom,guo2025mlomae}.
These works demonstrate recursive differentiation in software and
learning applications. Our focus is the finite-computation analysis:
for potential-based response chains, we establish stationarity
and first-order oracle complexity guarantees for the exact
response-composition objective, using fixed inner update budgets
that are independent of the perturbation radius.

\paragraph{Inexact tracking and compositional optimization.}
Dizon and Valkonen~\cite{dizon2024differential} develop an abstract
framework for inexact differential tracking, accommodating inner
and adjoint iterations, including iterative linear-system solves.
Our analysis instead controls scaled differences of recursively
perturbed response pairs uniformly in the perturbation radius,
using only first-order response updates with radius-independent
inner budgets.
In explicit nested-composition optimization, component values and
derivatives are accessed through oracles
\cite{wang2017scgd,yang2019multilevel,zhang2021nested,
balasubramanian2022multilevel,jiang2022optimal};
in our setting, the component maps are implicit responses
approximated by iterative solvers.

\paragraph{Differentiation beyond minimization.}
Implicit differentiation extends to root equations and KKT systems
\cite{blondel2022implicit}, as well as game equilibria and
Stackelberg learning dynamics
\cite{ling2018game,li2020endtoendgames,fiez2020stackelberg}.
For bilevel problems with minimax lower levels, first-order methods
have been developed using regularized gap functions
\cite{yao2025gap} and penalty formulations
\cite{shen2026minimax}. Our response model permits minimization and strongly
convex--strongly concave saddle problems within the same chain.

\paragraph{Organization.}
Section~\ref{sec:model} states the response model and exact adjoint.
Section~\ref{sec:perturbation} develops the perturbation targets.
Sections~\ref{sec:algorithm}--\ref{sec:convergence} establish finite-solve
tracking and complexity guarantees. Section~\ref{sec:numerics} presents
experiments. Additional proofs and implementation details are in the
appendices.

\paragraph{Notation.}
We write $\langle\cdot,\cdot\rangle$ and $\|\cdot\|$ for the Euclidean
inner product and norm, respectively, and use the same norm notation for
the induced matrix norm. The symbols $I_d$ and $A^\top$ denote the
$d\times d$ identity matrix and the transpose of $A$; the subscript on
$I_d$ is omitted when its dimension is clear, and
$A^{-\top}:=(A^{-1})^\top$. For symmetric matrices,
$A\preceq B$ means that $B-A$ is positive semidefinite. We set
$\operatorname{diag}(A,B)$ for the block-diagonal matrix with blocks
$A$ and $B$, $\operatorname{sym}(A):=(A+A^\top)/2$, and use
$\mathbf 1_{\mathcal E}$ for the indicator of an event or condition
$\mathcal E$. The notation $Dh$ denotes the Jacobian of a differentiable
map $h$. For a scalar function $h$ of two arguments, $\nabla_a h$ is its
gradient with respect to argument $a$, and
$\nabla_{ab}^2h:=D_{\mathrm{arg}_b}(\nabla_a h)$. Empty products and
empty sums equal one and zero, respectively. We use standard asymptotic
notation, with $\widetilde O(\cdot)$ suppressing logarithmic factors.

\section{Implicit Multilevel Compositions}
\label{sec:model}

We specify the implicit response maps in
\eqref{eq:intro-implicit-chain} and the assumptions used to analyze
the composite objective \eqref{eq:intro-outer-problem}.
We then derive the exact adjoint recursion, which motivates
the perturbation construction in Section~\ref{sec:perturbation}.

\subsection{Response model and assumptions}
\label{subsec:problem-formulation}

At level $i=1,\ldots,M$, let
$\Psi_i:\mathbb R^{d_i}\times\mathbb R^{d_{i-1}}\to\mathbb R$
be a potential, and let
$P_i=P_i^\top=P_i^{-1}\in\mathbb R^{d_i\times d_i}$.
Define the signed-gradient operator
$T_i(z,u):=P_i\nabla_1\Psi_i(z,u)$. The response map is
\begin{equation}
\label{eq:implicit-solution-map}
S_i(u):=\textnormal{the unique }z\in\mathbb R^{d_i}
\textnormal{ satisfying }T_i(z,u)=0,
\end{equation}
whose existence and uniqueness follow from the assumptions below.
Since $P_i$ is nonsingular, this equation is equivalent to
\begin{equation}
\label{eq:implicit-stationarity}
\nabla_1\Psi_i(S_i(u),u)=0.
\end{equation}
\begin{assumption}[Implicit-layer regularity]
\label{ass:implicit-layers}
For each $i=1,\ldots,M$, there are constants
$\mu_i,L_i>0$ and $\rho_i\ge0$ such that the following conditions hold.

\begin{enumerate}[label=\textnormal{(\roman*)}]
\item
For every $u\in\mathbb R^{d_{i-1}}$,
the map $T_i(\cdot,u)$ is uniformly $\mu_i$-strongly monotone:
\[
\left\langle T_i(z,u)-T_i(z',u),z-z' \right\rangle \ge \mu_i\|z-z'\|^2.
\]

\item
The function $\Psi_i$ is twice continuously differentiable and $\|\nabla^2\Psi_i(z,u)\|
\le L_i$ for all $(z,u)$.

\item
The Hessian of $\Psi_i$ is $\rho_i$-Lipschitz:
$\|\nabla^2\Psi_i(z,u)-\nabla^2\Psi_i(z',u')\| \le \rho_i \left( \|z-z'\|+\|u-u'\| \right).$
\end{enumerate}
\end{assumption}
The formulation $T_i=P_i\nabla_1\Psi_i$ includes the following two canonical
response classes, subject to the derivative bounds above.

\begin{example}[Strongly convex minimization responses]
\label{prop:minimization-layer}
Let $P_i=I$ and $\Psi_i(z,u)=g_i(z,u)$, where
$g_i(\cdot,u)$ is uniformly $\mu_i$-strongly convex for every $u$.
Then $T_i(z,u)=\nabla_1g_i(z,u)$ is $\mu_i$-strongly monotone, and
the response map is
$S_i(u) = \operatorname*{arg\,min}_{z\in\mathbb R^{d_i}}g_i(z,u).$
\end{example}

\begin{example}[Strongly convex--strongly concave saddle responses]
\label{prop:saddle-layer}
Let $z=(y,p)\in\mathbb R^{m_i}\times\mathbb R^{n_i}$, with
$d_i=m_i+n_i$, and let
$P_i:=\operatorname{diag}(I_{m_i},-I_{n_i})$.
Write $\Psi_i(y,p,u):=\Psi_i((y,p),u)$.
Suppose that, for every $u\in\mathbb R^{d_{i-1}}$,
$\Psi_i(\cdot,\cdot,u)$ is uniformly $\mu_{y,i}$-strongly convex in $y$
and uniformly $\mu_{p,i}$-strongly concave in $p$. Then
\[
T_i(z,u)
=
\begin{pmatrix}
\nabla_y\Psi_i(y,p,u)\\
-\nabla_p\Psi_i(y,p,u)
\end{pmatrix}
\]
is $\mu_i$-strongly monotone in $z$, where
$\mu_i:=\min\{\mu_{y,i},\mu_{p,i}\}$.
Consequently,
$S_i(u)=(y_i^\star(u),p_i^\star(u))$ is the unique saddle point of
$\Psi_i(\cdot,\cdot,u)$.
\end{example}

\paragraph{Relation to standard response models.}
For $P_i=I$, Assumption~\ref{ass:implicit-layers}(i) is
equivalent to uniform strong convexity of $\Psi_i(\cdot,u)$,
as commonly assumed in smooth bilevel optimization
\cite{ji2021bilevel,kwon2023fullyfirst,chen2025f2ba}.
For $P_i=\operatorname{diag}(I,-I)$, it is equivalent to
uniform strong convexity--concavity of the saddle potential.
The corresponding strongly monotone saddle-gradient formulation
is standard in saddle-point and variational-inequality methods
\cite{ostroukhov2022tensor}.
Thus the signed-gradient formulation covers both response types,
which may coexist in a chain.
Strong monotonicity is imposed on $T_i$, while the state Hessian
of $\Psi_i$ may be indefinite.

\begin{assumption}[Outer-function regularity]
\label{ass:outer-function}
The outer objective $f$ is twice continuously differentiable.
There exist constants $C_f,L_f>0$ and $\rho_f\ge0$ such that
$\|\nabla_2f(x,z)\| \le C_f$, $\|\nabla^2f(x,z)\| \le L_f$
for all $(x,z)$, and
$\|\nabla^2f(x,z)-\nabla^2f(x',z')\| \le \rho_f \bigl( \|x-x'\|+\|z-z'\| \bigr)$
for all $(x,z)$ and $(x',z')$.
\end{assumption}
The derivative bounds in these assumptions follow familiar
smooth bilevel formulations: Lipschitz gradients and lower-level
Hessians are used in
\cite{ji2021bilevel,kwon2023fullyfirst,chen2025f2ba},
and the bounded terminal gradient
$\|\nabla_2f\|\le C_f$ is also assumed in
\cite{kwon2023fullyfirst,chen2025f2ba}.
We additionally impose Lipschitz continuity of the outer Hessian
to control the regularity of the recursive perturbation targets;
this condition also appears in
\cite{kwon2023fullyfirst}.
All bounds are global and uniform in the response inputs.

\paragraph{First-order information.}
We assume access to the gradient blocks
$\nabla_1\Psi_i$, $\nabla_2\Psi_i$, $\nabla_1f$, and $\nabla_2f$.
A first-order oracle call evaluates one such block.
The method developed below uses these evaluations without querying
Hessians or Hessian--vector products, or differentiating through
inner solver iterations.

\subsection{Regularity of the response maps}
\label{subsec:response-regularity}

Set $\kappa_i:=L_i/\mu_i\ge1$ and define
\begin{equation}
\label{eq:H-B-def}
H_i(u)
:=
\nabla_{11}^2\Psi_i\bigl(S_i(u),u\bigr),
\qquad
B_i(u)
:=
\nabla_{21}^2\Psi_i\bigl(S_i(u),u\bigr).
\end{equation}
Here $H_i(u)\in\mathbb R^{d_i\times d_i}$ is symmetric and
$B_i(u)\in\mathbb R^{d_{i-1}\times d_i}$; Hessian symmetry gives
$B_i(u)^\top=\nabla_{12}^2\Psi_i(S_i(u),u)$.

\begin{proposition}[Response-map regularity]
\label{prop:response-regularity}
Under Assumption~\ref{ass:implicit-layers}, each response map $S_i$
is well defined and continuously differentiable, with
\begin{equation}
\label{eq:implicit-jacobian-general}
DS_i(u)=-H_i(u)^{-1}B_i(u)^\top.
\end{equation}
Moreover,
\begin{equation}
\label{eq:Hinv-general}
\|H_i(u)^{-1}\|\le\mu_i^{-1},\qquad
\|DS_i(u)\|\le\kappa_i,
\end{equation}
and, for all $u,u'$,
\[
\|S_i(u)-S_i(u')\|\le\kappa_i\|u-u'\|,\qquad
\|DS_i(u)-DS_i(u')\|\le L_{J_i}\|u-u'\|,
\]
where $L_{J_i}:=\rho_i(1+\kappa_i)^2/\mu_i$.
\end{proposition}

\begin{proof}
Well-posedness follows from continuity and strong monotonicity on
$\mathbb R^{d_i}$; a direct contraction argument is given in
Appendix~\ref{app:supporting-proofs}. Set
$A_i:=D_zT_i(S_i(u),u)=P_iH_i(u)$. Then
\[
\mu_i\|h\|^2\le\langle A_i h,h\rangle
\le\|A_i h\|\,\|h\|=\|H_i(u)h\|\,\|h\|.
\]
Thus $\|H_i(u)h\|\ge\mu_i\|h\|$, so $H_i(u)$ is nonsingular
and $\|H_i(u)^{-1}\|\le\mu_i^{-1}$, even when it is indefinite.
The implicit-function theorem and uniqueness give a global $C^1$ map.
Differentiating \eqref{eq:implicit-stationarity} yields
$H_iDS_i+B_i^\top=0$, and $\|B_i\|\le L_i$ gives the Jacobian bound.
The remaining Lipschitz estimates follow from the inverse-difference
identity, as detailed in Appendix~\ref{app:supporting-proofs}.
\end{proof}

By Proposition~\ref{prop:response-regularity}, the exact chain is
Lipschitz continuous in $x$. For convenience, let
$L_{z,0}:=1$, $L_{z,i}:=\prod_{j=1}^i\kappa_j$, $i=1,\ldots,M.$
Then
$\|z_i^{\star}(x)-z_i^{\star}(x')\|
\le L_{z,i}\|x-x'\|$.

\begin{lemma}[Smoothness of the outer objective]
\label{lem:outer-objective-smoothness}
Under Assumptions~\ref{ass:implicit-layers} and \ref{ass:outer-function}, the composite objective $\Phi$ is
continuously differentiable and
$\|\nabla\Phi(x)-\nabla\Phi(x')\| \le L_\Phi\|x-x'\|,$
where we may take
$L_\Phi := L_f(1+L_{z,M})^2 + C_f \sum_{r=1}^M L_{J_r}L_{z,r-1}^2 \prod_{j=r+1}^M\kappa_j.$
\end{lemma}

\begin{proof}
Writing $J(x):=Dz_M^\star(x)$, the chain rule gives
$\nabla\Phi=\nabla_1f+J^\top\nabla_2f$.
The first term in $L_\Phi$ controls variation of the gradient blocks
of $f$; the second controls variation of $J$ multiplied by the bounded
terminal gradient. The full estimate is in
Appendix~\ref{app:supporting-proofs}.
\end{proof}

\subsection{Adjoint recursion and response perturbations}
\label{subsec:adjoint}

We call $\nabla\Phi(x)$ the hypergradient and the vectors propagated
backward through the responses adjoint vectors.
Along the exact chain \eqref{eq:intro-implicit-chain}, we use the shorthand
\begin{equation}
\label{eq:H-B-chain}
H_i^{\star}(x)
:=
H_i\bigl(z_{i-1}^{\star}(x)\bigr),
\qquad
B_i^{\star}(x)
:=
B_i\bigl(z_{i-1}^{\star}(x)\bigr).
\end{equation}
By Proposition~\ref{prop:response-regularity},
$DS_i\bigl(z_{i-1}^{\star}(x)\bigr)
= -[H_i^{\star}(x)]^{-1}[B_i^{\star}(x)]^\top.$

\begin{proposition}[Exact adjoint recursion]
\label{prop:exact-adjoint}
Under Assumptions~\ref{ass:implicit-layers} and \ref{ass:outer-function}, define
$v_M^{\star}(x):=\nabla_2f(x,z_M^{\star}(x))$ and, recursively for
$i=M,\ldots,1$,
\begin{equation}
\label{eq:exact-adjoint}
v_{i-1}^{\star}(x)
=
-B_i^{\star}(x)[H_i^{\star}(x)]^{-1}v_i^{\star}(x).
\end{equation}
Then the gradient of the outer objective in
\eqref{eq:intro-outer-problem} is
\begin{equation}
\label{eq:exact-hypergradient}
\nabla\Phi(x)
=
\nabla_1f(x,z_M^{\star}(x))
+
v_0^{\star}(x).
\end{equation}
\end{proposition}

\begin{proof}
The chain rule gives
$v_{i-1}^\star=DS_i(z_{i-1}^\star)^\top v_i^\star$ and
$\nabla\Phi=\nabla_1f+v_0^\star$.
Substitute \eqref{eq:implicit-jacobian-general} and use symmetry of
$H_i^\star$ to obtain \eqref{eq:exact-adjoint}.
\end{proof}

Equation~\eqref{eq:exact-adjoint} exposes the computational bottleneck: each
implicit level requires the action of $[H_i^{\star}(x)]^{-1}$.  This is the usual
inverse-Hessian action for minimization layers; for saddle layers
$H_i^{\star}$ may be indefinite but remains nonsingular.

\begin{proposition}[Adjoint action through response perturbations]
\label{prop:adjoint-compatibility}
Under Assumption~\ref{ass:implicit-layers}, fix a level $i$, a parameter $u$,
and a vector $w\in\mathbb R^{d_i}$.  Let $z_{i,\delta}$ be the local response
defined for $\delta$ near zero by
$\nabla_1\Psi_i(z_{i,\delta},u)+\delta w=0$, with
$z_{i,0}=S_i(u)$.  Then
\begin{equation}
\label{eq:adjoint-compatible-direction}
DS_i(u)^\top w
=
B_i(u)
\left.\frac{d z_{i,\delta}}{d\delta}\right|_{\delta=0}
=
-B_i(u)H_i(u)^{-1}w .
\end{equation}
Thus the inverse Hessian action generated by a response
perturbation is exactly the one required by reverse differentiation.
\end{proposition}

\begin{proof}
Differentiating
$\nabla_1\Psi_i(z_{i,\delta},u)+\delta w=0$ at $\delta=0$ gives
$\left.dz_{i,\delta}/d\delta\right|_{\delta=0}
=-H_i(u)^{-1}w$.
Since $H_i(u)$ is symmetric,
Proposition~\ref{prop:response-regularity} gives
$DS_i(u)^\top w=-B_i(u)H_i(u)^{-1}w$.
\end{proof}

The signed Jacobian $P_iH_i$ need not be symmetric. The adjoint identity
uses the symmetric state Jacobian of $\nabla_1\Psi_i=0$, not that of
the signed equation $T_i=0$. For a generic implicit equation $F(z,u)=0$, let
$A:=D_zF$ and $C:=D_uF$, with $A$ nonsingular.
Reverse sensitivity requires $-C^\top A^{-\top}w$.
In contrast, the perturbation $F(z_\delta,u)+\delta w=0$
gives $\dot z_0=-A^{-1}w$, and hence
$C^\top\dot z_0=-C^\top A^{-1}w$.
These expressions need not coincide when $A$ is nonsymmetric.

Finally, \eqref{eq:adjoint-compatible-direction} can be written using
only first-order evaluations:
\[
\lim_{\delta\to0}
\frac{\nabla_2\Psi_i(z_{i,\delta},u)
-\nabla_2\Psi_i(S_i(u),u)}{\delta}
=DS_i(u)^\top w.
\]
A related perturbation-based finite-difference construction
is used for linearly constrained bilevel problems
\cite{kornowski2024constrained}.
Section~\ref{sec:perturbation} applies this identity recursively,
with each gradient difference supplying the perturbation vector
for the preceding response problem.

\section{Recursive First-Order Perturbation}
\label{sec:perturbation}

We construct recursive perturbations along the exact unperturbed
response chain, proceeding backward from the terminal level.
At each level, a difference of input gradients generates the
perturbation vector for the preceding response problem.
We use $t$ for the continuous perturbation parameter and $\delta$
for the radius used by the algorithm.
We prove recovery of the exact adjoint as $t\downarrow0$,
an $O(\delta)$ bias bound, and regularity estimates uniform in
$\delta$ that support the finite-solve tracking analysis.

\subsection{Recursive perturbation and consistency}
\label{subsec:recursive-perturbation}

Fix an admissible upper radius satisfying
\begin{equation}
\label{eq:admissible-radius}
0<\delta_{\rm adm}<\frac{\mu_M}{L_f},
\qquad
\bar\mu_M:=\mu_M-\delta_{\rm adm}L_f>0.
\end{equation}
We construct the perturbation paths for
$t\in[0,\delta_{\rm adm}]$.
The radius restriction preserves uniform strong monotonicity
of the terminal perturbed operator.

Fix $x$ and its exact unperturbed response chain.
Each perturbed response uses the same input
$z_{i-1}^\star(x)$ as its unperturbed counterpart.
For $t\in(0,\delta_{\rm adm}]$, we begin at the terminal level
by defining $z_{M,t}^{\star}(x)$ through
\begin{equation}
\label{eq:terminal-pert-path-general}
\nabla_1\Psi_M\bigl(z_{M,t}^{\star}(x),z_{M-1}^{\star}(x)\bigr)
+
t\nabla_2f\bigl(x,z_{M,t}^{\star}(x)\bigr)
=
0.
\end{equation}
Set
$v_{M,t}^{\star}(x)
:=\nabla_2f\bigl(x,z_{M,t}^{\star}(x)\bigr)$.
At each preceding level $i=M-1,\ldots,1$, the downstream
gradient difference supplies the perturbation vector
\begin{equation}
\label{eq:general-vfd}
v_{i,t}^{\star}(x)
:=
\frac{
\nabla_2\Psi_{i+1}\bigl(z_{i+1,t}^{\star}(x),z_i^{\star}(x)\bigr)
-
\nabla_2\Psi_{i+1}\bigl(z_{i+1}^{\star}(x),z_i^{\star}(x)\bigr)
}{t},
\end{equation}
and define the corresponding perturbed response $z_{i,t}^{\star}(x)$ by
\begin{equation}
\label{eq:general-intermediate-pert}
\nabla_1\Psi_i\bigl(z_{i,t}^{\star}(x),z_{i-1}^{\star}(x)\bigr)
+
t v_{i,t}^{\star}(x)
=
0.
\end{equation}
Finally, the adjoint approximation returned from level $1$ is
\begin{equation}
\label{eq:general-v0fd}
v_{0,t}^{\star}(x)
:=
\frac{
\nabla_2\Psi_1\bigl(z_{1,t}^{\star}(x),x\bigr)
-
\nabla_2\Psi_1\bigl(z_1^{\star}(x),x\bigr)
}{t}.
\end{equation}
Evaluating the construction at the algorithmic radius
$\delta\in(0,\delta_{\rm adm}]$ gives the gradient estimator
with exact response solves:
\begin{equation}
\label{eq:finite-radius-gradient-estimator}
G_\delta^{\star}(x)
:=
\nabla_1f\bigl(x,z_M^{\star}(x)\bigr)
+
v_{0,\delta}^{\star}(x).
\end{equation}

Set
$z_{i,0}^{\star}(x):=z_i^{\star}(x)$.
Although the gradient differences are defined for $t>0$,
the following theorem shows that they extend continuously
to zero and recover the exact adjoint sequence.

\begin{theorem}[Recovery of the exact adjoint as $t\downarrow0$]
\label{thm:zero-radius-recovery}
Suppose Assumptions~\ref{ass:outer-function}
and~\ref{ass:implicit-layers} hold, and let $\delta_{\rm adm}$
satisfy \eqref{eq:admissible-radius}.
Then, for every $x$, the perturbed state paths
$z_{i,t}^{\star}(x)$ are uniquely defined for
$t\in[0,\delta_{\rm adm}]$ and $i=1,\ldots,M$.
Each path is continuously differentiable on this interval,
with one-sided derivatives at the endpoints.
Writing a dot for
differentiation with respect to $t$,
\begin{equation}
\label{eq:zero-radius-direction}
\dot z_{i,0}^{\star}(x)
=
-[H_i^{\star}(x)]^{-1}v_i^{\star}(x),
\qquad i=1,\ldots,M,
\end{equation}
where $\{v_i^{\star}(x)\}_{i=0}^M$ is the exact adjoint sequence from
Proposition~\ref{prop:exact-adjoint}. Moreover, the adjoint approximations
admit continuous extensions to $t=0$ by setting
$v_{i,0}^{\star}(x):=v_i^{\star}(x)$, and
\begin{equation}
\label{eq:zero-radius-adjoint-recovery}
\lim_{t\downarrow0} v_{i-1,t}^{\star}(x)=v_{i-1}^{\star}(x),
\qquad i=1,\ldots,M.
\end{equation}
Consequently,
\begin{equation}
\label{eq:zero-radius-hypergradient}
\lim_{t\downarrow0}
\bigl[
\nabla_1 f(x,z_M^{\star}(x))+v_{0,t}^{\star}(x)
\bigr]
=
\nabla\Phi(x).
\end{equation}
\end{theorem}

\begin{proof}
For $t\in[0,\delta_{\rm adm}]$, the terminal signed perturbed operator is
$(\mu_M-tL_f)$-strongly monotone, hence uniformly strongly monotone with
modulus at least $\mu_M-\delta_{\rm adm}L_f>0$. The contraction-based argument of
Proposition~\ref{prop:response-regularity} gives a unique terminal response.
At every preceding level the perturbation is an additive constant in the
response variable, so $\mu_i$-strong monotonicity is unchanged.  Backward
induction gives well-posedness on $[0,\delta_{\rm adm}]$.

The terminal state Jacobian is nonsingular throughout the interval
by the same monotonicity bound. The implicit-function theorem gives a
$C^1$ terminal path. At $t=0$, differentiation of the
stationarity equation gives
$H_M^{\star}\dot z_{M,0}^{\star}+v_M^{\star}=0$, hence
$\dot z_{M,0}^{\star}=-[H_M^{\star}]^{-1}v_M^{\star}$. Therefore,
$\lim_{t\downarrow0}v_{M-1,t}^{\star}
=B_M^{\star}\dot z_{M,0}^{\star}=v_{M-1}^{\star}$.

For the induction step, suppose the downstream path
$t\mapsto z_{i+1,t}^{\star}$ is $C^1$ on
$[0,\delta_{\rm adm}]$ and satisfies
$\dot z_{i+1,0}^{\star}
=-[H_{i+1}^{\star}]^{-1}v_{i+1}^{\star}$.
Define
\[
q_i(t)
:=
\nabla_2\Psi_{i+1}(z_{i+1,t}^{\star},z_i^{\star})
-
\nabla_2\Psi_{i+1}(z_{i+1}^{\star},z_i^{\star}).
\]
Then $q_i$ is $C^1$, $q_i(0)=0$, and
$q_i(t)=t v_{i,t}^{\star}$ for $t>0$.
The chain rule and \eqref{eq:exact-adjoint} give
$q_i'(0)=B_{i+1}^{\star}\dot z_{i+1,0}^{\star}
=v_i^{\star}$.
Applying the implicit-function theorem to
$\nabla_1\Psi_i(z_{i,t}^{\star},z_{i-1}^{\star})+q_i(t)=0$
gives a $C^1$ path. Differentiating at $t=0$ yields
$H_i^{\star}\dot z_{i,0}^{\star}+v_i^{\star}=0$,
hence $\dot z_{i,0}^{\star}
=-[H_i^{\star}]^{-1}v_i^{\star}$.
The defining gradient difference consequently gives
$\lim_{t\downarrow0}v_{i-1,t}^{\star}
=B_i^{\star}\dot z_{i,0}^{\star}=v_{i-1}^{\star}$.
Backward induction proves the regularity and limits.
The hypergradient limit
\eqref{eq:zero-radius-hypergradient} follows from
\eqref{eq:exact-hypergradient}.
\end{proof}

\subsection{Perturbation bounds and adjoint bias}
\label{subsec:finite-radius-bias}
We now quantify the error at a finite perturbation radius.
We first bound the response displacements, which control the
scaled differences used in the backward recursion.
\begin{lemma}[Perturbation displacement bounds]
\label{lem:perturbation-displacement}
Under Assumptions~\ref{ass:implicit-layers} and \ref{ass:outer-function}, and under \eqref{eq:admissible-radius},
there exist constants $C_{\mathrm{path},i}>0$, independent of $x$ and
$t\in[0,\delta_{\rm adm}]$, such that
$\|z_{i,t}^{\star}(x)-z_i^{\star}(x)\| \le C_{\mathrm{path},i}t, \, i=1,\ldots,M.$
The constants may be chosen as
\[
C_{\mathrm{path},M}:=\frac{C_f}{\bar\mu_M},
\qquad
C_{\mathrm{path},i}
:=\frac{L_{i+1}}{\mu_i}C_{\mathrm{path},i+1},
\quad i=M-1,\ldots,1.
\]
The sharper terminal displacement bound
$\|z_{M,t}^{\star}-z_M^{\star}\|\le(C_f/\mu_M)t$ also holds.
\end{lemma}

\begin{proof}
See Appendix~\ref{app:supporting-proofs}.
\end{proof}

For $t>0$, define the scaled response difference
$d_{i,t}^{\star}(x)
:=[z_{i,t}^{\star}(x)-z_i^{\star}(x)]/t$,
$i=1,\ldots,M$.
Lemma~\ref{lem:perturbation-displacement} gives
$\|d_{i,t}^{\star}(x)\|\le C_{\mathrm{path},i}$
throughout the admissible interval.
By Theorem~\ref{thm:zero-radius-recovery}, these differences
extend continuously to $t=0$ with
$d_{i,0}^{\star}(x):=\dot z_{i,0}^{\star}(x)$.

To express the finite-radius recursion exactly, define the
Hessian blocks averaged along each response perturbation segment:
\begin{align}
\label{eq:averaged-H}
\overline H_{i,t}(x)
&:=
\int_0^1
\nabla_{11}^2\Psi_i
\bigl(
z_i^{\star}(x)+s(z_{i,t}^{\star}(x)-z_i^{\star}(x)),
z_{i-1}^{\star}(x)
\bigr)\,ds,
\\
\label{eq:averaged-B}
\overline B_{i,t}(x)
&:=
\int_0^1
\nabla_{21}^2\Psi_i
\bigl(
z_i^{\star}(x)+s(z_{i,t}^{\star}(x)-z_i^{\star}(x)),
z_{i-1}^{\star}(x)
\bigr)\,ds.
\end{align}

\begin{lemma}[Identities for the averaged Hessian blocks]
\label{lem:averaged-linearization}
For every $t\in(0,\delta_{\rm adm}]$ and $i=1,\ldots,M$,
\begin{equation}
\label{eq:averaged-recursion}
\overline H_{i,t}d_{i,t}^{\star}+v_{i,t}^{\star}=0,
\qquad
v_{i-1,t}^{\star}
=
\overline B_{i,t}d_{i,t}^{\star}.
\end{equation}
Moreover,
$\|\overline H_{i,t}^{-1}\|\le\mu_i^{-1}$ and
$\|\overline B_{i,t}\|\le L_i$.
\end{lemma}

\begin{proof}
See Appendix~\ref{app:supporting-proofs}.
\end{proof}

These identities have the same form as the exact adjoint
recursion, with the unperturbed Hessian blocks replaced by
their segment averages.
Hessian Lipschitz continuity bounds the block differences
by $O(t)$. The next theorem quantifies how this error
propagates backward through the chain.

\begin{theorem}[Adjoint approximation bias]
\label{thm:finite-radius-bias}
Under Assumptions~\ref{ass:implicit-layers} and \ref{ass:outer-function}, and under \eqref{eq:admissible-radius}, for each
$i=0,\ldots,M$ there exists $C_{\mathrm{bias},i}<\infty$, independent of
$x$ and $t\in(0,\delta_{\rm adm}]$, such that
$\|v_{i,t}^{\star}(x)-v_i^{\star}(x)\| \le C_{\mathrm{bias},i}t.$
The constants may be chosen recursively as
$C_{\mathrm{bias},M} := L_f C_{\mathrm{path},M},$
and, for $i=M,\ldots,1$,
\[
C_{\mathrm{bias},i-1}
:=
\kappa_i C_{\mathrm{bias},i}
+
\frac{\rho_i C_{\mathrm{path},i}^2}{2}(1+\kappa_i),
\qquad
\kappa_i=\frac{L_i}{\mu_i}.
\]
Moreover, the scaled response differences satisfy
\[
\|d_{i,t}^{\star}(x)-d_{i,0}^{\star}(x)\|
\le
\left(
\frac{C_{\mathrm{bias},i}}{\mu_i}
+
\frac{\rho_i C_{\mathrm{path},i}^2}{2\mu_i}
\right)t,
\qquad
i=1,\ldots,M.
\]
\end{theorem}

\begin{proof}
At the terminal level, Lipschitz continuity of $\nabla_2f$ and
Lemma~\ref{lem:perturbation-displacement} give
$\|v_{M,t}^{\star}-v_M^{\star}\|\le L_fC_{\mathrm{path},M}t$, so
$C_{\mathrm{bias},M}=L_fC_{\mathrm{path},M}$. For the induction step, assume
$\|v_{i,t}^{\star}-v_i^{\star}\|
\le C_{\mathrm{bias},i}t$.
Comparing
$\overline H_{i,t}d_{i,t}^{\star}+v_{i,t}^{\star}=0$
with
$H_i^{\star}d_{i,0}^{\star}+v_i^{\star}=0$
gives
\[
\|d_{i,t}^{\star}-d_{i,0}^{\star}\|\le
\left(\frac{C_{\mathrm{bias},i}}{\mu_i}
+\frac{\rho_iC_{\mathrm{path},i}^2}{2\mu_i}\right)t,
\]
because $\|\overline H_{i,t}^{-1}\|\le\mu_i^{-1}$,
$\|d_{i,0}^{\star}\|\le C_{\mathrm{path},i}$, and
$\|\overline H_{i,t}-H_i^{\star}\|\le(\rho_iC_{\mathrm{path},i}/2)t$.
We then propagate this estimate to the preceding adjoint.
Using
$v_{i-1,t}^{\star}=\overline B_{i,t}d_{i,t}^{\star}$,
$v_{i-1}^{\star}=B_i^{\star}d_{i,0}^{\star}$,
$\|\overline B_{i,t}\|\le L_i$, and
$\|\overline B_{i,t}-B_i^{\star}\|
\le(\rho_iC_{\mathrm{path},i}/2)t$, we obtain
\[
\begin{aligned}
\|v_{i-1,t}^{\star}-v_{i-1}^{\star}\|
&\le
L_i\|d_{i,t}^{\star}-d_{i,0}^{\star}\|
+\frac{\rho_iC_{\mathrm{path},i}^2}{2}t\le
\left[
\kappa_i C_{\mathrm{bias},i}
+\frac{\rho_iC_{\mathrm{path},i}^2}{2}(1+\kappa_i)
\right]t.
\end{aligned}
\]
This gives the stated recursion for
$C_{\mathrm{bias},i-1}$.
Backward induction proves both claims.
\end{proof}
The direct term $\nabla_1f(x,z_M^\star(x))$ is unchanged
in the perturbation estimator, so its gradient error is
exactly the level-$0$ adjoint error.

\begin{corollary}[Bias of the perturbation-based gradient estimator]
\label{cor:finite-radius-gradient-bias}
Under the conditions of
Theorem~\ref{thm:finite-radius-bias},
$\|G_\delta^{\star}(x)-\nabla\Phi(x)\| \le C_{\mathrm{bias},0}\delta.$
Consequently,
$\|G_\delta^{\star}(x)-\nabla\Phi(x)\|^2
\le C_{\mathrm{bias},0}^2\delta^2$.
\end{corollary}

\begin{proof}
By \eqref{eq:exact-hypergradient} and
\eqref{eq:finite-radius-gradient-estimator},
$G_\delta^{\star}(x)-\nabla\Phi(x)
=v_{0,\delta}^{\star}(x)-v_0^{\star}(x)$.
The result now follows from
Theorem~\ref{thm:finite-radius-bias}.
\end{proof}

\subsection{Uniform Lipschitz bounds}
\label{subsec:finite-radius-regularity}

We next establish Lipschitz bounds in $x$ for the perturbed
states, scaled response differences, and adjoint approximations.
The constants are uniform for $t\in(0,\delta_{\rm adm}]$.

\begin{lemma}[Uniform Lipschitz continuity of the perturbed states]
\label{lem:perturbed-state-x-lipschitz}
Under Assumptions~\ref{ass:implicit-layers}
and~\ref{ass:outer-function} and
\eqref{eq:admissible-radius}, for each $i=1,\ldots,M$ there
exists $\widehat L_{z,i}<\infty$, independent of
$t\in[0,\delta_{\rm adm}]$, such that
$\|z_{i,t}^{\star}(x)-z_{i,t}^{\star}(x')\|
\le \widehat L_{z,i}\|x-x'\|$.
\end{lemma}

\begin{proof}
See Appendix~\ref{app:supporting-proofs}.
\end{proof}

For the scaled differences, we use
$d_{i,t}^\star(x)
=t^{-1}\int_0^t\dot z_{i,s}^\star(x)\,ds$.
Lemma~\ref{lem:path-derivative-x-lipschitz} in
Appendix~\ref{app:perturbation-technical-details} gives
$\|\dot z_{i,t}^\star(x)\|\le C_{\mathrm{path},i}$ and
$\|\dot z_{i,t}^\star(x)-\dot z_{i,t}^\star(x')\|
\le L_{d,i}\|x-x'\|$,
with constants independent of $t$.
Averaging yields the following bounds.

\begin{corollary}[Uniform regularity of the scaled response differences]
\label{cor:finite-direction-regularity}
Under Assumptions~\ref{ass:implicit-layers}
and~\ref{ass:outer-function} and
\eqref{eq:admissible-radius}, for every $i=1,\ldots,M$
and $t\in(0,\delta_{\rm adm}]$,
\[
\|d_{i,t}^{\star}(x)\|\le C_{\mathrm{path},i},
\qquad
\|d_{i,t}^{\star}(x)-d_{i,t}^{\star}(x')\|
\le L_{d,i}\|x-x'\|.
\]
The constants are independent of $t$.
\end{corollary}

\begin{proof}
See Appendix~\ref{app:supporting-proofs}.
\end{proof}

Similarly, write
$B_{i,s}(x)
:=\nabla_{21}^2\Psi_i
(z_{i,s}^\star(x),z_{i-1}^\star(x))$.
For $i=0,\ldots,M-1$ and $t>0$, the fundamental theorem
of calculus gives
\begin{equation}
\label{eq:adjoint-path-integral}
v_{i,t}^{\star}(x)
=
\frac{1}{t}
\int_0^t
B_{i+1,s}(x)\dot z_{i+1,s}^{\star}(x)\,ds.
\end{equation}
This representation gives uniform Lipschitz bounds
for the adjoint approximations.

\begin{corollary}[Uniform regularity of the adjoint approximations]
\label{cor:finite-adjoint-regularity}
Under Assumptions~\ref{ass:implicit-layers}
and~\ref{ass:outer-function} and
\eqref{eq:admissible-radius}, for each $i=0,\ldots,M$ there
exists $L_{v,i}<\infty$, independent of
$t\in(0,\delta_{\rm adm}]$, such that
$\|v_{i,t}^{\star}(x)-v_{i,t}^{\star}(x')\|
\le L_{v,i}\|x-x'\|$.
One may choose
$L_{v,M}:=L_f(1+\widehat L_{z,M})$
and, for $i=0,\ldots,M-1$,
\[
L_{v,i}
:=
\rho_{i+1}
\bigl(\widehat L_{z,i+1}+L_{z,i}\bigr)
C_{\mathrm{path},i+1}
+
L_{i+1}L_{d,i+1}.
\]
\end{corollary}

\begin{proof}
See Appendix~\ref{app:supporting-proofs}.
\end{proof}

\paragraph{Exact targets at fixed $\delta$.}
For the algorithmic analysis, fix
$\delta\in(0,\delta_{\rm adm}]$ and write
$d_i^\star(x):=d_{i,\delta}^\star(x)
=[z_{i,\delta}^\star(x)-z_i^\star(x)]/\delta$,
$i=1,\ldots,M$.
We retain $v_{i,\delta}^\star(x)$ for the adjoint approximations.
The superscript $\star$ denotes exact response solves,
not zero perturbation bias.
All preceding bounds are uniform in $\delta$.

\paragraph{From exact targets to inexact solves.}
The preceding bounds concern exact response solves.
With inexact solves, a downstream response error affects
the adjoint approximation that defines the next upstream
perturbed problem. The following example isolates this
effect in a chain with zero exact-solve perturbation bias.

\begin{proposition}[Propagation of inner-solve errors in a two-level chain]
\label{prop:two-level-finite-solve-obstruction}
Consider the scalar two-level chain
\[
\Psi_1(z_1,x)=\frac12(z_1-a_1x)^2,
\qquad
\Psi_2(z_2,z_1)=\frac12(z_2-a_2z_1)^2,
\]
with outer objective
$f(x,z_2)=\frac{\tau}{2}x^2+cz_2$,
where $\tau>0$ and $a_1a_2c\neq0$.
For every $\delta>0$, the exactly solved recursive
perturbation estimator satisfies
$G_\delta^\star(x)=\nabla\Phi(x)$.

Suppose the unperturbed responses are exact, but the
terminal perturbed response is returned as
$\widehat z_{2,\delta}
=a_2z_1^\star-\delta c+r_2$,
where $r_2$ is its signed state error.
If the level-$1$ perturbed problem defined by the resulting
adjoint approximation is solved exactly, then
$\widehat v_{1,\delta}=a_2c-(a_2/\delta)r_2$, and
\begin{equation}
\label{eq:two-level-finite-solve-obstruction}
\widehat G_\delta-\nabla\Phi(x)
=
-\frac{a_1a_2}{\delta}r_2.
\end{equation}
\end{proposition}

\begin{proof}
The exact responses are
$z_1^\star=a_1x$ and $z_2^\star=a_2z_1^\star$.
Since $\nabla_2\Psi_2=-a_2(z_2-a_2z_1)$,
the terminal gradient difference gives
$\widehat v_{1,\delta}=a_2c-a_2r_2/\delta$.
The upstream solve yields
$\widehat z_{1,\delta}
=a_1x-\delta\widehat v_{1,\delta}$,
and its input-gradient difference divided by $\delta$
is $a_1\widehat v_{1,\delta}$.
Thus
$\widehat G_\delta=\tau x+a_1\widehat v_{1,\delta}$,
whereas $\nabla\Phi(x)=\tau x+a_1a_2c$.
Subtracting gives
\eqref{eq:two-level-finite-solve-obstruction};
setting $r_2=0$ proves the exact-solve identity.
\end{proof}

The $1/\delta$ amplification also appears in finite-difference
hypergradient estimation for equality-constrained bilevel problems
\cite{kornowski2024constrained}.
Here, the amplified error additionally changes the upstream
perturbation vector.
For fixed nonzero $a_1a_2$, an $O(\delta)$ gradient error
in this example requires $r_2=O(\delta^2)$.
This example specifies the terminal state error independently
of the solver.
In Algorithm~\ref{alg:remid}, paired updates couple the errors
of the two branches.
Section~\ref{sec:algorithm} uses this coupling to track their
scaled differences and the recursively generated adjoint errors.

\begin{remark}[Bilevel reduction]
\label{rem:bilevel-reduction}
For $M=1$, the construction uses one perturbed response
and one input-gradient difference.
Solve errors may still be amplified by $1/\delta$,
but there is no upstream response problem whose
perturbation vector they alter.
\end{remark}

\section{First-Order Algorithm and Tracking Analysis}
\label{sec:algorithm}

We now approximate the perturbation targets using finitely many
warm-started first-order updates.
Throughout this section, Assumptions~\ref{ass:implicit-layers}
and~\ref{ass:outer-function} hold, and
$0<\delta\le\delta_{\rm adm}$ satisfies
\eqref{eq:admissible-radius}.
A further radius restriction for tracking is introduced below.

Let $\widehat G^k$ denote the estimator produced by
Algorithm~\ref{alg:remid} at $x^k$. Its error separates into
\[
\widehat G^k-\nabla\Phi(x^k)
=
\underbrace{\widehat G^k-G_\delta^{\star}(x^k)}_{\text{inexact-solve tracking}}
+
\underbrace{G_\delta^{\star}(x^k)-\nabla\Phi(x^k)}_{\text{perturbation bias}}.
\]
Section~\ref{sec:perturbation} bounds the bias.
Here we control the tracking error by jointly analyzing
the forward response updates and the backward
scaled response differences.

\subsection{Warm-started first-order realization}
\label{subsec:first-order-realization}

\paragraph{Forward and backward updates.}
At outer iteration $k$, \ReMID{} first refreshes the
unperturbed responses from level $1$ to $M$, using the
newly updated input at each level.
It then updates the perturbed responses from level $M$
to $1$, generating each upstream perturbation vector
from the completed downstream update. 

Set $\bar z_0^k:=x^k$ and denote the refreshed
unperturbed states by $\bar z_i^k$, so that
$z_i^{k+1}=\bar z_i^k$.
At each level, both branches use the same input
$\bar z_{i-1}^k$, stepsize $\eta_i$, and budget $N_i$.
These shared choices yield a recursion directly for the scaled
response difference, avoiding separate absolute solve-error
bounds divided by $\delta$. Both branches retain their endpoints as warm starts
for the next outer iteration.
Warm starts are standard in bilevel optimization
\cite{arbel2022amortized}; here the tracking analysis
must also account for recursively generated
perturbation vectors.

\begin{figure}[!t]
\centering
\begin{tikzpicture}[
    state/.style={
        draw,
        rounded corners=1.5pt,
        minimum height=7.2mm,
        minimum width=16.5mm,
        inner sep=2pt,
        font=\small
    },
    wide/.style={
        draw,
        rounded corners=1.5pt,
        minimum height=7.5mm,
        inner sep=2.5pt,
        font=\small
    },
    flow/.style={
        -{Latex[length=1.7mm,width=1.1mm]},
        line width=0.55pt
    },
    guide/.style={
        densely dashed,
        -{Latex[length=1.4mm,width=0.9mm]},
        line width=0.4pt
    },
    heading/.style={
        font=\small\bfseries
    },
    note/.style={
        font=\scriptsize,
        align=center
    }
]

\node[heading] at (5.30,2.45) {Forward unperturbed sweep};
\node[state] (x0) at (0,1.65) {$x^k=\bar z_0^k$};
\node[state] (z1) at (2.65,1.65) {$\bar z_1^k$};
\node (fdots) at (5.30,1.65) {$\cdots$};
\node[state] (zMm1) at (7.95,1.65) {$\bar z_{M-1}^k$};
\node[state] (zM) at (10.60,1.65) {$\bar z_M^k$};
\draw[flow] (x0) -- (z1);
\draw[flow] (z1) -- (fdots);
\draw[flow] (fdots) -- (zMm1);
\draw[flow] (zMm1) -- (zM);
\node[note] at (3.98,0.90)
    {same refreshed input at level $i$: $\bar z_{i-1}^k$};
\node[note] (terminal) at (10.60,0.82)
    {terminal perturbation\\[-1pt] $\delta\nabla_2f(x^k,\cdot)$};
\node[state] (v0) at (0,-0.45) {$v_{0,\delta}^k$};
\node[state] (z1d) at (2.65,-0.45) {$z_{1,\delta}^{k+1}$};
\node (bdots) at (5.30,-0.45) {$\cdots$};
\node[state] (vMm1) at (7.95,-0.45) {$v_{M-1,\delta}^k$};
\node[state] (zMd) at (10.60,-0.45) {$z_{M,\delta}^{k+1}$};
\draw[flow] (z1d) -- (v0);
\draw[flow] (bdots) -- (z1d);
\draw[flow] (vMm1) -- (bdots);
\draw[flow] (zMd) -- (vMm1);
\draw[guide] (terminal.south) -- (zMd.north);
\node[heading] at (5.30,-1.20) {Backward perturbed/adjoint sweep};
\node[wide] (Ghat) at (3.40,-2.25)
    {$\widehat G^k=\nabla_1f(x^k,\bar z_M^k)+v_{0,\delta}^k$};
\node[wide] (xnext) at (8.20,-2.25)
    {$x^{k+1}=x^k-\alpha\widehat G^k$};
\draw[flow] (Ghat) -- (xnext);
\end{tikzpicture}

\caption{
Forward--backward realization of \ReMID.
The unperturbed sweep refreshes the response inputs in
forward order. The perturbed sweep uses these same inputs
in backward order, generating $v_{i-1,\delta}^k$ from
paired input-gradient differences at each level.
The final adjoint approximation $v_{0,\delta}^k$
forms the gradient estimator and outer update.
}
\label{fig:forward-backward-pipeline}
\end{figure}
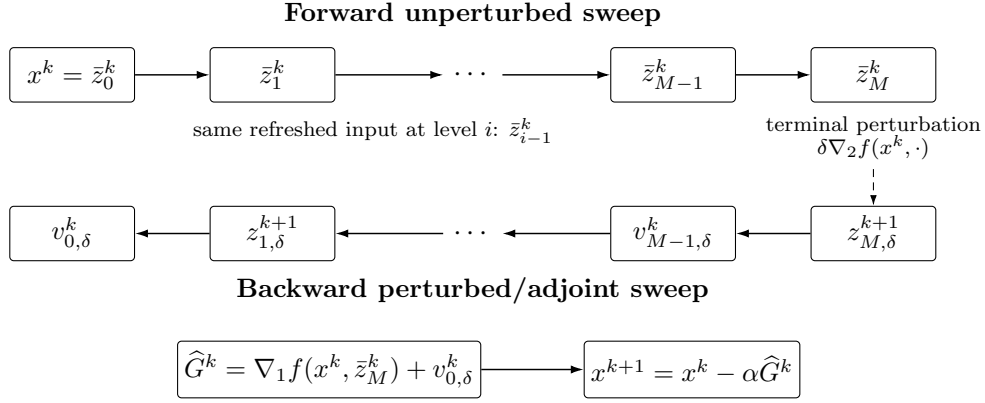

Algorithm~\ref{alg:remid} gives the complete update sequence.
The method stores only the current inner iterates and
branch endpoints; it does not differentiate through
solver trajectories or solve adjoint linear systems.

\begin{algorithm}[t]
\caption{\ReMID: Recursive multilevel implicit differentiation via first-order response perturbations}
\label{alg:remid}
\begin{algorithmic}[1]
\Require Number of outer iterations $K\ge1$, outer initial point $x^0$,
unperturbed warm starts $\{z_i^0\}_{i=1}^M$,
outer stepsize $\alpha$, perturbation parameter $\delta$, inner stepsizes
$\{\eta_i\}_{i=1}^M$, and positive integer inner budgets $\{N_i\}_{i=1}^M$.
\State Initialize the perturbed branch by $z_{i,\delta}^0\gets z_i^0$ for all $i=1,\ldots,M$.
\For{$k=0,1,\ldots,K-1$}
    \State $\bar z_0^k\gets x^k$.
    \Statex \hspace{0.5em}\textbf{Forward unperturbed sweep:}
    \For{$i=1,2,\ldots,M$}
        \State $z_i^{k,0}\gets z_i^k$.
        \For{$\ell=0,1,\ldots,N_i-1$}
            \State $z_i^{k,\ell+1}\gets z_i^{k,\ell}
            -\eta_iP_i\nabla_1\Psi_i\!\left(z_i^{k,\ell},\bar z_{i-1}^k\right)$.
        \EndFor
        \State $\bar z_i^k\gets z_i^{k,N_i}$ and $z_i^{k+1}\gets\bar z_i^k$.
    \EndFor
    \Statex \hspace{0.5em}\textbf{Backward perturbed/adjoint sweep:}
    \For{$i=M,M-1,\ldots,1$}
        \State $z_{i,\delta}^{k,0}\gets z_{i,\delta}^k$.
        \If{$i=M$}
            \For{$\ell=0,1,\ldots,N_M-1$}
                \State $z_{M,\delta}^{k,\ell+1}\gets z_{M,\delta}^{k,\ell}
                -\eta_MP_M\!\left[
                \nabla_1\Psi_M\!\left(z_{M,\delta}^{k,\ell},\bar z_{M-1}^k\right)
                +\delta\nabla_2 f\!\left(x^k,z_{M,\delta}^{k,\ell}\right)
                \right]$.
            \EndFor
        \Else
            \For{$\ell=0,1,\ldots,N_i-1$}
                \State $z_{i,\delta}^{k,\ell+1}\gets z_{i,\delta}^{k,\ell}
                -\eta_iP_i\!\left[
                \nabla_1\Psi_i\!\left(z_{i,\delta}^{k,\ell},\bar z_{i-1}^k\right)
                +\delta v_{i,\delta}^k
                \right]$.
            \EndFor
        \EndIf
        \State $z_{i,\delta}^{k+1}\gets z_{i,\delta}^{k,N_i}$.
        \State $v_{i-1,\delta}^k\gets
        \dfrac{\nabla_2\Psi_i\!\left(z_{i,\delta}^{k+1},\bar z_{i-1}^k\right)
        -\nabla_2\Psi_i\!\left(\bar z_i^k,\bar z_{i-1}^k\right)}{\delta}$.
    \EndFor
    \State $\widehat G^k\gets\nabla_1 f\!\left(x^k,\bar z_M^k\right)+v_{0,\delta}^k$.
    \State $x^{k+1}\gets x^k-\alpha\widehat G^k$.
\EndFor
\end{algorithmic}
\end{algorithm}

\paragraph{Inner contraction.}
For level $i$, define the signed-gradient update map
$\mathcal T_i^{\eta_i}(z,u)
:=z-\eta_iP_i\nabla_1\Psi_i(z,u)$.
At nonterminal levels, the perturbation vector is fixed
during the local solve and leaves the contraction
constants unchanged.
At the terminal level, the perturbation depends on the
response state through $\nabla_2f$.
Accordingly, set
\[
\widetilde\mu_i
:=
\mu_i-\delta L_f\mathbf 1_{\{i=M\}},
\qquad
\widetilde L_i
:=
L_i+\delta L_f\mathbf 1_{\{i=M\}}.
\]
The radius condition \eqref{eq:admissible-radius}
ensures $\widetilde\mu_M\ge\bar\mu_M>0$.
The following lemma gives contraction factors for
both branches.

\begin{lemma}[Inner contraction factors]
\label{lem:inner-contraction}
Under Assumptions~\ref{ass:implicit-layers} and \ref{ass:outer-function}, let $\delta\in(0,\delta_{\rm adm}]$ satisfy
\eqref{eq:admissible-radius}. For each level $i$, let $\vartheta_i$ and
$\widetilde\vartheta_i$ denote squared-distance contraction factors for the
unperturbed and perturbed inner iterations, respectively. The following choices
are admissible.

\begin{enumerate}[label=\textnormal{(\alph*)}]
\item
For a general potential-based layer, if
$0<\eta_i< \frac{2\widetilde\mu_i}{\widetilde L_i^2},$
then the same stepsize is admissible for the unperturbed map as well, and we may
take
\[
\vartheta_i
=
1-2\eta_i\mu_i+\eta_i^2L_i^2,
\qquad
\widetilde\vartheta_i
=
1-2\eta_i\widetilde\mu_i
+\eta_i^2\widetilde L_i^2.
\]
In particular,
$\widetilde\vartheta_i=\vartheta_i$ for $i<M$.

\item
Suppose that the $i$th layer is a strongly convex minimization
layer, so that $P_i=I$ and
$\Psi_i(z,u)=g_i(z,u)$ with
$\mu_i I \preceq \nabla_{11}^2g_i(z,u) \preceq L_i I.$
For $i<M$, any $0<\eta_i\le 1/L_i$ is admissible, with
\[
\vartheta_i=\widetilde\vartheta_i=(1-\eta_i\mu_i)^2.
\]
In particular, $\eta_i=1/L_i$ gives
$\vartheta_i=\widetilde\vartheta_i=(1-\mu_i/L_i)^2$.
At the terminal level, any common unperturbed--perturbed stepsize satisfying
$0<\eta_M\le1/\widetilde L_M$ is admissible, with
\[
\vartheta_M=(1-\eta_M\mu_M)^2,
\qquad
\widetilde\vartheta_M=(1-\eta_M\widetilde\mu_M)^2.
\]
\end{enumerate}

The unperturbed contraction factor also applies to
averages of the state Hessian of $\Psi_i$ along response
segments:
\begin{equation}
\label{eq:averaged-inner-contraction}
\left\|
\bigl(
I-\eta_iP_i\overline H
\bigr)h
\right\|^2
\le
\vartheta_i\|h\|^2,
\end{equation}
for every such averaged Hessian $\overline H$.
\end{lemma}

\begin{proof}
For a $\mu$-strongly monotone, $L$-Lipschitz operator $T$,
$\|(z-\eta T(z))-(z'-\eta T(z'))\|^2
\le(1-2\eta\mu+\eta^2L^2)\|z-z'\|^2$.
At nonterminal levels the constant perturbation cancels in differences.
At level $M$, the operator $\delta P_M\nabla_2f(x,\cdot)$ changes the
monotonicity and Lipschitz bounds by at most $\delta L_f$.
This proves part~(a).

For part~(b), the averaged Hessians have spectra in $[\mu_i,L_i]$,
or $[\widetilde\mu_M,\widetilde L_M]$ for the terminal perturbed problem.
The stated stepsizes give $\|I-\eta H\|\le1-\eta\mu$ with the corresponding
lower spectral bound. Averaging also preserves the monotonicity and norm
bounds in part~(a), proving \eqref{eq:averaged-inner-contraction}.
\end{proof}

We use $\vartheta_i$ and $\widetilde\vartheta_i$
for any valid factors from
Lemma~\ref{lem:inner-contraction}.

\subsection{Response and adjoint error bounds}
\label{subsec:primal-forcing-tracking}

We track the unperturbed states and the scaled differences
$d_i^k:=(z_{i,\delta}^k-z_i^k)/\delta$.
Set $x^{-1}:=x^0$ and $s^k:=x^k-x^{k-1}$,
so that $s^0=0$.
For $k\ge1$, the stored states were computed at iteration
$k-1$, and we measure their errors relative to the targets
at $x^{k-1}$:
\begin{equation}
\label{eq:tracking-errors}
\mathcal E_i^k
:=
\|z_i^k-z_i^\star(x^{k-1})\|^2,
\qquad
\mathcal Z_i^k
:=
\|d_i^k-d_i^\star(x^{k-1})\|^2.
\end{equation}
Thus $\mathcal E_i^{k+1}$ and $\mathcal Z_i^{k+1}$
measure tracking at the current outer point $x^k$.
Set $\mathcal E_0^k:=0$ and
$\chi_i:=\vartheta_i^{N_i}$.

The forward response error has three sources:
the previous warm-start error, the updated input error,
and the movement of the outer variable.
\begin{lemma}[Forward primal tracking recursion]
\label{lem:primal-tracking}
Under Assumption~\ref{ass:implicit-layers}, suppose that the unperturbed inner
iterations use squared-distance contraction factors $\vartheta_i<1$ as in
Lemma~\ref{lem:inner-contraction}. Then the unperturbed iterates generated by
Algorithm~\ref{alg:remid} satisfy
\begin{equation}
\label{eq:primal-tracking-recursion}
\mathcal E_i^{k+1}
\le
\tau_i\mathcal E_i^k
+\lambda_i\mathcal E_{i-1}^{k+1}
+\gamma_i\|s^k\|^2.
\end{equation}
For $i=1$, one may take
$\tau_1:=2\chi_1$, $\lambda_1:=0$, and
$\gamma_1:=2\chi_1L_{z,1}^2$.
For $i=2,\ldots,M$, one may take
\begin{equation}
\label{eq:forward-primal-coefficients}
\tau_i:=4\chi_i,
\qquad
\lambda_i:=2\kappa_i^2(1+4\chi_i),
\qquad
\gamma_i:=8\chi_iL_{z,i}^2.
\end{equation}
\end{lemma}

The cross-level term involves the new error in the preceding-level response
$\mathcal E_{i-1}^{k+1}$. The term $\mathcal E_{i-1}^{k+1}$ reflects the use of
the refreshed input: current response errors propagate
forward through the chain.

\begin{proof}
See Appendix~\ref{app:supporting-proofs}.
\end{proof}

We next express the adjoint error in terms of the response
and scaled-difference errors.
Define \[
C_{vZ,i}:=2L_i^2+6\rho_i^2C_{\mathrm{path},i}^2\delta^2,
\qquad
C_{vE,i}:=6\rho_i^2C_{\mathrm{path},i}^2,
\qquad i=1,\ldots,M.
\]

\begin{lemma}[Propagation of adjoint approximation errors]
\label{lem:finite-adjoint-propagation}
Under Assumptions~\ref{ass:implicit-layers} and \ref{ass:outer-function}, and under \eqref{eq:admissible-radius}, for every
$i=1,\ldots,M$,
\begin{equation}
\label{eq:finite-adjoint-propagation}
\begin{aligned}
\|v_{i-1,\delta}^k-v_{i-1,\delta}^{\star}(x^k)\|^2
\le{}&
C_{vZ,i}\mathcal Z_i^{k+1}
+C_{vE,i}\mathcal E_i^{k+1}
+
C_{vE,i}\mathbf 1_{\{i\ge2\}}\mathcal E_{i-1}^{k+1}.
\end{aligned}
\end{equation}
Consequently, for $i<M$,
\begin{equation}
\label{eq:finite-adjoint-propagation-shifted}
\begin{aligned}
\|v_{i,\delta}^k-v_{i,\delta}^{\star}(x^k)\|^2
\le{}&
C_{vZ,i+1}\mathcal Z_{i+1}^{k+1}
+
C_{vE,i+1}
\bigl(
\mathcal E_{i+1}^{k+1}
+\mathcal E_i^{k+1}
\bigr).
\end{aligned}
\end{equation}
\end{lemma}

\begin{proof}
Fix $i$ and $k$, and write $d^\star=d_i^\star(x^k)$.
The fundamental theorem of calculus gives
$v_{i-1,\delta}^k=\overline B_i^k d_i^{k+1}$ and
$v_{i-1,\delta}^\star(x^k)=\overline B_i^\star d^\star$,
where
\[
\begin{aligned}
\overline B_i^k
&:=\int_0^1\nabla_{21}^2\Psi_i
 \bigl(z_i^{k+1}+r\delta d_i^{k+1},\bar z_{i-1}^k\bigr)\,dr,\\
\overline B_i^\star
&:=\int_0^1\nabla_{21}^2\Psi_i
 \bigl(z_i^\star(x^k)+r\delta d^\star,z_{i-1}^\star(x^k)\bigr)\,dr.
\end{aligned}
\]
Consequently,
$v_{i-1,\delta}^k-v_{i-1,\delta}^\star(x^k)
=\overline B_i^k(d_i^{k+1}-d^\star)
+(\overline B_i^k-\overline B_i^\star)d^\star$.
We have $\|\overline B_i^k\|\le L_i$ and
\[
\|\overline B_i^k-\overline B_i^\star\|
\le\rho_i\left(
\sqrt{\mathcal E_i^{k+1}}
+\delta\sqrt{\mathcal Z_i^{k+1}}
+\sqrt{\mathcal E_{i-1}^{k+1}}\right),
\]
where the last term vanishes for $i=1$.
Using $\|d^\star\|\le C_{\mathrm{path},i}$,
$\|a+b\|^2\le2\|a\|^2+2\|b\|^2$, and
$(a+b+c)^2\le3(a^2+b^2+c^2)$ proves
\eqref{eq:finite-adjoint-propagation}. Shifting the index gives
\eqref{eq:finite-adjoint-propagation-shifted}.
\end{proof}

Writing each gradient difference as an averaged Hessian
block acting on a scaled response difference removes
the explicit division by $\delta$ from the error bound.
The estimate uses the bounded exact target
$d_i^\star(x^k)$ and requires no prior bound on the
computed adjoints.

\paragraph{Local perturbation targets.}
The inner updates use approximate inputs and adjoints.
To separate these target errors from finite-iteration
errors, we introduce the exact solutions of the local
problems solved at iteration $k$. Define $\widehat z_i^{k,*}:=S_i(\bar z_{i-1}^k)$.
For $i<M$, let $\widehat z_{i,\delta}^{k,*}$ solve
$\nabla_1\Psi_i(z,\bar z_{i-1}^k)
+\delta v_{i,\delta}^k=0$.
At the terminal level, let $\widehat z_{M,\delta}^{k,*}$
solve
$\nabla_1\Psi_M(z,\bar z_{M-1}^k)
+\delta\nabla_2f(x^k,z)=0$.
Their scaled difference is
\begin{equation}
\label{eq:local-direction-target}
\widehat d_i^{k,*}
:=
\frac{\widehat z_{i,\delta}^{k,*}-\widehat z_i^{k,*}}{\delta}.
\end{equation}
To compare these local targets with the exact-chain
targets, shrink $\delta_{\rm adm}$, if necessary, so that
\begin{equation}
\label{eq:tracking-radius-condition}
\delta_{\rm adm}\rho_iC_{\mathrm{path},i}
\le\frac{\mu_i}{2}
\quad (i<M),
\qquad
\delta_{\rm adm}
\bigl(\rho_MC_{\mathrm{path},M}+L_f\bigr)
\le\frac{\mu_M}{2}.
\end{equation}
The constants $C_{\mathrm{path},i}$ are retained from
the original admissible interval, so this choice is
not circular.

The next lemma bounds the local target mismatch
in terms of the input and downstream adjoint errors.

\begin{lemma}[Local target error]
\label{lem:local-direction-reference}
Under Assumptions~\ref{ass:implicit-layers} and \ref{ass:outer-function}, and under
\eqref{eq:admissible-radius} and \eqref{eq:tracking-radius-condition}, for $i<M$,
\begin{equation}
\label{eq:forward-local-direction-reference}
\begin{aligned}
\|\widehat d_i^{k,*}-d_i^{\star}(x^k)\|^2
\le{}&
C_{d,u,i}\mathcal E_{i-1}^{k+1}
+
C_{d,v,i}
\|v_{i,\delta}^k-v_{i,\delta}^{\star}(x^k)\|^2,
\end{aligned}
\end{equation}
where one may take
\[
C_{d,u,i}
:=
\frac{8\rho_i^2C_{\mathrm{path},i}^2(1+\kappa_i)^2}{\mu_i^2},
\qquad
C_{d,v,i}:=\frac{8}{\mu_i^2}.
\]
Whenever $i=1$, the input mismatch vanishes because
$\bar z_0^k=x^k=z_0^{\star}(x^k)$, and we set $C_{d,u,1}:=0$.

At the terminal level with $M\ge2$,
\[
\|\widehat d_M^{k,*}-d_M^{\star}(x^k)\|^2
\le
C_{d,u,M}\mathcal E_{M-1}^{k+1},
\]
where
\[
C_{d,u,M}
:=
\frac{4}{\mu_M^2}
\left[
\rho_MC_{\mathrm{path},M}(1+\kappa_M)+L_f\kappa_M
\right]^2.
\]
If $M=1$, the terminal input is exact and $C_{d,u,1}=0$.  In all cases,
set $C_{d,v,M}:=0$.
\end{lemma}

\begin{proof}
See Appendix~\ref{app:supporting-proofs}.
\end{proof}

\subsection{Tracking the scaled response differences}
\label{subsec:auxiliary-tracking}
We now combine the paired-update bounds across levels.
Forward response errors affect the common inputs, while
downstream adjoint errors affect the upstream perturbation vectors.
The resulting coupled recursions track both effects with
coefficients uniformly bounded over the admissible radii.

Set
\[
\overline\vartheta_i:=\max\{\vartheta_i,\widetilde\vartheta_i\},
\qquad
\theta_i:=\frac{1+\overline\vartheta_i}{2},
\qquad
\chi_{d,i}:=\theta_i^{N_i}.
\]
By Lemma~\ref{lem:inner-contraction},
$0\le\overline\vartheta_i<\theta_i<1$. For $i<M$, define the local target error bound
\[
\mathcal R_{i-1}^k
:=
C_{d,u,i}\mathcal E_{i-1}^{k+1}
+
C_{d,v,i}
\|v_{i,\delta}^k-v_{i,\delta}^{\star}(x^k)\|^2.
\]
At level $M$, set
$\mathcal R_{M-1}^k:=C_{d,u,M}\mathcal E_{M-1}^{k+1}$.
Lemma~\ref{lem:local-direction-reference} gives
\begin{equation}
\label{eq:local-direction-residual}
\|\widehat d_i^{k,*}-d_i^{\star}(x^k)\|^2
\le
\mathcal R_{i-1}^k.
\end{equation}
In particular, $\mathcal R_0^k=0$ when $M=1$.

Current response errors propagate forward through
$\mathcal E_{i-1}^{k+1}$, while scaled-difference errors
propagate backward through $\mathcal Z_{i+1}^{k+1}$.
The following theorem captures both dependencies.

\begin{theorem}[Coupled tracking recursion for the scaled response differences]
\label{thm:auxiliary-tracking}
Under Assumptions~\ref{ass:implicit-layers} and \ref{ass:outer-function}, let $\delta\in(0,\delta_{\rm adm}]$ with
\eqref{eq:admissible-radius} and \eqref{eq:tracking-radius-condition}, and suppose
that the inner stepsizes and contraction factors satisfy
Lemma~\ref{lem:inner-contraction}. Then, for every $i=1,\ldots,M$ and
$k\ge0$,
\begin{equation}
\label{eq:auxiliary-tracking-recursion}
\begin{aligned}
\mathcal Z_i^{k+1}
\le{}&
\tau_{d,i}\mathcal Z_i^k
+\beta_{d,i}\mathcal Z_{i+1}^{k+1}
+
\lambda_{d,i}
\bigl(
\mathcal E_{i+1}^{k+1}+\mathcal E_i^{k+1}
\bigr)
+\psi'_{d,i}\mathcal E_i^k
+\psi_{d,i}\mathcal E_{i-1}^{k+1}
+\gamma_{d,i}\|s^k\|^2,
\end{aligned}
\end{equation}
where
$\mathcal Z_{M+1}^{k+1}=\mathcal E_{M+1}^{k+1}:=0$ and
$\mathcal E_0^{k+1}:=0$.  The coefficients are finite and nonnegative,
$\beta_{d,M}=\lambda_{d,M}=0$, and
$\tau_{d,i}=O(\theta_i^{N_i})$ for fixed problem data.  Explicit admissible
coefficients are given in Appendix~\ref{app:auxiliary-tracking-details}; their
uniformity in $\delta$ is established in Lemma~\ref{lem:delta-uniformity}.
\end{theorem}

\begin{proof}[Proof sketch]
Although the two sweeps are executed separately, their
level-$i$ inner iterates can be paired by index $\ell$
because they share the input $\bar z_{i-1}^k$,
stepsize, and budget.
Set
$d_i^{k,\ell}
:=[z_{i,\delta}^{k,\ell}-z_i^{k,\ell}]/\delta$
and
$e_i^{k,\ell}:=d_i^{k,\ell}-\widehat d_i^{k,*}$.
Subtracting the updates and using local stationarity gives
\[
e_i^{k,\ell+1}
=(I-\eta_iP_i\overline H_i^{k,\ell})e_i^{k,\ell}
-\eta_iP_i\Xi_i^{k,\ell}.
\]
Here $\overline H_i^{k,\ell}$ and $\widehat H_i^{k,*}$ are the state
Hessians averaged along the computed and local exact perturbation
segments, respectively, at input $\bar z_{i-1}^k$. The residual is
\[
\Xi_i^{k,\ell}
:=(\overline H_i^{k,\ell}-\widehat H_i^{k,*})\widehat d_i^{k,*}
+
\begin{cases}
0,&i<M,\\
\nabla_2f(x^k,z_{M,\delta}^{k,\ell})
-\nabla_2f(x^k,\widehat z_{M,\delta}^{k,*}),&i=M.
\end{cases}
\]

To bound the Hessian mismatch, write
$\widehat d_i^{k,*}
=d_i^\star(x^k)
+(\widehat d_i^{k,*}-d_i^\star(x^k))$.
For the first term, we use Hessian Lipschitz continuity
and the exact-target bound
$\|d_i^\star(x^k)\|\le C_{\mathrm{path},i}$.
For the second, we use
$\|\overline H_i^{k,\ell}-\widehat H_i^{k,*}\|\le 2L_i$
and the local target error bound $\mathcal R_{i-1}^k$.
This controls the residual without an a priori bound
on the recursively computed perturbation vectors.
Inner contraction and Young's inequality then give
a one-step bound, which is unrolled over $N_i$ updates.
The uniform target regularity bounds control temporal
drift, while the local target and adjoint bounds give
the cross-level terms.
The full derivation and coefficients are given in
Appendix~\ref{app:auxiliary-tracking-details}.
\end{proof}
Increasing $N_i$ contracts the warm-start errors,
but current input and downstream adjoint errors remain.
The Lyapunov argument in
Section~\ref{subsec:outer-descent}
handles these cross-level dependencies jointly.

\begin{lemma}[Uniformity over the admissible perturbation interval]
\label{lem:delta-uniformity}
Fix $\delta_{\rm adm}$ satisfying
\eqref{eq:admissible-radius} and
\eqref{eq:tracking-radius-condition}, and set
$\bar L_M:=L_M+\delta_{\rm adm}L_f$.
Choose inner stepsizes independently of $\delta$.
At the terminal level, take
$0<\eta_M<2\bar\mu_M/\bar L_M^2$ for a general
potential-based layer, or $\eta_M=1/\bar L_M$
for a strongly convex minimization layer.
At nonterminal levels, use fixed stepsizes satisfying
the corresponding contraction conditions.
Then there exist constants
$\vartheta_i^{\rm adm}<1$ and $\theta_i^{\rm adm}<1$,
independent of $\delta\in(0,\delta_{\rm adm}]$, such that
\[
\max\{\vartheta_i,\widetilde\vartheta_i\}
\le\vartheta_i^{\rm adm}<1,
\qquad
\theta_i\le\theta_i^{\rm adm}<1,
\qquad i=1,\ldots,M.
\]
Hence fixed budgets can be chosen so that
$\sup_\delta\tau_i<1$ and
$\sup_\delta\tau_{d,i}<1$ at every level.

For these fixed choices, all target regularity constants
and tracking coefficients are bounded uniformly in
$\delta$. In particular, the tracking recursions contain
no negative powers of $\delta$.
\end{lemma}

\begin{proof}
See Appendix~\ref{app:supporting-proofs}.
\end{proof}

Warm starts are standard in bilevel optimization \cite{arbel2022amortized};
here the additional difficulty is the recursively generated multilevel
perturbation vector.

\subsection{Hypergradient error bound}
\label{subsec:gradient-estimator-tracking}

The response and adjoint bounds now control the tracking
term in the gradient error decomposition.
Adding the finite-radius bias gives the total error bound.
\begin{lemma}[Tracking error of the gradient estimator]
\label{lem:gradient-estimator-tracking}
Under Assumptions~\ref{ass:implicit-layers} and \ref{ass:outer-function}, and under \eqref{eq:admissible-radius}, the
estimator generated by Algorithm~\ref{alg:remid} satisfies
\begin{equation}
\label{eq:gradient-estimator-tracking}
\begin{aligned}
\|\widehat G^k-G_\delta^{\star}(x^k)\|^2
\le{}&
2C_{vZ,1}\mathcal Z_1^{k+1}
+2C_{vE,1}\mathcal E_1^{k+1}
+2L_f^2\mathcal E_M^{k+1}.
\end{aligned}
\end{equation}
Consequently,
\begin{equation}
\label{eq:gradient-total-error}
\begin{aligned}
\|\widehat G^k-\nabla\Phi(x^k)\|^2
\le{}&
4C_{vZ,1}\mathcal Z_1^{k+1}
+4C_{vE,1}\mathcal E_1^{k+1}
+4L_f^2\mathcal E_M^{k+1}
+2C_{\mathrm{bias},0}^2\delta^2.
\end{aligned}
\end{equation}
\end{lemma}

\begin{proof}
By definition,
\[
\begin{aligned}
\widehat G^k-G_\delta^{\star}(x^k)
={}&
\nabla_1f(x^k,z_M^{k+1})
-\nabla_1f(x^k,z_M^{\star}(x^k))
+
v_{0,\delta}^k-v_{0,\delta}^{\star}(x^k).
\end{aligned}
\]
The $L_f$-Lipschitz property and
Lemma~\ref{lem:finite-adjoint-propagation} with $i=1$ give
\eqref{eq:gradient-estimator-tracking}.  Combining this estimate with
$\|G_\delta^{\star}(x^k)-\nabla\Phi(x^k)\|\le C_{\mathrm{bias},0}\delta$ and
$\|a+b\|^2\le2\|a\|^2+2\|b\|^2$ gives
\eqref{eq:gradient-total-error}.
\end{proof}

The next section combines these tracking recursions
with outer descent to establish stationarity guarantees.


\section{Convergence and Complexity}
\label{sec:convergence}

The tracking bounds from the preceding section are coupled across
levels. We combine them with outer descent through a weighted
Lyapunov function to obtain an $O(K^{-1}+\delta^2)$ bound on the
average squared gradient norm of the exact objective.
For fixed problem data and depth, the inner budgets and outer
stepsize can be chosen independently of the target accuracy.
We then derive first-order oracle complexity bounds and quantify
their conditioning dependence for strongly convex minimization chains.

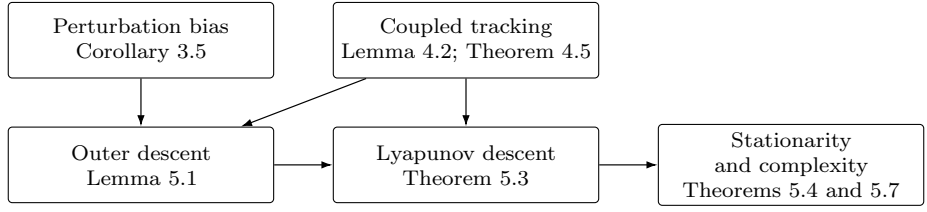
\begin{figure}[t]
\centering
\begin{tikzpicture}[
    proofnode/.style={
        draw, rounded corners=1.5pt, align=center,
        font=\scriptsize, inner sep=3pt,
        text width=33mm, minimum height=10mm
    },
    proofarrow/.style={
        -{Latex[length=1.5mm,width=1mm]}, line width=0.45pt
    }
]
\node[proofnode] (bias) at (0,0)
    {Perturbation bias\\
     Corollary~\ref{cor:finite-radius-gradient-bias}};
\node[proofnode] (tracking) at (4.3,0)
    {Coupled tracking\\
     Lemma~\ref{lem:primal-tracking};
     Theorem~\ref{thm:auxiliary-tracking}};
\node[proofnode] (descent) at (0,-1.65)
    {Outer descent\\
     Lemma~\ref{lem:outer-descent}};
\node[proofnode] (closure) at (4.3,-1.65)
    {Lyapunov descent\\
     Theorem~\ref{thm:lyapunov-closure}};
\node[proofnode] (stationarity) at (8.6,-1.65)
    {Stationarity and complexity\\
     Theorems~\ref{thm:global-stationarity}
     and~\ref{thm:minimization-complexity}};
\draw[proofarrow] (bias) -- (descent);
\draw[proofarrow] (tracking) -- (descent);
\draw[proofarrow] (tracking) -- (closure);
\draw[proofarrow] (descent) -- (closure);
\draw[proofarrow] (closure) -- (stationarity);
\end{tikzpicture}
\caption{Dependencies among the main convergence and complexity results.}
\label{fig:proof-dependency}
\end{figure}

Throughout this section, Assumptions~\ref{ass:implicit-layers}
and~\ref{ass:outer-function} hold, and
$\Phi^\star:=\inf_x\Phi(x)>-\infty$.
Fix an admissible upper radius $\delta_{\rm adm}$ satisfying
\eqref{eq:admissible-radius} and
\eqref{eq:tracking-radius-condition}, and let
$0<\delta\le\delta_{\rm adm}$.
Choose the inner stepsizes independently of $\delta$ as in
Lemma~\ref{lem:delta-uniformity}, and fix inner budgets such that
\[
\sup_{0<\delta\le\delta_{\rm adm}}\tau_i<1,
\qquad
\sup_{0<\delta\le\delta_{\rm adm}}\tau_{d,i}<1,
\qquad i=1,\ldots,M.
\]
Uniform geometric contraction ensures that such budgets exist.

\subsection{Lyapunov descent}
\label{subsec:outer-descent}
We first express outer descent in terms of the tracking errors
and perturbation bias. Define
\[
b_i^E
:=
2C_{vE,1}\mathbf1_{\{i=1\}}
+2L_f^2\mathbf1_{\{i=M\}},
\qquad
b_i^Z
:=
2C_{vZ,1}\mathbf1_{\{i=1\}},
\qquad
b_\delta:=C_{\mathrm{bias},0}^2.
\]
When $M=1$, the two contributions to $b_1^E$ are added.

\begin{lemma}[Inexact outer descent]
\label{lem:outer-descent}
For any $\alpha>0$, the outer update in
Algorithm~\ref{alg:remid} satisfies
\begin{equation}
\label{eq:outer-descent}
\begin{aligned}
\Phi(x^{k+1})-\Phi(x^k)
\le{}&
-\frac{\alpha}{2}\|\nabla\Phi(x^k)\|^2
-\left(
\frac{1}{2\alpha}-\frac{L_\Phi}{2}
\right)\|s^{k+1}\|^2
\\
&+
\alpha\sum_{i=1}^M b_i^E\mathcal E_i^{k+1}
+\alpha\sum_{i=1}^M b_i^Z\mathcal Z_i^{k+1}
+\alpha b_\delta\delta^2.
\end{aligned}
\end{equation}
\end{lemma}

\begin{proof}
By $L_\Phi$-smoothness and $x^{k+1}=x^k-\alpha\widehat G^k$,
\[
\Phi(x^{k+1})
\le
\Phi(x^k)
-\alpha\langle\nabla\Phi(x^k),\widehat G^k\rangle
+\frac{L_\Phi\alpha^2}{2}\|\widehat G^k\|^2.
\]
Using
$-\langle a,b\rangle=-(\|a\|^2+\|b\|^2)/2+\|a-b\|^2/2$,
$s^{k+1}=-\alpha\widehat G^k$, and
\eqref{eq:gradient-total-error} gives \eqref{eq:outer-descent}.
\end{proof}

\paragraph{Forward--backward Lyapunov weights.}
The auxiliary recursion transfers error from level $i+1$ to
level $i$, whereas the primal recursion transfers error from
level $i-1$ to level $i$.
We therefore construct the auxiliary weights from level $1$
to $M$, absorbing the backward couplings first.
The resulting primal-error terms are then absorbed by constructing
the primal weights from level $M$ to $1$.
First define
\begin{equation}
\label{eq:auxiliary-lyapunov-weights-first}
\widetilde u_1
:=
\frac{b_1^Z}{1-\tau_{d,1}},
\qquad
\overline u_1:=\tau_{d,1}\widetilde u_1,
\end{equation}
and, for $i=2,\ldots,M$, successively define
\begin{equation}
\label{eq:auxiliary-lyapunov-weights}
\widetilde u_i
:=
\frac{b_i^Z+\beta_{d,i-1}\widetilde u_{i-1}}
{1-\tau_{d,i}},
\qquad
\overline u_i:=\tau_{d,i}\widetilde u_i.
\end{equation}
After this sweep, collect the coefficients of the new response errors as
\begin{equation}
\label{eq:effective-primal-coefficient}
\widehat b_i^E
:=
b_i^E
+\lambda_{d,i}\widetilde u_i
+\mathbf1_{\{i\ge2\}}\lambda_{d,i-1}\widetilde u_{i-1}
+\mathbf1_{\{i\le M-1\}}\psi_{d,i+1}\widetilde u_{i+1}.
\end{equation}
Set $\widetilde\omega_{M+1}:=0$ and $\lambda_{M+1}:=0$.  For
$i=M,M-1,\ldots,1$, define
\begin{equation}
\label{eq:forward-primal-load}
q_i
:=
\widehat b_i^E+\lambda_{i+1}\widetilde\omega_{i+1},
\end{equation}
\begin{equation}
\label{eq:primal-lyapunov-weights}
\widetilde\omega_i
:=
\frac{q_i+\psi'_{d,i}\widetilde u_i}{1-\tau_i},
\qquad
\overline\omega_i
:=
\tau_i\widetilde\omega_i+\psi'_{d,i}\widetilde u_i.
\end{equation}
Thus $\widetilde\omega_i=\overline\omega_i+q_i$.  All $\widetilde u_i$ are determined before the recursion for
$\widetilde\omega_i$ begins. The weights are finite and nonnegative.

Define the temporal-drift coefficient
\begin{equation}
\label{eq:Gamma-definition}
\Gamma(\delta)
:=
\sum_{i=1}^M
\left(
\widetilde\omega_i\gamma_i
+\widetilde u_i\gamma_{d,i}
\right).
\end{equation}
The quantity $\Gamma(\delta)$ collects the tracking terms
proportional to $\|s^k\|^2$.
To choose one outer stepsize for all admissible radii, we next
bound it uniformly in $\delta$.

\begin{lemma}[Uniform bound on the Lyapunov coefficient]
\label{lem:Gamma-uniformity}
Under the $\delta$-independent inner stepsizes and budgets of
Lemma~\ref{lem:delta-uniformity}, there exists
$\bar\Gamma_{\rm adm}<\infty$, independent of
$\delta\in(0,\delta_{\rm adm}]$, such that
\begin{equation}
\label{eq:Gamma-uniform-majorant}
0\le \Gamma(\delta)\le \bar\Gamma_{\rm adm}
\qquad
\text{for all }0<\delta\le\delta_{\rm adm}.
\end{equation}
The majorant depends only on the data on the admissible parameter interval and the fixed
inner-solver choices, and hence is independent of the target accuracy.
\end{lemma}

\begin{proof}
Lemma~\ref{lem:delta-uniformity} bounds the tracking coefficients uniformly
in $\delta$. The fixed budgets keep $1-\tau_i$ and $1-\tau_{d,i}$
uniformly positive. Successive substitution in the weight recursions
therefore bounds all weights and $\Gamma(\delta)$ uniformly.
An explicit majorant is given in Appendix~\ref{app:supporting-proofs}.
\end{proof}

\begin{theorem}[Forward--backward Lyapunov descent]
\label{thm:lyapunov-closure}
Suppose Assumptions~\ref{ass:implicit-layers} and \ref{ass:outer-function}
hold and $\Phi^\star:=\inf_x\Phi(x)>-\infty$.
Run Algorithm~\ref{alg:remid} with the schedule using updated inputs,
inner stepsizes as in Lemma~\ref{lem:delta-uniformity} and fixed budgets
satisfying $\sup_{0<\delta\le\delta_{\rm adm}}\tau_i<1$ and
$\sup_{0<\delta\le\delta_{\rm adm}}\tau_{d,i}<1$ for every $i$, and a fixed parameter
$0<\delta\le\delta_{\rm adm}$ satisfying
\eqref{eq:admissible-radius} and~\eqref{eq:tracking-radius-condition}.
Use the weights constructed above and initialize $s^0=0$.  Define
\begin{equation}
\label{eq:Lyapunov-function}
\mathcal V^k
:=
\Phi(x^k)-\Phi^{\star}
+\alpha\sum_{i=1}^M\overline\omega_i\mathcal E_i^k
+\alpha\sum_{i=1}^M\overline u_i\mathcal Z_i^k
+\alpha\Gamma(\delta)\|s^k\|^2.
\end{equation}
If
\begin{equation}
\label{eq:outer-stepsize-condition}
\alpha
\le
\frac{1}{2\max\{L_\Phi,\sqrt{\Gamma(\delta)}\}},
\end{equation}
then
\begin{equation}
\label{eq:Lyapunov-descent}
\mathcal V^{k+1}-\mathcal V^k
\le
-\frac{\alpha}{2}\|\nabla\Phi(x^k)\|^2
+\alpha b_\delta\delta^2.
\end{equation}
Moreover, the condition, independent of $\delta$,
\begin{equation}
\label{eq:outer-stepsize-uniform}
\alpha
\le
\frac{1}{2\max\{L_\Phi,\sqrt{\bar\Gamma_{\rm adm}}\}}
\end{equation}
is sufficient for every $\delta\in(0,\delta_{\rm adm}]$.
\end{theorem}

\begin{proof}
We display the two elimination sweeps explicitly.  First substitute the
auxiliary recursions in the order $i=1,\ldots,M$.  The definitions
\eqref{eq:auxiliary-lyapunov-weights-first}--
\eqref{eq:effective-primal-coefficient} give
\begin{equation}
\label{eq:auxiliary-elimination-display}
\begin{aligned}
\sum_{i=1}^M(b_i^Z+\overline u_i)\mathcal Z_i^{k+1}
\le{}&
\sum_{i=1}^M\overline u_i\mathcal Z_i^k
+\sum_{i=1}^M(\widehat b_i^E-b_i^E)\mathcal E_i^{k+1}
+\sum_{i=1}^M\psi'_{d,i}\widetilde u_i\mathcal E_i^k
+\sum_{i=1}^M\widetilde u_i\gamma_{d,i}\|s^k\|^2.
\end{aligned}
\end{equation}
Indeed, the coefficient passed from $\mathcal Z_i^{k+1}$ to
$\mathcal Z_{i+1}^{k+1}$ is $\beta_{d,i}\widetilde u_i$, which is precisely
the load included in the definition of $\widetilde u_{i+1}$.

After \eqref{eq:auxiliary-elimination-display}, eliminate the primal
recursions in the reverse order $i=M,\ldots,1$.  Equations
\eqref{eq:forward-primal-load}--\eqref{eq:primal-lyapunov-weights} yield
\begin{equation}
\label{eq:primal-elimination-display}
\begin{aligned}
&\sum_{i=1}^M
(\overline\omega_i+\widehat b_i^E)\mathcal E_i^{k+1}
+\sum_{i=1}^M\psi'_{d,i}\widetilde u_i\mathcal E_i^k
\le
\sum_{i=1}^M\overline\omega_i\mathcal E_i^k
+\sum_{i=1}^M\widetilde\omega_i\gamma_i\|s^k\|^2.
\end{aligned}
\end{equation}
Here the load $\lambda_{i+1}\widetilde\omega_{i+1}$ is propagated from level
$i+1$ to level $i$, while
$\overline\omega_i=\tau_i\widetilde\omega_i+
\psi'_{d,i}\widetilde u_i$ leaves exactly the old-time Lyapunov weight.
Thus every new-time tracking term is eliminated once, and
\eqref{eq:Gamma-definition} is the total coefficient of $\|s^k\|^2$.

Combining with \eqref{eq:outer-descent} gives
\[
\begin{aligned}
\mathcal V^{k+1}-\mathcal V^k
\le{}&
-\frac{\alpha}{2}\|\nabla\Phi(x^k)\|^2
-\left(
\frac{1}{2\alpha}-\frac{L_\Phi}{2}-\alpha\Gamma(\delta)
\right)\|s^{k+1}\|^2
+\alpha b_\delta\delta^2.
\end{aligned}
\]
Condition \eqref{eq:outer-stepsize-condition} implies
$\alpha L_\Phi\le1/2$ and $\alpha^2\Gamma(\delta)\le1/4$, so the coefficient
of $\|s^{k+1}\|^2$ is nonnegative.  Dropping this term yields
\eqref{eq:Lyapunov-descent}.  Finally,
\eqref{eq:Gamma-uniform-majorant} implies that
\eqref{eq:outer-stepsize-uniform} is sufficient uniformly over the admissible
parameter interval.
\end{proof}

\subsection{Stationarity guarantees}
\label{subsec:global-convergence}

Telescoping the Lyapunov descent yields a stationarity bound with
an $O(K^{-1})$ optimization term and an $O(\delta^2)$ bias term.

\begin{theorem}[Non-asymptotic stationarity]
\label{thm:global-stationarity}
Under the conditions of Theorem~\ref{thm:lyapunov-closure}, for every
$K\ge1$,
\begin{equation}
\label{eq:average-stationarity}
\frac{1}{K}\sum_{k=0}^{K-1}\|\nabla\Phi(x^k)\|^2
\le
\frac{2\mathcal V^0}{\alpha K}
+2C_{\mathrm{bias},0}^2\delta^2.
\end{equation}
Consequently,
\[
\min_{0\le k<K}\|\nabla\Phi(x^k)\|^2
\le
\frac{2\mathcal V^0}{\alpha K}
+2C_{\mathrm{bias},0}^2\delta^2.
\]
If $R$ is sampled uniformly from $\{0,\ldots,K-1\}$,
the same bound holds for
$\mathbb E_R\|\nabla\Phi(x^R)\|^2$.
This output rule requires no exact-gradient evaluations.
\end{theorem}

\begin{proof}
Sum \eqref{eq:Lyapunov-descent} over $k=0,\ldots,K-1$, use
$\mathcal V^K\ge0$, and divide by $\alpha K/2$.
\end{proof}

\begin{remark}[Minimization and saddle response chains]
Theorem~\ref{thm:global-stationarity} applies to the full
signed-potential response class, including chains mixing
strongly convex minimization and strongly convex--strongly
concave saddle responses.
Positive definiteness of the potential Hessians is not required.
The explicit single-$\kappa$ complexity bounds established below
specialize to the homogeneous strongly convex minimization regime.
\end{remark}

\begin{corollary}[$\epsilon$-stationarity]
\label{cor:epsilon-stationarity}
Under the conditions of Theorem~\ref{thm:global-stationarity},
let $\epsilon>0$. If
\[
\delta
\le
\min\left\{
\delta_{\rm adm},
\frac{\epsilon}{2C_{\mathrm{bias},0}}
\right\},
\qquad
K\ge\frac{4\mathcal V^0}{\alpha\epsilon^2},
\]
then
$\min_{0\le k<K}\|\nabla\Phi(x^k)\|\le\epsilon$.
The randomized output above also satisfies
$\mathbb E_R\|\nabla\Phi(x^R)\|^2\le\epsilon^2$.
If $C_{\mathrm{bias},0}=0$, the second restriction on $\delta$ is omitted.
\end{corollary}

\paragraph{Initialization and accuracy dependence.}
Paired initialization gives $d_i^0=0$ and
$\mathcal Z_i^0\le C_{\mathrm{path},i}^2$.
With fixed unperturbed initial states and the uniform stepsize
\eqref{eq:outer-stepsize-uniform}, the uniformly bounded weights imply
$\sup_{0<\delta\le\delta_{\rm adm}}\mathcal V^0(\delta)<\infty$.
Thus, for fixed problem data and depth, choosing $\delta=O(\epsilon)$
gives $\epsilon$-stationarity in $O(\epsilon^{-2})$ outer iterations
with fixed inner budgets.
The same conclusion holds for any initialization with a uniformly
bounded $\mathcal V^0$.

\subsection{Conditioning and oracle complexity}
\label{subsec:complexity}

We now translate the stationarity guarantee into outer-iteration
and first-order oracle complexity bounds.
The condition-number estimates below hold for fixed depth $M$;
their hidden constants may depend on $M$.

\paragraph{General complexity bounds.}
Let
\[
\mu_{\min}:=\min_i\mu_i,
\qquad
\mu_{\max}:=\max_i\mu_i,
\qquad
\kappa:=\max_i\frac{L_i}{\mu_i},
\qquad
\kappa_\mu:=\frac{\mu_{\max}}{\mu_{\min}}.
\]
Then $L_{i+1}/\mu_i\le\kappa\kappa_\mu$.  The tracking restriction already implies
$\delta_{\rm adm}L_f\le\mu_M/2$. Recall that
\[
\bar\mu_M=\mu_M-\delta_{\rm adm}L_f,
\qquad
\bar L_M=L_M+\delta_{\rm adm}L_f.
\]
For every $0<\delta\le\delta_{\rm adm}$,
\[
\widetilde\mu_M\ge\bar\mu_M\ge\frac{\mu_M}{2},
\qquad
\widetilde L_M\le\bar L_M\le L_M+\frac{\mu_M}{2},
\qquad
\frac{\bar L_M}{\bar\mu_M}\le2\kappa+1.
\]

Choose
\[
\alpha
=
\frac{1}{2\max\{L_\Phi,\sqrt{\bar\Gamma_{\rm adm}}\}},
\qquad
\delta
\le
\min\left\{
\delta_{\rm adm},
\frac{\epsilon}{2C_{\mathrm{bias},0}}
\right\},
\]
with the second parameter restriction omitted when
$C_{\mathrm{bias},0}=0$.  Let $\overline{\mathcal V}_0<\infty$ be any
bound uniform in $\delta$ satisfying
$\mathcal V^0(\delta)\le\overline{\mathcal V}_0$ for all admissible $\delta$.
Paired initialization with fixed unperturbed warm starts provides such a bound.  Corollary~\ref{cor:epsilon-stationarity}
then gives
\begin{equation}
\label{eq:general-outer-complexity}
K
=
O\!\left(
\overline{\mathcal V}_0
\max\{L_\Phi,\sqrt{\bar\Gamma_{\rm adm}}\}\epsilon^{-2}
\right).
\end{equation}

Under the blockwise oracle convention of
Section~\ref{subsec:problem-formulation}, one outer iteration uses
$2\sum_{i=1}^MN_i+N_M+2M+1$ first-order calls.  Hence
\begin{equation}
\label{eq:general-fo-complexity}
\mathcal N_{\rm FO}
=
O\!\left(
K\left[M+\sum_{i=1}^MN_i\right]
\right).
\end{equation}

\paragraph{Strongly convex minimization chains.}
For strongly convex minimization responses, gradient descent
provides a contraction gap of order $\kappa^{-1}$.
Suppose $P_i=I$ and $\Psi_i=g_i$, with each
$g_i(\cdot,u)$ uniformly $\mu_i$-strongly convex and $L_i$-smooth.
Choose $\eta_i=1/L_i$ for $i<M$ and $\eta_M=1/\bar L_M$. Lemma~\ref{lem:delta-uniformity} and the spectral
contraction estimate in Lemma~\ref{lem:inner-contraction} imply
\[
\overline\vartheta_i
\le
\left(1-\frac{1}{3\kappa}\right)^2,
\qquad
1-\theta_i
\ge
\frac{5}{18\kappa},
\qquad i=1,\ldots,M.
\]
Thus
\begin{equation}
\label{eq:minimization-inner-interface}
\max\{\chi_i,\chi_{d,i}\}\le\xi_i
\qquad\text{with}\qquad
N_i=O\!\left(\kappa\log\frac2{\xi_i}\right).
\end{equation}

For explicit $\kappa$-rates, assume the homogeneous fixed-depth regime
\begin{equation}
\label{eq:single-kappa-regularity}
\kappa_\mu=O(1),
\qquad
\mu_{\min}=\Theta(\kappa^{-1}),
\qquad
\max_{1\le i\le M}\{L_i,\rho_i\}+L_f+C_f+\rho_f=O(1).
\end{equation}
Choose the admissible parameter independently of $\epsilon$ so that
\begin{equation}
\label{eq:minimization-uniform-radius}
0<\delta_{\rm adm}
\le c_{\rm adm}\kappa^{-(2M+1)}
\end{equation}
for a sufficiently small constant $c_{\rm adm}$, and require
\begin{equation}
\label{eq:forward-contraction-accuracy}
\chi_{d,i}\le c\kappa^{-(4M+4)},
\qquad i=1,\ldots,M,
\end{equation}
with $c>0$ sufficiently small. For $M\ge2$,
Proposition~\ref{prop:minimization-kappa-bookkeeping} in
Appendix~\ref{app:lyapunov-details} gives
\begin{equation}
\label{eq:minimization-bookkeeping-orders}
C_{\mathrm{bias},0}=O(\kappa^{2M+1}),
\qquad
L_\Phi=O(\kappa^{2M+1}).
\end{equation}
Moreover,
\begin{equation}
\label{eq:forward-Gamma-order}
\bar\Gamma_{\rm adm}=O(\kappa^{2M-2}).
\end{equation}
Since
$\chi_{d,i}\le\exp(-5N_i/(18\kappa))$,
\eqref{eq:forward-contraction-accuracy} is achieved with
$N_i=O(\kappa M\log(2+\kappa))$ at fixed depth.

Consequently, for constants
$C_N,c_{\rm adm},c_\delta,c_\alpha>0$ independent of $\kappa$ and $\epsilon$,
we may choose
\begin{equation}
\label{eq:minimization-parameter-choice}
N_i
=
\left\lceil C_N\kappa M\log(2+\kappa)\right\rceil,
\qquad
0<\delta
\le
\min\left\{
\delta_{\rm adm},
 c_\delta\epsilon\kappa^{-(2M+1)}
\right\},
\qquad
\alpha=c_\alpha\kappa^{-(2M+1)},
\end{equation}
where $\delta_{\rm adm}$ satisfies
\eqref{eq:minimization-uniform-radius} and $c_\alpha$ is chosen so that
\eqref{eq:outer-stepsize-uniform} holds.  For $M\ge2$, the bounds on $L_\Phi$ and
$\bar\Gamma_{\rm adm}$ ensure that this choice of $\alpha$
satisfies \eqref{eq:outer-stepsize-uniform}.
The case $M=1$ follows from the dedicated estimates in
Appendix~\ref{app:bilevel-details}.

\begin{theorem}[Complexity for strongly convex minimization chains]
\label{thm:minimization-complexity}
Fix $M\ge1$ and $0<\epsilon\le1$.  Suppose the strongly convex minimization
specialization satisfies Assumptions~\ref{ass:implicit-layers} and~\ref{ass:outer-function}
and, \eqref{eq:single-kappa-regularity},
$\Phi^{\star}>-\infty$, and
\[
\Phi(x^0)-\Phi^{\star}=O(1),
\qquad
\mathcal E_i^0=O(1),
\quad i=1,\ldots,M.
\]
Initialize $z_{i,\delta}^0=z_i^0$ and $s^0=0$.  Use inner stepsizes
$\eta_i=1/L_i$ for $i<M$ and $\eta_M=1/\bar L_M$, and the
parameter choices in \eqref{eq:minimization-parameter-choice}, with
$C_N$ sufficiently large and $c_{\rm adm},c_\delta,c_\alpha$
sufficiently small.  These constants may depend on fixed $M$ and the
regularity bounds, but not on $\kappa$ or $\epsilon$.
Then Algorithm~\ref{alg:remid} reaches
$\min_{0\le k<K}\|\nabla\Phi(x^k)\|\le\epsilon$,
and a uniformly sampled iterate satisfies
$\mathbb E_R\|\nabla\Phi(x^R)\|^2\le\epsilon^2$,
within
\begin{equation}
\label{eq:minimization-complexity}
K
=
O\!\left(\kappa^{2M+1}\epsilon^{-2}\right),
\qquad
\mathcal N_{\rm FO}
=
\widetilde O\!\left(
M^2\kappa^{2M+2}\epsilon^{-2}
\right).
\end{equation}
\end{theorem}

\begin{proof}
For $M\ge2$, Proposition~\ref{prop:minimization-kappa-bookkeeping}
gives
$\max\{L_\Phi,\sqrt{\bar\Gamma_{\rm adm}}\}
=O(\kappa^{2M+1})$.
The weight estimates and paired initialization yield
$\mathcal V^0=O(1)$, as verified in
Appendix~\ref{app:lyapunov-details}.
Thus \eqref{eq:general-outer-complexity} gives
$K=O(\kappa^{2M+1}\epsilon^{-2})$.

For $M=1$, the bounds and initialization estimates in
Appendix~\ref{app:bilevel-details} similarly give
$\max\{L_\Phi,\sqrt{\bar\Gamma_{\rm adm}}\}=O(\kappa^3)$
and $\mathcal V^0=O(1)$, yielding
$K=O(\kappa^3\epsilon^{-2})$.
Finally, $\sum_{i=1}^M N_i=O(M^2\kappa\log(2+\kappa))$
at fixed depth, and \eqref{eq:general-fo-complexity}
gives the stated oracle complexity.
\end{proof}

\begin{example}[Structural amplification with depth]
\label{ex:depth-amplification}
Consider
\[
g_i(z,u)=\frac{\mu_i}{2}z^2-b_i z u,
\qquad \mu_i>0,
\qquad i=1,\ldots,M,
\]
and
$f(x,z_M)=a z_M+\frac{\tau}{2}x^2$ with $\tau>0$.  Then
\[
S_i(u)=\frac{b_i}{\mu_i}u,
\qquad
v_0^{\star}=a\prod_{i=1}^M\frac{b_i}{\mu_i},
\qquad
\nabla\Phi(x)=\tau x+v_0^{\star}.
\]
Thus, if $b_i/\mu_i=q>1$ for all $i$, the exact reverse sensitivity grows as
$q^M$.  This amplification is already present in the exact quadratic chain,
before any inexact-solve approximation is introduced.
\end{example}

\begin{remark}[Depth dependence]
\label{rem:depth-dependence}
Example~\ref{ex:depth-amplification} shows that multiplicative sensitivity
propagation with depth is structural unless additional assumptions, such as
contractive response Jacobians, are imposed.  It does not establish the
tightness of the exponent $2M+1$ in
Theorem~\ref{thm:minimization-complexity}, which also reflects worst-case
smoothness propagation and squared-error bounds.
\end{remark}

\begin{remark}[Bilevel consistency]
For $M=1$, the bounds reduce to
$K=O(\kappa^3\epsilon^{-2})$ and
$\mathcal N_{\rm FO}=\widetilde O(\kappa^4\epsilon^{-2})$.
The oracle bound has the same condition-number and accuracy powers as
the nonaccelerated fully first-order bilevel order in
\cite{chen2025f2ba}; this is an order comparison, not an equivalence
of problem classes or assumptions.
\end{remark}

\begin{remark}[Centered refinement]
\label{rem:centered-complexity}
Appendix~\ref{app:centered} establishes $O(\delta^2)$ bias for an exact
centered estimator under stronger smoothness; the finite-iteration
complexity of that construction is not analyzed here.
\end{remark}
\begin{table}[t]
\centering
\caption{Solver-dependent complexity consequences.  The minimization-chain
rates use the homogeneous fixed-depth regime in
Theorem~\ref{thm:minimization-complexity}.  For general potential-based
layers, $\xi_i$ denotes a prescribed contraction factor.}
\label{tab:complexity-summary}
\small
\setlength{\tabcolsep}{3.5pt}
\renewcommand{\arraystretch}{1.15}
\begin{tabular}{@{}lccc@{}}
\toprule
\textbf{Structure}
&\textbf{Inner work / layer}
&\textbf{Outer iterations}
&\textbf{FO oracle complexity}
\\
\midrule
SC minimization chain (fixed $M\ge1$)
&$\widetilde O(M\kappa)$
&$O(\kappa^{2M+1}\epsilon^{-2})$
&$\widetilde O(M^2\kappa^{2M+2}\epsilon^{-2})$
\\
General potential-based layer
&$O\!\left(\kappa^2\log(2/\xi_i)\right)$
&\multicolumn{2}{c}{general Lyapunov bound}
\\
\bottomrule
\end{tabular}
\end{table}
\paragraph{General potential-based chains.}
For general potential-based layers, take
$\eta_i=\mu_i/L_i^2$ for $i<M$ and
$\eta_M=\bar\mu_M/\bar L_M^2$.  Lemma~\ref{lem:delta-uniformity} gives a
squared contraction gap bounded below by a constant multiple of
$\kappa^{-2}$, uniformly in $\delta$. For a prescribed tolerance
$0<\xi_i<1$, this gives
\begin{equation}
\label{eq:signed-potential-inner-interface}
\max\{\chi_i,\chi_{d,i}\}\le\xi_i
\qquad\text{with}\qquad
N_i=O\!\left(\kappa^2\log\frac2{\xi_i}\right).
\end{equation}
For any admissible choice of contraction factors,
\begin{equation}
\label{eq:signed-potential-complexity-interface}
K
=
O\!\left(
\overline{\mathcal V}_0
\max\{L_\Phi,\sqrt{\bar\Gamma_{\rm adm}}\}\epsilon^{-2}
\right),
\qquad
\mathcal N_{\rm FO}
=
O\!\left(K\left[M+\sum_{i=1}^M N_i\right]\right),
\end{equation}
Here $\bar\Gamma_{\rm adm}$ is the corresponding uniform Lyapunov bound.
These bounds yield $O(\epsilon^{-2})$ outer iterations for fixed
problem data, depth, and admissible contraction requirements.
The conditioning dependence remains expressed through
$L_\Phi$, $\bar\Gamma_{\rm adm}$, and the inner budgets.
For prescribed tolerances $\xi_i=\kappa^{-O(M)}$,
the inner work is $\widetilde O(M\kappa^2)$ per level.

\paragraph{Mixed chains.}
Partition $\{1,\ldots,M\}$ into $\mathcal I_{\rm min}$, the strongly
convex minimization levels, and $\mathcal I_{\rm sp}$, the remaining
signed-potential levels. Under the standing assumptions, a common radius
satisfying \eqref{eq:tracking-radius-condition}, and prescribed contraction
requirements with $\log(2/\xi_i)=O(M\log(2+\kappa))$,
\begin{equation}
\label{eq:mixed-layerwise-budget}
N_i
=
\begin{cases}
\widetilde O(M\kappa),&i\in\mathcal I_{\rm min},\\[1mm]
\widetilde O(M\kappa^2),&i\in\mathcal I_{\rm sp}.
\end{cases}
\end{equation}
Combining these solver costs with \eqref{eq:general-fo-complexity} gives
\begin{equation}
\label{eq:mixed-fo-complexity}
\mathcal N_{\rm FO}
=
\widetilde O\!\left(
KM\left[
1+|\mathcal I_{\rm min}|\kappa+|\mathcal I_{\rm sp}|\kappa^2
\right]
\right).
\end{equation}
This bound separates the layerwise solver costs from the outer
iteration count $K$, which is determined by
\eqref{eq:general-outer-complexity} for the chosen contraction
requirements.

\section{Numerical experiments}
\label{sec:numerics}

We first examine the perturbation and finite-solve errors in controlled
examples, then compare computational performance on multilevel PDE control
and transductive graph learning. The baselines use recursive implicit
differentiation with CG for minimization responses (ID-CG) and MINRES for
saddle responses (ID-MINRES), or differentiate finite response trajectories
(Sato/ITD and its signed-gradient counterpart, ITD-SG)
\cite{grazzi2020iteration,paige1975solution,sato2021multilevel}.
All application comparisons use a common algorithm-time budget per method,
including the required derivative computations.  Experimental details are collected in
Appendices~\ref{app:experimental-protocol}--\ref{app:domain-adaptation}.

\subsection{Perturbation accuracy and finite-solve errors}
\label{subsec:controlled-diagnostics}

Figure~\ref{fig:controlled-diagnostics}(a) isolates perturbation bias using
nonlinear strongly convex chains and high-accuracy response solves.
The observed log--log slopes are close to one for the one-sided estimator
and two for the centered estimator. The latter is used only as a bias
diagnostic; \ReMID{} uses one-sided perturbations.
For quadratic chains with an affine terminal objective, the exact
perturbation estimator has no bias. Panel~(b) therefore measures only
finite-solve error: its decay largely aligns across condition numbers
when plotted against $T/\kappa$, while deeper chains require more inner
work to reach a given error. Panel~(c) recovers the same first-order bias
for minimization, saddle, and mixed chains.

\begin{figure}[t]
\centering
\includegraphics[width=.98\textwidth]
{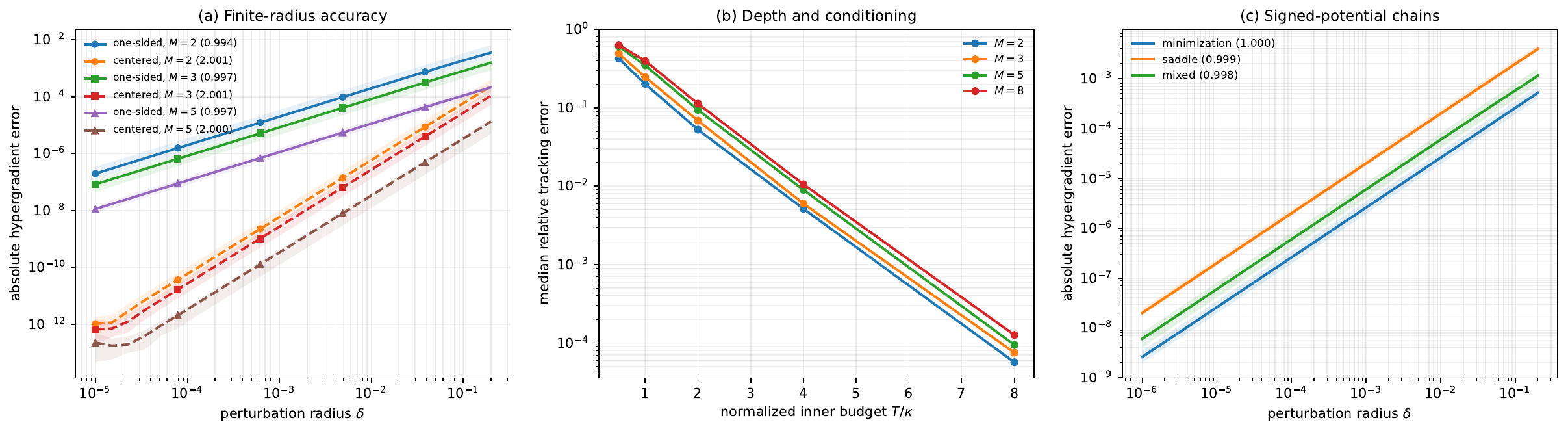}
\caption{Controlled diagnostics: (a) one-sided and centered perturbation
bias at depths $M=2,3,5$; (b) finite-solve error versus $T/\kappa$;
(c) one-sided bias for minimization, saddle, and mixed chains.
Curves show medians with interquartile ranges.}
\label{fig:controlled-diagnostics}
\end{figure}

Two additional tests distinguish finite-solve error from perturbation bias
and assess the forward refresh. At fixed inner budgets, reducing $\delta$
eventually reaches an error floor; increasing the budget lowers that floor.
With matched inner-gradient budgets, the updated-input schedule yields
smaller tracking errors than the previous-input schedule, especially at
greater depths. The constructions and results are given in
Appendix~\ref{app:controlled-diagnostics},
Figure~\ref{fig:finite-computation-diagnostics}.

\subsection{Multilevel PDE-constrained optimal control}
\label{subsec:numerical:pde}

We consider a five-level distributed-control problem for the
two-dimensional heat equation. With $z_0=x$, each response minimizes
\[
\Psi_i(z,u)=
\frac{\omega_i}{2}\|M_i(\mathcal K_i z-d_i^{\rm tar})\|^2
+\frac{\lambda_c}{2}\|z-u\|^2+\frac{\mu}{2}\|z\|^2,
\]
where $\mathcal K_i$ maps a control to the terminal PDE state and $M_i$
selects observed locations. Each response has $4096$ variables.
The outer objective is
\[
\Phi(x)=\phi_5\!\left(\sqrt{\omega_{\rm val}}M_{\rm val}
(\mathcal K_{\rm val}z_M^\star(x)-d_{\rm val}^{\rm tar})\right)
+\frac{\tau}{2}\|x\|^2,
\]
where $\phi_c(q)=c^2(\sqrt{1+\|q\|^2/c^2}-1)$.
The responses are affine, so the quadratic outer regularizer makes
$\Phi$ strongly convex.

All methods use local GD steps $\eta_i=1.8/\widehat L_i$, with
$\widehat L_i$ estimated from the state Hessian.
Figure~\ref{fig:pde-main} compares progress over 300 algorithm seconds.
\ReMID{} decreases the objective gap and reference-gradient norm more
rapidly than the two baselines and attains the smallest mean terminal
values. ID-CG eventually overtakes Sato/ITD, whose curves flatten toward
the end of the budget. Terminal statistics and oracle
counts are reported in Table~\ref{tab:pde-main}.

\begin{figure}[t]
\centering
\begin{subfigure}[t]{0.485\textwidth}
\centering
\includegraphics[width=\linewidth]{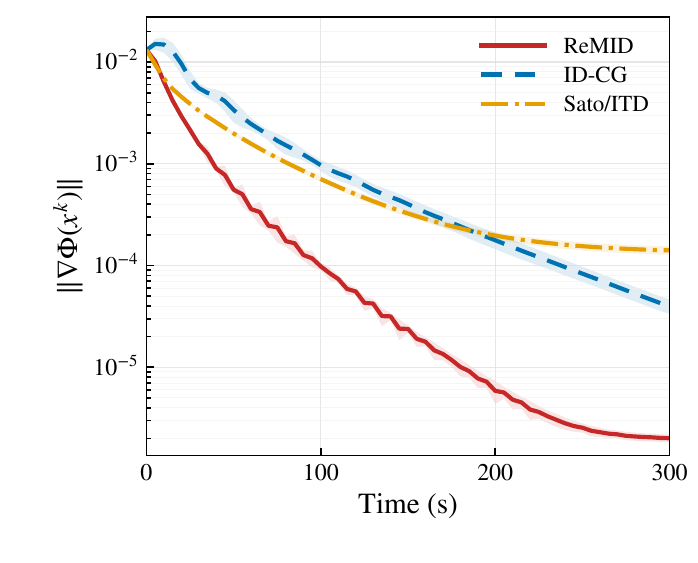}
\caption{Stationarity.}
\end{subfigure}\hfill
\begin{subfigure}[t]{0.485\textwidth}
\centering
\includegraphics[width=\linewidth]{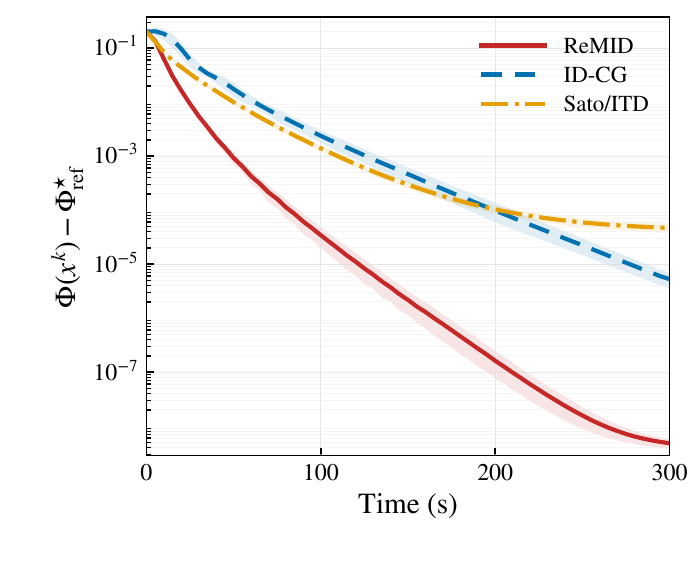}
\caption{Objective gap.}
\end{subfigure}
\caption{Five-level PDE control with a 300-second algorithm-time
budget. Lines and bands show means and one sample standard deviation
over five runs. Objective gaps use independently computed numerical
reference optima. Both vertical axes are logarithmic.}
\label{fig:pde-main}
\end{figure}

\paragraph{Saddle-point extension.}
We introduce an adversarial source at each level, giving strongly
convex--strongly concave responses in the control--disturbance pair
$z_i=(y_i,\xi_i)$, with $8192$ variables per level.
All three methods use the same local signed-gradient steps with
$P_i=\operatorname{diag}(I,-I)$; the potential and stepsizes are specified
in Appendix~\ref{app:pde-signed}.
Under a 500-second budget, \ReMID{} also attains the smallest mean
terminal objective gap and reference-gradient norm in this setting
(Figure~\ref{fig:pde-signed-time} and Table~\ref{tab:pde-signed}).
This comparison tests the method on indefinite response Hessians, for
which the implicit baseline uses MINRES rather than CG.

\begin{figure}[t]
\centering
\begin{subfigure}[t]{0.485\textwidth}
\centering
\includegraphics[width=\linewidth]{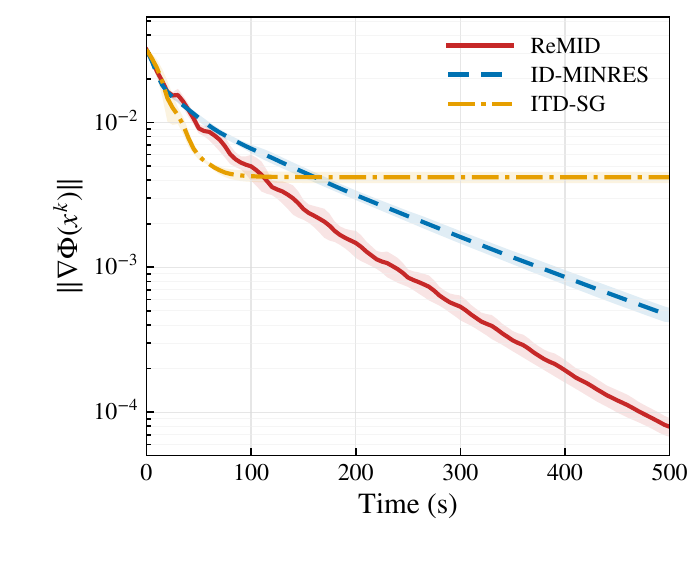}
\caption{Stationarity.}
\end{subfigure}\hfill
\begin{subfigure}[t]{0.485\textwidth}
\centering
\includegraphics[width=\linewidth]{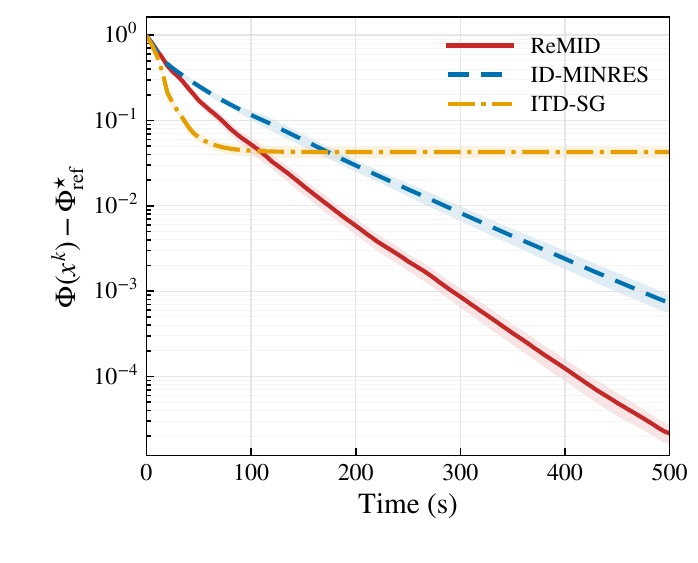}
\caption{Objective gap.}
\end{subfigure}
\caption{Saddle-point PDE with a 500-second algorithm-time
budget. Lines and bands show means and one sample standard deviation
over five runs. Objective gaps use independently computed numerical
reference optima. Both vertical axes are logarithmic.}
\label{fig:pde-signed-time}
\end{figure}

\subsection{Sequential transductive graph learning on MNIST}
\label{subsec:mnist-domain-adaptation}

We construct eight graph-learning levels from $2000$ MNIST images rotated
progressively to $60^\circ$. At each level, a strongly convex potential
combines a smooth nonlinear graph penalty, quadratic interlevel coupling,
and regularization:
\begin{equation}
\label{eq:mnist-main-response}
z_i^\star(x)=\operatorname*{arg\,min}_{z\in\mathbb R^{2000\times10}}
\Psi_i\bigl(z,z_{i-1}^\star(x)\bigr).
\end{equation}
The outer correction $x\in\mathbb R^{128\times10}$ acts on a fixed source
classifier through $z_0(x)=A_1(W_{\rm src}+x)$, where $A_1$ contains fixed
random Fourier features. Source labels anchor the first response; a
disjoint labeled set $\mathcal V$ defines the outer objective
\begin{equation}
\label{eq:mnist-main-outer}
\Phi(x)=\operatorname{CE}_{\mathcal V}(z_M^\star(x),y)
+\frac{\tau}{2}\|x\|_F^2.
\end{equation}
Here $\operatorname{CE}_{\mathcal V}$ is mean cross-entropy. All images
contribute graph features, but the remaining labels are unused.

Figure~\ref{fig:mnist-main-panels} compares optimization over 150 algorithm
seconds. ID-CG and Sato/ITD exhibit larger initial overshoots before
recovering. \ReMID{} reaches the lowest mean terminal objective difference
and reference-gradient norm among these configurations
(Table~\ref{tab:mnist-domain-adaptation}). Since the nonlinear response
composition need not be convex, objective differences use empirical
reference values rather than certified global optima.

\begin{figure}[t]
\centering
\begin{subfigure}[t]{0.485\textwidth}
\centering
\includegraphics[width=\linewidth]{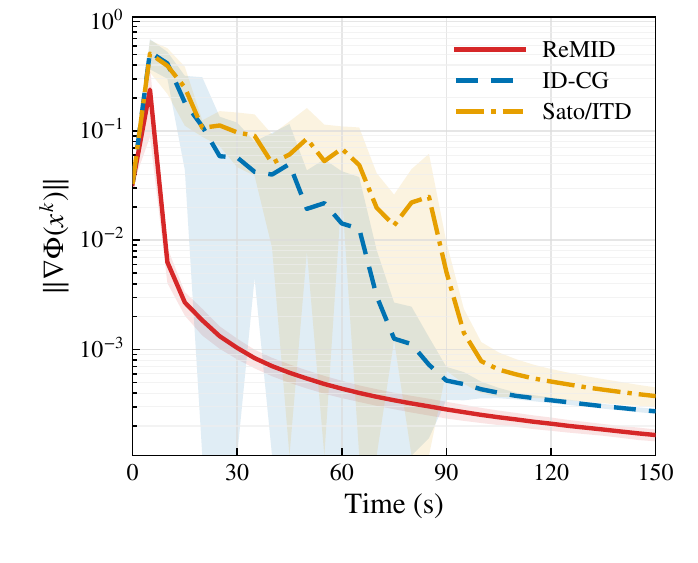}
\caption{Stationarity.}
\label{fig:mnist-stationarity-time}
\end{subfigure}\hfill
\begin{subfigure}[t]{0.485\textwidth}
\centering
\includegraphics[width=\linewidth]{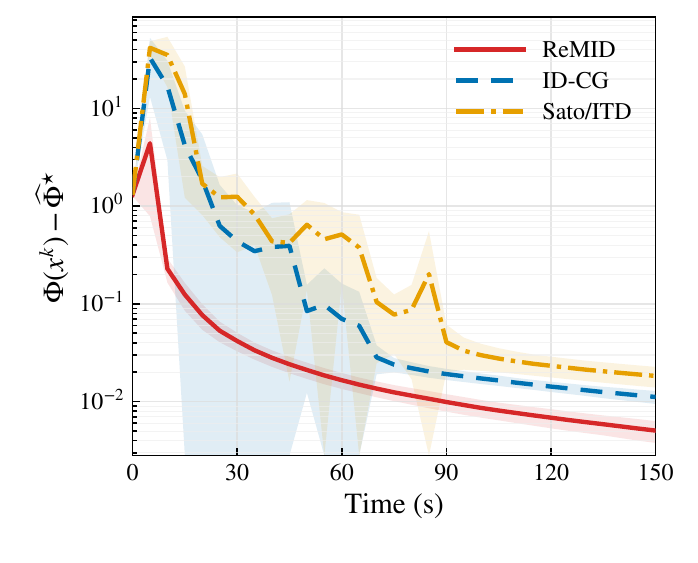}
\caption{Objective difference.}
\label{fig:mnist-main-obj-time}
\end{subfigure}
\caption{Eight-level MNIST graph learning with a 150-second algorithm-time
budget. Lines and bands show means and one sample standard deviation
over five runs. Panel~(b) subtracts an independently optimized empirical
reference value for each run. Both vertical axes are logarithmic.}
\label{fig:mnist-main-panels}
\end{figure}

\section{Conclusion}
\label{sec:conclusion}

We established convergence of recursive first-order differentiation
with inexact response solves for potential-based response compositions.  The central
step is a coupled tracking estimate: downstream adjoint errors enter
upstream perturbation problems, while unperturbed-state errors propagate
forward.  Bounds for the scaled response differences that are uniform in $\delta$ and a
triangular Lyapunov construction control these two dependencies.
Together with the perturbation bias, they yield
$O(K^{-1})+O(\delta^2)$ average squared stationarity for the exact
composite objective.
For fixed depth and problem data, this gives an $O(\epsilon^{-2})$
outer-iteration dependence.  The strongly convex specialization provides
explicit condition-number bounds, while the potential-based case retains
the response-solver contraction parameters.
The controlled experiments examine the bias and inexact-solve mechanisms.
Across PDE control, saddle-point PDE, and MNIST graph learning,
\ReMID{} attains the smallest mean terminal objective differences and
reference-gradient norms among the tested configurations under the
specified algorithm-time budgets.

The theory does not cover nonsymmetric implicit or general
variational-inequality responses, and its single-$\kappa$ bounds are not
uniform in depth.  Extending the perturbation principle to such systems and
sharpening its depth dependence remain open.



\appendix

\section{Supporting regularity and tracking proofs}
\label{app:supporting-proofs}

This appendix collects the supporting regularity and local tracking arguments
for Sections~\ref{sec:model}--\ref{sec:algorithm}.

\subsection{Implicit-response regularity}

\begin{proof}[Proof of Proposition~\ref{prop:response-regularity}]
\emph{Well-posedness.} Fix $i,u$ and any $b\in\mathbb R^{d_i}$.
The Hessian bound makes $T_i(\cdot,u)$ $L_i$-Lipschitz. For
$0<a<2\mu_i/L_i^2$, the map
$z\mapsto z-a(T_i(z,u)-b)$ is a contraction on $\mathbb R^{d_i}$,
with squared factor $1-2a\mu_i+a^2L_i^2<1$.
The Banach fixed-point theorem gives a unique solution of $T_i(z,u)=b$.
Taking $b=0$ proves well-posedness; the same argument covers additive
perturbations.

\emph{Lipschitz continuity.} Let $z:=S_i(u)$ and $z':=S_i(u')$. Since
$T_i(z,u)=T_i(z',u')=0$, strong monotonicity of
$T_i(\cdot,u)$ gives
$\mu_i\|z-z'\|^2 \le \langle T_i(z',u')-T_i(z',u),z-z'\rangle \le L_i\|u-u'\|\,\|z-z'\|.$
Dividing by $\mu_i\|z-z'\|$ when $z\neq z'$ yields
$\|S_i(u)-S_i(u')\|\le(L_i/\mu_i)\|u-u'\|$; the case $z=z'$
is immediate.

\emph{Differentiability.} For $A_i(u)=P_iH_i(u)$, strong monotonicity gives
$\langle A_i(u)h,h\rangle\ge\mu_i\|h\|^2$ and hence
$\|H_i(u)h\|=\|A_i(u)h\|\ge\mu_i\|h\|$.
Thus $\|H_i(u)^{-1}\|\le\mu_i^{-1}$.
The implicit-function theorem and uniqueness patch the local branches
into a global $C^1$ response. Differentiation gives
$H_iDS_i+B_i^\top=0$, with $\|DS_i\|\le\kappa_i$.

\emph{Jacobian Lipschitz continuity.} The response bound gives
$\max\{ \|H_i(u)-H_i(u')\|, \|B_i(u)-B_i(u')\| \} \le \rho_i(1+\kappa_i)\|u-u'\|.$
The inverse-difference identity
$H^{-1}-\widetilde H^{-1}
=H^{-1}(\widetilde H-H)\widetilde H^{-1}$
together with the inverse bound yields
$\|H_i(u)^{-1}-H_i(u')^{-1}\| \le \frac{\rho_i(1+\kappa_i)}{\mu_i^2} \|u-u'\|.$
Therefore,
\[
\begin{aligned}
\|DS_i(u)-DS_i(u')\|
&\le
\|H_i(u)^{-1}-H_i(u')^{-1}\|\,\|B_i(u)\|
+
\|H_i(u')^{-1}\|\,\|B_i(u)-B_i(u')\|\\
&\le
\rho_i(1+\kappa_i)
\left(
\frac{L_i}{\mu_i^2}
+
\frac{1}{\mu_i}
\right)
\|u-u'\|,
\end{aligned}
\]
which proves the result.
\end{proof}

\begin{proof}[Justification of Example~\ref{prop:saddle-layer}]
Fix $u$ and write $\mathcal H_i(z,u):=\nabla_{11}^2\Psi_i(z,u)$,
to distinguish this block from $H_i(u)$ evaluated at the response.
Since $D_zT_i=P_i\mathcal H_i$ and $\mathcal H_i$ is symmetric,
\[
\operatorname{sym}\!\bigl(D_zT_i(z,u)\bigr)
=
\begin{pmatrix}
\nabla_{yy}^2\Psi_i & 0\\
0 & -\nabla_{pp}^2\Psi_i
\end{pmatrix}
\succeq
\mu_i I.
\]
Here $\operatorname{sym}(A):=(A+A^\top)/2$.
Indeed, the mixed blocks cancel because
$\nabla_{p y}^2\Psi_i=(\nabla_{yp}^2\Psi_i)^\top$.
Integrating $D_zT_i$ along the line segment joining $z$ and $z'$
therefore gives
$\langle T_i(z,u)-T_i(z',u),z-z'\rangle \ge \mu_i\|z-z'\|^2.$
Thus $T_i(\cdot,u)$ is $\mu_i$-strongly monotone.

By Proposition~\ref{prop:response-regularity}, $T_i(\cdot,u)$ has a
unique zero $z^{\star}=(y^{\star},p^{\star})$. Since $P_i$ is nonsingular,
$T_i(z^{\star},u)=0$ is equivalent to
$\nabla_y\Psi_i(y^{\star},p^{\star},u)=0$ and
$\nabla_p\Psi_i(y^{\star},p^{\star},u)=0$.
Strong convexity in $y$ and strong concavity in $p$ then imply
$\Psi_i(y^{\star},p,u) \le \Psi_i(y^{\star},p^{\star},u) \le \Psi_i(y,p^{\star},u), \qquad \forall\,y,p,$
so $(y^{\star},p^{\star})$ is the unique saddle point.
\end{proof}

\begin{proof}[Proof of Lemma~\ref{lem:outer-objective-smoothness}]
By Proposition~\ref{prop:response-regularity}, every exact state and hence $\Phi$ is
continuously differentiable.

Let $\mathcal J_i(x):=Dz_i^\star(x)$ denote the composite-state Jacobian,
distinct from $DS_i$ whose Lipschitz constant is $L_{J_i}$.
The chain rule gives $\mathcal J_0=I$ and
$\mathcal J_i=DS_i(z_{i-1}^\star)\mathcal J_{i-1}$, so
$\|\mathcal J_i(x)\|\le L_{z,i}$.

We next bound the variation of $\mathcal J_i$. Using
Proposition~\ref{prop:response-regularity},
\[
\begin{aligned}
\|\mathcal J_i(x)-\mathcal J_i(x')\|
\le{}&
L_{J_i}L_{z,i-1}^2\|x-x'\|
+
\kappa_i\|\mathcal J_{i-1}(x)-\mathcal J_{i-1}(x')\|.
\end{aligned}
\]
Since $\mathcal J_0=I$, induction gives
$\|\mathcal J_M(x)-\mathcal J_M(x')\| \le \left( \sum_{r=1}^M L_{J_r}L_{z,r-1}^2 \prod_{j=r+1}^M\kappa_j \right) \|x-x'\|.$

Now write $z:=z_M^{\star}(x)$, $z':=z_M^{\star}(x')$,
$J:=\mathcal J_M(x)$, and $J':=\mathcal J_M(x')$. By the chain rule,
$\nabla\Phi(x) = \nabla_1f(x,z) + J^\top\nabla_2f(x,z).$
The bound on $\nabla^2f$ and the Lipschitz continuity of the exact
terminal state imply
$\|\nabla_1f(x,z)-\nabla_1f(x',z')\|
\le L_f(1+L_{z,M})\|x-x'\|$ and
$\|\nabla_2f(x,z)-\nabla_2f(x',z')\|
\le L_f(1+L_{z,M})\|x-x'\|$.
Moreover, $\|J'\|\le L_{z,M}$ and
$\|\nabla_2f(x,z)\|\le C_f$. Hence
\[
\begin{aligned}
\|\nabla\Phi(x)-\nabla\Phi(x')\|
\le{}&
L_f(1+L_{z,M})^2\|x-x'\|
+
C_f\|J-J'\|.
\end{aligned}
\]
Substituting the preceding bound on $\|J-J'\|$ gives the stated
constant $L_\Phi$.
\end{proof}

\subsection{Perturbation bounds and regularity}
\label{app:perturbation-technical-details}

For $t\in[0,\delta_{\rm adm}]$, define
\[
H_{i,t}(x)
:=
\nabla_{11}^2\Psi_i
\bigl(z_{i,t}^{\star}(x),z_{i-1}^{\star}(x)\bigr),
\qquad
B_{i,t}(x)
:=
\nabla_{21}^2\Psi_i
\bigl(z_{i,t}^{\star}(x),z_{i-1}^{\star}(x)\bigr).
\]
Thus $H_{i,0}(x)=H_i^{\star}(x)$ and
$B_{i,0}(x)=B_i^{\star}(x)$.
\begin{lemma}[Differentiated perturbation equations]
\label{lem:differentiated-perturbation}
Under Assumptions~\ref{ass:implicit-layers} and \ref{ass:outer-function}, and under \eqref{eq:admissible-radius}, the paths
$t\mapsto z_{i,t}^{\star}(x)$ are continuously differentiable on
$[0,\delta_{\rm adm}]$. At the terminal level, let
$Q_{M,t}(x) := H_{M,t}(x) + t\nabla_{22}^2f\bigl(x,z_{M,t}^{\star}(x)\bigr).$
Then
\begin{equation}
\label{eq:terminal-path-derivative}
\dot z_{M,t}^{\star}(x)
=
-Q_{M,t}(x)^{-1}
\nabla_2f\bigl(x,z_{M,t}^{\star}(x)\bigr),
\end{equation}
where $\|Q_{M,t}(x)^{-1}\|\le\bar\mu_M^{-1}$. For
$i=M-1,\ldots,1$,
\begin{equation}
\label{eq:intermediate-path-derivative}
\dot z_{i,t}^{\star}(x)
=
-H_{i,t}(x)^{-1}
B_{i+1,t}(x)\dot z_{i+1,t}^{\star}(x),
\end{equation}
and $\|H_{i,t}(x)^{-1}\|\le\mu_i^{-1}$.
\end{lemma}

\begin{lemma}[Uniform regularity of the path derivatives]
\label{lem:path-derivative-x-lipschitz}
Under Assumptions~\ref{ass:implicit-layers} and \ref{ass:outer-function}, and under \eqref{eq:admissible-radius}, let
$C_{\mathrm{path},i}$ be the constants from
Lemma~\ref{lem:perturbation-displacement}. Then, for every
$i=1,\ldots,M$ and $t\in[0,\delta_{\rm adm}]$,
$\|\dot z_{i,t}^{\star}(x)\| \le C_{\mathrm{path},i}.$
Moreover, there exist finite constants $L_{d,i}$, independent of
$t$, such that
$\|\dot z_{i,t}^{\star}(x)-\dot z_{i,t}^{\star}(x')\| \le L_{d,i}\|x-x'\|.$
At the terminal level, one may choose
\[
L_{d,M}
:=
\frac{L_f(1+\widehat L_{z,M})}{\bar\mu_M}
+
\frac{C_f}{\bar\mu_M^2}
\left[
\rho_M(\widehat L_{z,M}+L_{z,M-1})
+
\delta_{\rm adm}\rho_f(1+\widehat L_{z,M})
\right],
\]
and, for $i=M-1,\ldots,1$,
\[
L_{d,i}
:=
\left[
\frac{L_{i+1}\rho_i(\widehat L_{z,i}+L_{z,i-1})}{\mu_i^2}
+
\frac{\rho_{i+1}(\widehat L_{z,i+1}+L_{z,i})}{\mu_i}
\right]C_{\mathrm{path},i+1}
+
\frac{L_{i+1}}{\mu_i}L_{d,i+1}.
\]
\end{lemma}

\begin{proof}[Proof of Lemma~\ref{lem:perturbation-displacement}]
At the terminal level, subtracting the exact stationarity condition
from \eqref{eq:terminal-pert-path-general} and using strong
monotonicity gives
$\mu_M\|z_{M,t}^{\star}-z_M^{\star}\|
\le t\|\nabla_2f(x,z_{M,t}^{\star})\|\le tC_f$.
Thus the sharper displacement-only bound
$\|z_{M,t}^{\star}-z_M^{\star}\|\le (C_f/\mu_M)t$ holds. Since
$\bar\mu_M\le\mu_M$, it also implies the stated common bound with
$C_{\mathrm{path},M}=C_f/\bar\mu_M$, which is convenient because the same
constant controls $\|\dot z_{M,t}^{\star}\|$ uniformly for
$t\in[0,\delta_{\rm adm}]$.

For $i<M$, subtracting the exact stationarity condition from
\eqref{eq:general-intermediate-pert} yields
$\mu_i\|z_{i,t}^{\star}-z_i^{\star}\| \le L_{i+1}\|z_{i+1,t}^{\star}-z_{i+1}^{\star}\|.$
The recursion follows by backward induction.
\end{proof}

\begin{proof}[Proof of Lemma~\ref{lem:averaged-linearization}]
The first identity follows by subtracting the exact stationarity
condition from the perturbed one, dividing by $t$, and applying the
fundamental theorem of calculus along the segment from $z_i^{\star}$ to
$z_{i,t}^{\star}$. The same argument applied to
$\nabla_2\Psi_i$ gives the second identity.

For every $h\in\mathbb R^{d_i}$, strong monotonicity implies
$\langle P_i\overline H_{i,t}h,h\rangle \ge \mu_i\|h\|^2.$
Since $P_i$ is orthogonal, this yields
$\|\overline H_{i,t}^{-1}\|\le\mu_i^{-1}$.
The bound on $\overline B_{i,t}$ follows directly from
Assumption~\ref{ass:implicit-layers}.
\end{proof}

\begin{proof}[Proof of Lemma~\ref{lem:perturbed-state-x-lipschitz}]
Fix $x,x'$ and $t\in[0,\delta_{\rm adm}]$. At the terminal level,
comparison of \eqref{eq:terminal-pert-path-general} at $x$ and $x'$
and $\bar\mu_M$-strong monotonicity give
$\|z_{M,t}^{\star}(x)-z_{M,t}^{\star}(x')\|\le\widehat L_{z,M}\|x-x'\|$ with
$\widehat L_{z,M}:=(L_ML_{z,M-1}+\delta_{\rm adm}L_f)/\bar\mu_M$.

For $i<M$, set $q_i(t,x):=t v_{i,t}^{\star}(x)$. By
\eqref{eq:general-vfd} and $\|\nabla^2\Psi_{i+1}\|\le L_{i+1}$,
$\|q_i(t,x)-q_i(t,x')\|$ is bounded by
$L_{i+1}(\widehat L_{z,i+1}+L_{z,i+1}+2L_{z,i})\|x-x'\|$ whenever the
level-$(i+1)$ bound holds. Comparing the two level-$i$ perturbed equations
and using $\mu_i$-strong monotonicity therefore gives the explicit admissible
recursion
\[
\widehat L_{z,i}
:=
\frac{1}{\mu_i}
\Bigl[
L_iL_{z,i-1}
+
L_{i+1}(\widehat L_{z,i+1}+L_{z,i+1}+2L_{z,i})
\Bigr],
\qquad i=M-1,\ldots,1.
\]
This constant is finite and independent of $t$. Backward induction completes
the proof.
\end{proof}

\begin{proof}[Proof of Lemma~\ref{lem:differentiated-perturbation}]
Theorem~\ref{thm:zero-radius-recovery} gives $C^1$ perturbation
paths on $[0,\delta_{\rm adm}]$.
Strong monotonicity of the signed terminal operator gives
$\|Q_{M,t}^{-1}\|\le(\mu_M-tL_f)^{-1}\le\bar\mu_M^{-1}$;
at nonterminal levels, it gives
$\|H_{i,t}^{-1}\|\le\mu_i^{-1}$.
Differentiating \eqref{eq:terminal-pert-path-general} yields
\eqref{eq:terminal-path-derivative}.
For $i<M$, set $q_i(t,x):=t v_{i,t}^\star(x)$ for $t>0$
and $q_i(0,x):=0$.
By \eqref{eq:general-vfd},
$\partial_tq_i=B_{i+1,t}\dot z_{i+1,t}^\star$,
including at $t=0$ by continuity.
Differentiating \eqref{eq:general-intermediate-pert} then gives
\eqref{eq:intermediate-path-derivative}.
\end{proof}

\begin{proof}[Proof of Lemma~\ref{lem:path-derivative-x-lipschitz}]
At the terminal level, \eqref{eq:terminal-path-derivative} gives
$\dot z_{M,t}^{\star}=-Q_{M,t}^{-1}\nabla_2f(x,z_{M,t}^{\star})$, hence
$\|\dot z_{M,t}^{\star}\|\le C_f/\bar\mu_M=C_{\mathrm{path},M}$. For two outer
points, abbreviate $Q:=Q_{M,t}(x)$, $Q':=Q_{M,t}(x')$,
$g:=\nabla_2f(x,z_{M,t}^{\star}(x))$, and
$g':=\nabla_2f(x',z_{M,t}^{\star}(x'))$. Then
\[
\begin{aligned}
\|\dot z_{M,t}^{\star}(x)-\dot z_{M,t}^{\star}(x')\|
&\le
\|Q^{-1}\|\,\|g-g'\|
+
\|Q^{-1}-Q'^{-1}\|\,\|g'\|\\
&\le
\frac{L_f(1+\widehat L_{z,M})}{\bar\mu_M}\|x-x'\|
+
\frac{C_f}{\bar\mu_M^2}\|Q-Q'\|.
\end{aligned}
\]
The Hessian Lipschitz assumptions and
Lemma~\ref{lem:perturbed-state-x-lipschitz} give
\[
\|Q-Q'\|
\le
\left[
\rho_M(\widehat L_{z,M}+L_{z,M-1})
+
\delta_{\rm adm}\rho_f(1+\widehat L_{z,M})
\right]\|x-x'\|.
\]
Combining the last two displays yields the stated constant $L_{d,M}$.

For $i<M$, write
$R_{i,t}(x):=H_{i,t}(x)^{-1}B_{i+1,t}(x)$, so that
$\dot z_{i,t}^{\star}(x)=-R_{i,t}(x)\dot z_{i+1,t}^{\star}(x)$. The bounds
$\|R_{i,t}(x)\|\le L_{i+1}/\mu_i$ and the induction hypothesis give
\[
\|\dot z_{i,t}^{\star}(x)\|
\le
\frac{L_{i+1}}{\mu_i}C_{\mathrm{path},i+1}
=
C_{\mathrm{path},i}.
\]
Moreover, the inverse-difference identity and Hessian Lipschitz continuity
imply
\[
\|R_{i,t}(x)-R_{i,t}(x')\|
\le
\Gamma_i\|x-x'\|,
\]
where
\[
\Gamma_i
:=
\frac{L_{i+1}\rho_i(\widehat L_{z,i}+L_{z,i-1})}{\mu_i^2}
+
\frac{\rho_{i+1}(\widehat L_{z,i+1}+L_{z,i})}{\mu_i}.
\]
Therefore,
\[
\begin{aligned}
\|\dot z_{i,t}^{\star}(x)-\dot z_{i,t}^{\star}(x')\|
&\le
\|R_{i,t}(x)-R_{i,t}(x')\|\,
\|\dot z_{i+1,t}^{\star}(x)\|+
\|R_{i,t}(x')\|\,
\|\dot z_{i+1,t}^{\star}(x)-\dot z_{i+1,t}^{\star}(x')\|\\
&\le
\left[
\Gamma_i C_{\mathrm{path},i+1}
+
\frac{L_{i+1}}{\mu_i}L_{d,i+1}
\right]\|x-x'\|,
\end{aligned}
\]
which is precisely the stated recursion for $L_{d,i}$. This completes the
backward induction.
\end{proof}

\begin{proof}[Proof of Corollary~\ref{cor:finite-direction-regularity}]
Apply Lemma~\ref{lem:path-derivative-x-lipschitz} to the integral
representation of $d_{i,t}^\star$ given in
Section~\ref{subsec:finite-radius-regularity}.
\end{proof}

\begin{proof}[Proof of Corollary~\ref{cor:finite-adjoint-regularity}]
For $i=M$, the result follows directly from
$v_{M,t}^{\star}(x)=\nabla_2f(x,z_{M,t}^{\star}(x))$ and
Lemma~\ref{lem:perturbed-state-x-lipschitz}. For $i<M$, subtract the two
integral representations \eqref{eq:adjoint-path-integral}. The Hessian
Lipschitz bound gives
$\|B_{i+1,s}(x)-B_{i+1,s}(x')\|\le
\rho_{i+1}(\widehat L_{z,i+1}+L_{z,i})\|x-x'\|$, while
Lemma~\ref{lem:path-derivative-x-lipschitz} gives the uniform bounds on
$\dot z_{i+1,s}^{\star}$ and its variation. Integrating yields precisely the
stated constant $L_{v,i}$.
\end{proof}





\subsection{Local inexact-solve tracking}

\begin{proof}[Proof of Lemma~\ref{lem:primal-tracking}]
For $i=1$, the local unperturbed target is
$S_1(x^k)=z_1^{\star}(x^k)$.  Hence contraction of the $N_1$ inner steps gives
\[
\mathcal E_1^{k+1}
\le
\chi_1\|z_1^k-z_1^{\star}(x^k)\|^2
\le
2\chi_1\mathcal E_1^k
+2\chi_1L_{z,1}^2\|s^k\|^2.
\]

Let $i\ge2$ and define the updated-input local target
$\widehat z_i^{k,*}:=S_i(z_{i-1}^{k+1})=S_i(\bar z_{i-1}^k)$.
Contraction gives
$\|z_i^{k+1}-\widehat z_i^{k,*}\|^2
\le\chi_i\|z_i^k-\widehat z_i^{k,*}\|^2$,
while Proposition~\ref{prop:response-regularity} yields
$\|\widehat z_i^{k,*}-z_i^{\star}(x^k)\|^2
\le\kappa_i^2\mathcal E_{i-1}^{k+1}$.
Moreover,
\[
\begin{aligned}
\|z_i^k-\widehat z_i^{k,*}\|^2
&\le
2\mathcal E_i^k
+2\|z_i^{\star}(x^{k-1})-\widehat z_i^{k,*}\|^2
\\
&\le
2\mathcal E_i^k
+4L_{z,i}^2\|s^k\|^2
+4\kappa_i^2\mathcal E_{i-1}^{k+1}.
\end{aligned}
\]
Combining the last three displays with
$\|a+b\|^2\le2\|a\|^2+2\|b\|^2$ gives
\[
\mathcal E_i^{k+1}
\le
4\chi_i\mathcal E_i^k
+2\kappa_i^2(1+4\chi_i)\mathcal E_{i-1}^{k+1}
+8\chi_iL_{z,i}^2\|s^k\|^2,
\]
which proves the claim.
\end{proof}

\begin{proof}[Proof of Lemma~\ref{lem:local-direction-reference}]
For $i<M$, set
$u:=\bar z_{i-1}^k$, $u^{\star}:=z_{i-1}^{\star}(x^k)$,
$\widehat d:=\widehat d_i^{k,*}$, and $d^{\star}:=d_i^{\star}(x^k)$.  If
$\widehat H$ and $H^{\star}$ denote the averaged Hessians along the local and exact
perturbation segments, then
$\widehat H\widehat d+v_{i,\delta}^k=0$ and
$H^{\star}d^{\star}+v_{i,\delta}^{\star}(x^k)=0$.
Using $\|\widehat H^{-1}\|\le\mu_i^{-1}$,
$\|d^{\star}\|\le C_{\mathrm{path},i}$, the Lipschitz estimate
$\|S_i(u)-S_i(u^{\star})\|\le\kappa_i\|u-u^{\star}\|$, and Hessian Lipschitz continuity
first gives
\[
\begin{aligned}
\mu_i\|\widehat d-d^{\star}\|
\le{}&
\|v_{i,\delta}^k-v_{i,\delta}^{\star}(x^k)\|
+\rho_iC_{\mathrm{path},i}(1+\kappa_i)\|u-u^{\star}\|
+\delta\rho_iC_{\mathrm{path},i}\|\widehat d-d^{\star}\|.
\end{aligned}
\]
The nonterminal condition in \eqref{eq:tracking-radius-condition} gives
$\delta\rho_iC_{\mathrm{path},i}\le\mu_i/2$, so the last term is absorbed and
we obtain
\[
\|\widehat d-d^{\star}\|
\le
\frac{2}{\mu_i}
\|v_{i,\delta}^k-v_{i,\delta}^{\star}(x^k)\|
+
\frac{2\rho_iC_{\mathrm{path},i}(1+\kappa_i)}{\mu_i}
\|u-u^{\star}\|.
\]
Squaring proves \eqref{eq:forward-local-direction-reference}.  At the
terminal level, the same comparison gives
\[
\begin{aligned}
\mu_M\|\widehat d_M^{k,*}-d_M^{\star}(x^k)\|
\le{}&
\left[\rho_MC_{\mathrm{path},M}(1+\kappa_M)+L_f\kappa_M\right]
\|u-u^{\star}\|
\\
&+\delta(\rho_MC_{\mathrm{path},M}+L_f)
\|\widehat d_M^{k,*}-d_M^{\star}(x^k)\|.
\end{aligned}
\]
The terminal condition in \eqref{eq:tracking-radius-condition} absorbs the
final term and gives the stated $C_{d,u,M}$. For $i\ge2$,
$\|u-u^{\star}\|^2=\mathcal E_{i-1}^{k+1}$; when $i=1$, the input is exact and the
corresponding input-mismatch term is zero. This proves all claims.
\end{proof}

\begin{proof}[Proof of Lemma~\ref{lem:delta-uniformity}]
Fix $\delta_{\rm adm}$ satisfying
\eqref{eq:admissible-radius} and \eqref{eq:tracking-radius-condition}. The exact
perturbation bounds are already uniform on $[0,\delta_{\rm adm}]$ by
Lemmas~\ref{lem:perturbation-displacement},
\ref{lem:perturbed-state-x-lipschitz}, and
\ref{lem:path-derivative-x-lipschitz}, Corollary~\ref{cor:finite-adjoint-regularity},
and Theorem~\ref{thm:finite-radius-bias}.  At the terminal layer,
$\widetilde\mu_M\ge\bar\mu_M$ and
$\widetilde L_M\le\bar L_M$.  Hence the stated worst-case signed-gradient
stepsize gives
$\widetilde\vartheta_M
\le1-2\eta_M\bar\mu_M+\eta_M^2\bar L_M^2<1$,
and the minimization choice gives the uniform squared-distance factor
$(1-\bar\mu_M/\bar L_M)^2<1$.  The nonterminal factors are independent of
$\delta$.  This proves the uniform contraction bounds and therefore the
existence of $\delta$-independent budgets $N_i$.

The explicit coefficients in
Appendix~\ref{app:auxiliary-tracking-details} depend only on these uniform
regularity constants, fixed inner stepsizes, contraction factors, and
nonnegative powers of $\delta$.  Moreover,
$\theta_i-\overline\vartheta_i=1-\theta_i=(1-\overline\vartheta_i)/2$, so the
uniform contraction gap keeps every denominator bounded away from zero.
Hence all tracking coefficients are uniformly bounded.  Increasing the fixed
budgets if necessary gives uniform margins
$\sup_\delta\tau_i<1$ and $\sup_\delta\tau_{d,i}<1$.
\end{proof}

\begin{proof}[Proof of Lemma~\ref{lem:Gamma-uniformity}]
Fix the $\delta$-independent inner stepsizes and budgets from
Lemma~\ref{lem:delta-uniformity}.  All coefficients below are nonnegative.
For every coefficient
$c(\delta)\in\{b_i^E,b_i^Z,\lambda_i,\gamma_i,
\beta_{d,i},\lambda_{d,i},\psi_{d,i},\psi'_{d,i},\gamma_{d,i}\}$,
write $c^{\rm adm}:=\sup_{0<\delta\le\delta_{\rm adm}}c(\delta)$.
Lemma~\ref{lem:delta-uniformity}, together with the explicit coefficient
formulas in Appendix~\ref{app:auxiliary-tracking-details}, implies that all
these suprema are finite.  The uniform budget conditions imposed in Section~\ref{sec:convergence}
give constants $\bar\tau_i,\bar\tau_{d,i}<1$ such that
\begin{equation}
\label{eq:uniform-tau-majorants}
\tau_i(\delta)\le\bar\tau_i<1,
\qquad
\tau_{d,i}(\delta)\le\bar\tau_{d,i}<1
\end{equation}
for every admissible $\delta$.

We now construct upper bounds independent of $\delta$ in the same acyclic order as
the Lyapunov weights.  Define
\begin{equation}
\label{eq:uniform-U-majorants}
U_1
:=
\frac{(b_1^Z)^{\rm adm}}{1-\bar\tau_{d,1}},
\qquad
U_i
:=
\frac{(b_i^Z)^{\rm adm}+(\beta_{d,i-1})^{\rm adm}U_{i-1}}
     {1-\bar\tau_{d,i}},
\quad i=2,\ldots,M.
\end{equation}
Because the coefficients are nonnegative, induction in increasing $i$ gives
\begin{equation}
\label{eq:u-uniform-majorization}
\widetilde u_i(\delta)\le U_i,
\qquad
0<\delta\le\delta_{\rm adm}.
\end{equation}
Indeed, the claim for $i=1$ follows directly from
\eqref{eq:auxiliary-lyapunov-weights-first} and
\eqref{eq:uniform-tau-majorants}; the recursion
\eqref{eq:auxiliary-lyapunov-weights} and the induction hypothesis give the
remaining levels.

For the boundary indices, use
$U_0=U_{M+1}=(\lambda_{d,0})^{\rm adm}
=(\psi_{d,M+1})^{\rm adm}:=0$.
Next define the uniform primal coefficients
\begin{equation}
\label{eq:uniform-BE-majorants}
B_i^E
:=
(b_i^E)^{\rm adm}
+(\lambda_{d,i})^{\rm adm}U_i
+\mathbf1_{\{i\ge2\}}(\lambda_{d,i-1})^{\rm adm}U_{i-1}
+\mathbf1_{\{i\le M-1\}}(\psi_{d,i+1})^{\rm adm}U_{i+1}.
\end{equation}
Then \eqref{eq:effective-primal-coefficient} and
\eqref{eq:u-uniform-majorization} imply
$\widehat b_i^E(\delta)\le B_i^E$ uniformly.  Set
$W_{M+1}:=0$ and $(\lambda_{M+1})^{\rm adm}:=0$, and recursively for
$i=M,M-1,\ldots,1$ define
\begin{equation}
\label{eq:uniform-W-majorants}
W_i
:=
\frac{
 B_i^E
 +(\lambda_{i+1})^{\rm adm}W_{i+1}
 +(\psi'_{d,i})^{\rm adm}U_i
}{1-\bar\tau_i}.
\end{equation}
A backward induction using
\eqref{eq:forward-primal-load}--\eqref{eq:primal-lyapunov-weights} yields
\begin{equation}
\label{eq:omega-uniform-majorization}
\widetilde\omega_i(\delta)\le W_i,
\qquad
0<\delta\le\delta_{\rm adm}.
\end{equation}

Finally, define
\begin{equation}
\label{eq:Gamma-explicit-uniform-majorant}
\bar\Gamma_{\rm adm}
:=
\sum_{i=1}^M
\left[
 W_i(\gamma_i)^{\rm adm}
 +U_i(\gamma_{d,i})^{\rm adm}
\right].
\end{equation}
Every quantity on the right-hand side is finite, and
\eqref{eq:Gamma-definition}, \eqref{eq:u-uniform-majorization}, and
\eqref{eq:omega-uniform-majorization} imply
\[
0\le\Gamma(\delta)\le\bar\Gamma_{\rm adm}
\qquad
\text{for every }0<\delta\le\delta_{\rm adm}.
\]
This proves \eqref{eq:Gamma-uniform-majorant}.
\end{proof}



\section{Proof of the coupled tracking recursion}
\label{app:auxiliary-tracking-details}

This appendix gives one explicit admissible coefficient choice for
Theorem~\ref{thm:auxiliary-tracking}, under
\eqref{eq:admissible-radius} and \eqref{eq:tracking-radius-condition} and the
contraction conditions of Lemma~\ref{lem:inner-contraction}. Define
\[
C_{\mathrm{dec},i}
:=
\frac{32\eta_i^2\theta_i}
{(\theta_i-\overline\vartheta_i)^2}
\bigl(\rho_i^2C_{\mathrm{path},i}^2+L_f^2\mathbf1_{\{i=M\}}\bigr),
\qquad
C_{\mathrm{ref},i}
:=
\frac{16\eta_i^2\theta_iL_i^2}
{(\theta_i-\overline\vartheta_i)(1-\theta_i)}.
\]
Recall that, for $i<M$,
$\mathcal R_{i-1}^k=C_{d,u,i}\mathcal E_{i-1}^{k+1}
+C_{d,v,i}\|v_{i,\delta}^k-v_{i,\delta}^{\star}(x^k)\|^2$,
whereas $\mathcal R_{M-1}^k
=C_{d,u,M}\mathcal E_{M-1}^{k+1}$ at the terminal level.
An admissible coefficient choice is
\begin{equation}
\label{eq:forward-aux-coefficients-1}
\tau_{d,i}
:=
4\chi_{d,i}(1+C_{\mathrm{dec},i}\delta^2),
\qquad
\psi'_{d,i}
:=
8C_{\mathrm{dec},i}\chi_{d,i},
\end{equation}
\begin{equation}
\label{eq:forward-aux-coefficients-2}
\beta_{d,i}
:=
2(\tau_{d,i}+C_{\mathrm{ref},i}+1)
C_{d,v,i}C_{vZ,i+1},
\qquad i=1,\ldots,M-1,
\end{equation}
\begin{equation}
\label{eq:forward-aux-coefficients-3}
\lambda_{d,i}
:=
2(\tau_{d,i}+C_{\mathrm{ref},i}+1)
C_{d,v,i}C_{vE,i+1},
\qquad i<M,
\end{equation}
with $\beta_{d,M}=\lambda_{d,M}:=0$, and
\begin{equation}
\label{eq:forward-aux-coefficients-4}
\psi_{d,i}
:=
2\mathbf1_{\{i\ge2\}}
(\tau_{d,i}+C_{\mathrm{ref},i}+1)C_{d,u,i}
+8\mathbf1_{\{i\ge2\}}\kappa_i^2C_{\mathrm{dec},i}\chi_{d,i},
\end{equation}
\begin{equation}
\label{eq:forward-aux-coefficients-5}
\gamma_{d,i}
:=
2\tau_{d,i}L_{d,i}^2
+8C_{\mathrm{dec},i}\chi_{d,i}L_{z,i}^2.
\end{equation}

\begin{proof}[Proof of Theorem~\ref{thm:auxiliary-tracking}]
Fix $i$ and $k$.  Although Algorithm~\ref{alg:remid} executes the
unperturbed and perturbed loops in different sweeps, define their paired inner
sequences using the common input $\bar z_{i-1}^k$ and set
$d_i^{k,\ell}:=[z_{i,\delta}^{k,\ell}-z_i^{k,\ell}]/\delta$.
Let $\overline H_i^{k,\ell}$ be the Hessian of $\Psi_i$ averaged along the
segment joining the two level-$i$ inner iterates.  With
\[
r_i^{k,\ell}
:=
\begin{cases}
v_{i,\delta}^k,&i<M,\\
\nabla_2f(x^k,z_{M,\delta}^{k,\ell}),&i=M,
\end{cases}
\]
subtracting the unperturbed and perturbed updates gives
\[
d_i^{k,\ell+1}
=
d_i^{k,\ell}
-\eta_iP_i\overline H_i^{k,\ell}d_i^{k,\ell}
-\eta_iP_ir_i^{k,\ell}.
\]
Define analogously $\widehat H_i^{k,*}$ on the local exact
segment and set
\[
\widehat r_i^{k,*}:=
\begin{cases}
v_{i,\delta}^k,&i<M,\\
\nabla_2f(x^k,\widehat z_{M,\delta}^{k,*}),&i=M.
\end{cases}
\]
The local stationarity equation is
$\widehat H_i^{k,*}\widehat d_i^{k,*}+\widehat r_i^{k,*}=0$.
Thus, with $e_i^{k,\ell}:=d_i^{k,\ell}-\widehat d_i^{k,*}$,
\begin{equation}
\label{eq:aux-inner-error-dynamics}
e_i^{k,\ell+1}
=
(I-\eta_iP_i\overline H_i^{k,\ell})e_i^{k,\ell}
-\eta_iP_i\Xi_i^{k,\ell},
\end{equation}
where
\[
\Xi_i^{k,\ell}
:=
(\overline H_i^{k,\ell}-\widehat H_i^{k,*})\widehat d_i^{k,*}
+(r_i^{k,\ell}-\widehat r_i^{k,*}).
\]
Lemma~\ref{lem:inner-contraction} and Young's inequality give
\begin{equation}
\label{eq:aux-inner-one-step}
\|e_i^{k,\ell+1}\|^2
\le
\theta_i\|e_i^{k,\ell}\|^2
+
\frac{\eta_i^2\theta_i}
{\theta_i-\overline\vartheta_i}
\|\Xi_i^{k,\ell}\|^2.
\end{equation}

To avoid a circular bound on $\widehat d_i^{k,*}$, decompose
$(\overline H_i^{k,\ell}-\widehat H_i^{k,*})\widehat d_i^{k,*}$ into the
same operator applied to $d_i^{\star}(x^k)$ and to
$\widehat d_i^{k,*}-d_i^{\star}(x^k)$.  The first term uses
$\|d_i^{\star}(x^k)\|\le C_{\mathrm{path},i}$, while the second is controlled by
$\mathcal R_{i-1}^k$.  The unperturbed and perturbed trajectories contract
toward their respective local targets with factor
$\overline\vartheta_i^\ell$; in particular,
\[
\|z_{i,\delta}^k-\widehat z_{i,\delta}^{k,*}\|^2
\le
2\|z_i^k-\widehat z_i^{k,*}\|^2
+2\delta^2\|d_i^k-\widehat d_i^{k,*}\|^2.
\]
Hessian Lipschitz continuity, together with the terminal Lipschitz bound for
$\nabla_2f$, controls the residual as follows. Write
$a_\ell:=\|z_i^{k,\ell}-\widehat z_i^{k,*}\|^2$ and
$b_\ell:=\|z_{i,\delta}^{k,\ell}-\widehat z_{i,\delta}^{k,*}\|^2$.
Then
$\|\overline H_i^{k,\ell}-\widehat H_i^{k,*}\|
\le\rho_i(\sqrt{a_\ell}+\sqrt{b_\ell})$ and also
$\|\overline H_i^{k,\ell}-\widehat H_i^{k,*}\|\le2L_i$.
Applying the first bound to $d_i^\star(x^k)$ and the second to
$\widehat d_i^{k,*}-d_i^\star(x^k)$ gives
\[
\|\Xi_i^{k,\ell}\|^2
\le
8\bigl(\rho_i^2C_{\mathrm{path},i}^2
+L_f^2\mathbf1_{\{i=M\}}\bigr)(a_\ell+b_\ell)
+16L_i^2\mathcal R_{i-1}^k.
\]
Since
$a_\ell+b_\ell\le\overline\vartheta_i^\ell(a_0+b_0)$ and
$b_0\le2a_0+2\delta^2\|d_i^k-\widehat d_i^{k,*}\|^2$,
we obtain
\begin{equation}
\label{eq:aux-residual-bound}
\begin{aligned}
\|\Xi_i^{k,\ell}\|^2
\le{}&
32(\rho_i^2C_{\mathrm{path},i}^2+L_f^2\mathbf1_{\{i=M\}})
\overline\vartheta_i^\ell
\bigl(
\|z_i^k-\widehat z_i^{k,*}\|^2
+\delta^2\|d_i^k-\widehat d_i^{k,*}\|^2
\bigr)
\\
&+16L_i^2\mathcal R_{i-1}^k.
\end{aligned}
\end{equation}
Unrolling \eqref{eq:aux-inner-one-step}, using
\[
\sum_{\ell=0}^{N_i-1}
\theta_i^{N_i-1-\ell}\overline\vartheta_i^\ell
\le
\frac{\chi_{d,i}}{\theta_i-\overline\vartheta_i},
\qquad
\sum_{\ell=0}^{N_i-1}\theta_i^{N_i-1-\ell}
\le
\frac{1}{1-\theta_i},
\]
yields
\begin{equation}
\label{eq:aux-local-inner-bound}
\begin{aligned}
\|d_i^{k+1}-\widehat d_i^{k,*}\|^2
\le{}&
\chi_{d,i}(1+C_{\mathrm{dec},i}\delta^2)
\|d_i^k-\widehat d_i^{k,*}\|^2
+
C_{\mathrm{dec},i}\chi_{d,i}
\|z_i^k-\widehat z_i^{k,*}\|^2
+C_{\mathrm{ref},i}\mathcal R_{i-1}^k.
\end{aligned}
\end{equation}

The change in the exact scaled response difference and the local-reference error satisfy
\[
\|d_i^k-\widehat d_i^{k,*}\|^2
\le
2\mathcal Z_i^k
+4\mathcal R_{i-1}^k
+4L_{d,i}^2\|s^k\|^2.
\]
For the unperturbed warm start, the updated-input target gives the simpler bound
\begin{equation}
\label{eq:aux-nominal-warmstart}
\|z_i^k-\widehat z_i^{k,*}\|^2
\le
4\mathcal E_i^k
+4\mathbf1_{\{i\ge2\}}\kappa_i^2\mathcal E_{i-1}^{k+1}
+4L_{z,i}^2\|s^k\|^2.
\end{equation}
Indeed, insert $z_i^{\star}(x^{k-1})$ and $z_i^{\star}(x^k)$ between the two terms and
use
$\|\widehat z_i^{k,*}-z_i^{\star}(x^k)\|
\le\kappa_i\sqrt{\mathcal E_{i-1}^{k+1}}$.

Both temporal-drift terms retain an inner contraction factor.
In contrast, $\mathcal R_{i-1}^k$ contains only current-point spatial
errors. Thus no uncontracted input-drift term enters $\gamma_{d,i}$,
which is essential for the condition-number bookkeeping below.

Combining \eqref{eq:aux-local-inner-bound} with
$\mathcal Z_i^{k+1}
\le2\|d_i^{k+1}-\widehat d_i^{k,*}\|^2+2\mathcal R_{i-1}^k$
gives
\begin{equation}
\label{eq:aux-pre-residual-substitution}
\begin{aligned}
\mathcal Z_i^{k+1}
\le{}&
\tau_{d,i}\mathcal Z_i^k
+2(\tau_{d,i}+C_{\mathrm{ref},i}+1)\mathcal R_{i-1}^k
+
8C_{\mathrm{dec},i}\chi_{d,i}\mathcal E_i^k
+8\mathbf1_{\{i\ge2\}}\kappa_i^2
C_{\mathrm{dec},i}\chi_{d,i}\mathcal E_{i-1}^{k+1}
\\
&+
\left(
2\tau_{d,i}L_{d,i}^2
+8C_{\mathrm{dec},i}\chi_{d,i}L_{z,i}^2
\right)\|s^k\|^2.
\end{aligned}
\end{equation}
For $i<M$, Lemma~\ref{lem:finite-adjoint-propagation} gives
\[
\begin{aligned}
\mathcal R_{i-1}^k
\le{}&
C_{d,u,i}\mathcal E_{i-1}^{k+1}
+C_{d,v,i}C_{vZ,i+1}\mathcal Z_{i+1}^{k+1}
+
C_{d,v,i}C_{vE,i+1}
(\mathcal E_{i+1}^{k+1}+\mathcal E_i^{k+1}),
\end{aligned}
\]
while at $i=M$ only the first term remains because $C_{d,v,M}=0$.
Substituting this estimate into
\eqref{eq:aux-pre-residual-substitution} gives
\eqref{eq:auxiliary-tracking-recursion} with
\eqref{eq:forward-aux-coefficients-1}--\eqref{eq:forward-aux-coefficients-5}.
\end{proof}

\section{Condition-number bounds for minimization chains}
\label{app:lyapunov-details}

This appendix verifies the conditioning orders used in
Theorem~\ref{thm:minimization-complexity}.  We work at fixed depth $M\ge2$ under
\eqref{eq:single-kappa-regularity}.  Constants may depend on $M$ but not on
$\kappa$.  Shrink the admissible parameter independently of the target accuracy so
that
\[
\delta_{\rm adm}
\le c_{\rm adm}\kappa^{-(2M+1)},
\]
for a sufficiently small constant $c_{\rm adm}>0$, and choose
\[
\chi_{d,i}\le c\kappa^{-(4M+4)},
\qquad i=1,\ldots,M,
\]
for a sufficiently small constant $c>0$.  Since
$\vartheta_i\le\overline\vartheta_i<\theta_i$, this also gives
$\chi_i\le\chi_{d,i}$.  All bounds below are uniform for
$0<\delta\le\delta_{\rm adm}$.

\begin{proposition}[Condition-number bounds for the forward response chain]
\label{prop:minimization-kappa-bookkeeping}
Suppose $M\ge2$.  Under the preceding choices,
\begin{equation}
\label{eq:appendix-main-orders}
C_{\mathrm{bias},0}=O(\kappa^{2M+1}),
\qquad
L_\Phi=O(\kappa^{2M+1}),
\qquad
\bar\Gamma_{\rm adm}=O(\kappa^{2M-2}),
\end{equation}
where $\bar\Gamma_{\rm adm}$ is the uniform majorant from
Lemma~\ref{lem:Gamma-uniformity}.  In particular, the bound on the Lyapunov
coefficient is uniform over the whole admissible parameter interval.
\end{proposition}

\begin{proof}
The exact-response and exact-perturbation quantities are independent of the
inexact-solve scheduling.  The sensitivity recursions of
Section~\ref{sec:perturbation} give
\begin{equation}
\label{eq:appendix-sensitivity-orders}
\begin{aligned}
L_{z,i}=O(\kappa^i),
\qquad &
C_{\mathrm{path},i}=O(\kappa^{M-i+1}),
\qquad &
\widehat L_{z,i}=O(\kappa^{2M-i}),
\\
L_{d,i}=O(\kappa^{3M-2i+2}),
\qquad &
L_{v,i}=O(\kappa^{3M-2i}).
\end{aligned}
\end{equation}
For completeness, these orders follow by direct induction.  First,
$L_{z,i}=O(\kappa^i)$.  The terminal bounds
$C_{\mathrm{path},M}=O(\kappa)$ and
$\widehat L_{z,M}=O(\kappa^M)$, followed by their respective backward
recursions, give
$C_{\mathrm{path},i}=O(\kappa^{M-i+1})$ and
$\widehat L_{z,i}=O(\kappa^{2M-i})$.  Substitution into the terminal formula
for $L_{d,M}$ gives $L_{d,M}=O(\kappa^{M+2})$; the recursion for $L_{d,i}$
then adds two powers at each preceding level and yields
$L_{d,i}=O(\kappa^{3M-2i+2})$.  Finally, the formula for $L_{v,M}$ and its
backward recursion give $L_{v,i}=O(\kappa^{3M-2i})$.

The perturbation bias recursion gives
$C_{\mathrm{bias},0}=O(\kappa^{2M+1})$, while
Lemma~\ref{lem:outer-objective-smoothness} gives
$L_\Phi=O(\kappa^{2M+1})$.

The local finite-difference constants also retain the orders
\begin{equation}
\label{eq:appendix-local-orders}
C_{vZ,i}=O(1),
\qquad
C_{vE,i}=O(\kappa^{2(M-i+1)}),
\qquad
C_{d,v,i}=O(\kappa^2),
\qquad
C_{d,u,i}=O(\kappa^{2(M-i)+6}).
\end{equation}
We make the stepsize dependence in the auxiliary constants explicit. Since
$\theta_i-\overline\vartheta_i=1-\theta_i=(1-\overline\vartheta_i)/2$,
the definitions in Appendix~\ref{app:auxiliary-tracking-details} involve the
common denominator $(1-\overline\vartheta_i)^2$. For $i<M$,
$\eta_i=1/L_i$ and
$1-\overline\vartheta_i=1-(1-\mu_i/L_i)^2\ge \mu_i/L_i$, whence
\[
\frac{\eta_i^2}{(1-\overline\vartheta_i)^2}
=O(\mu_i^{-2})=O(\kappa^2),
\qquad
\frac{\eta_i^2L_i^2}{(1-\overline\vartheta_i)^2}
=O(\kappa_i^2)=O(\kappa^2).
\]
At $i=M$ the same estimates hold with
$(\widetilde\mu_M,\widetilde L_M)$, because
$\widetilde\mu_M\ge\mu_M/2$ under the standing parameter restriction. Combining
these cancellations with $C_{\mathrm{path},i}=O(\kappa^{M-i+1})$ gives
\begin{equation}
\label{eq:appendix-dec-ref-orders}
C_{\mathrm{dec},i}=O(\kappa^{2(M-i)+4}),
\qquad
C_{\mathrm{ref},i}=O(\kappa^2).
\end{equation}
Moreover, uniformly for $0<\delta\le\delta_{\rm adm}$,
\begin{equation}
\label{eq:appendix-dec-radius-order}
C_{\mathrm{dec},i}\delta^2
\le
C_{\mathrm{dec},i}\delta_{\rm adm}^2
=O(\kappa^{-2M-2i+2})=O(1).
\end{equation}
Thus $1+C_{\mathrm{dec},i}\delta^2$ introduces no hidden polynomial loss in
$\tau_{d,i}$ anywhere on the admissible interval.  With
$c_{\rm adm}$ chosen sufficiently small, the same uniform parameter choice also
satisfies the absorption conditions in
Lemma~\ref{lem:local-direction-reference}.

We next record the effect of the updated-input scheduling on the tracking
coefficients.  From Lemma~\ref{lem:primal-tracking} and
$\chi_i\le c\kappa^{-(4M+4)}$,
\begin{equation}
\label{eq:appendix-forward-primal-orders}
\tau_i=O(\kappa^{-(4M+4)}),
\qquad
\lambda_i=O(\kappa^2),
\qquad
\gamma_i=O(\kappa^{-4M-4+2i}).
\end{equation}
In particular, the temporal-drift coefficient of the primal recursion is now
multiplied by the inner contraction.

Using
\eqref{eq:forward-aux-coefficients-1}--\eqref{eq:forward-aux-coefficients-5}
and \eqref{eq:appendix-local-orders}--\eqref{eq:appendix-dec-ref-orders},
we obtain
\begin{equation}
\label{eq:appendix-forward-auxiliary-orders}
\begin{aligned}
\tau_{d,i}&=O(\kappa^{-(4M+4)}),
&\beta_{d,i}&=O(\kappa^4),
\\
\lambda_{d,i}&=O(\kappa^{2(M-i)+4}),
&\psi'_{d,i}&=O(\kappa^{-2M-2i}),
\\
\psi_{d,i}&=O(\kappa^{2(M-i)+8}),
&\gamma_{d,i}&=O(\kappa^{2M-4i}).
\end{aligned}
\end{equation}
The last estimate follows directly from
$\tau_{d,i}L_{d,i}^2
=O(\kappa^{2M-4i})$ and
$C_{\mathrm{dec},i}\chi_{d,i}L_{z,i}^2=O(\kappa^{-2M})$.
The key consequence of the schedule using updated inputs is that no uncompensated
input-drift term of the form $\psi_{d,i}L_{z,i-1}^2$ or
$\lambda_{d,i}L_{z,i}^2$ enters $\gamma_{d,i}$.

The auxiliary-weight recursion is unchanged in direction and gives
\begin{equation}
\label{eq:appendix-u-orders}
\widetilde u_i=O(\kappa^{4(i-1)}).
\end{equation}
For the coefficient of the updated primal error \eqref{eq:effective-primal-coefficient},
\begin{equation}
\label{eq:appendix-bhat-orders}
\widehat b_i^E
=
\begin{cases}
O(\kappa^{2M+2i+6}),&i<M,\\
O(\kappa^{4M-2}),&i=M.
\end{cases}
\end{equation}
Indeed, for $i<M$ the largest direct contribution is
$\psi_{d,i+1}\widetilde u_{i+1}$, while at $i=M$ the leading cross-level
term is $\lambda_{d,M-1}\widetilde u_{M-1}$.
The backward primal-weight recursion
\eqref{eq:forward-primal-load}--\eqref{eq:primal-lyapunov-weights} then yields
\begin{equation}
\label{eq:appendix-weight-orders}
\widetilde\omega_M=O(\kappa^{4M-2}),
\qquad
\widetilde\omega_i=O(\kappa^{6M-2i+2}),
\quad i<M.
\end{equation}
For $i=M-1$, the direct load in \eqref{eq:appendix-bhat-orders} gives
$O(\kappa^{4M+4})$; below that level, the term
$\lambda_{i+1}\widetilde\omega_{i+1}$ propagates the displayed order
backward.  Equivalently, the worst new-time propagation path can be read as
\[
 b_1^Z
 \xrightarrow{\ \beta_{d,1}\cdots\beta_{d,M-1}\ }
 \widetilde u_M
 \xrightarrow{\ \psi_{d,M}\ }
 \mathcal E_{M-1}^{k+1}
 \xrightarrow{\ \lambda_{M-1}\cdots\lambda_{i+1}\ }
 \widetilde\omega_i
 \xrightarrow{\ \gamma_i\ }
 \|s^k\|^2.
\]
Its condition-number exponent is
$4(M-1)+8+2(M-1-i)-4M-4+2i=2M-2$,
so the large cross-level coefficients are fully accounted for rather than
suppressed in the bounds.

The barred weights that actually appear in the Lyapunov function satisfy
\begin{equation}
\label{eq:appendix-barred-weight-orders}
\overline u_i=O(\kappa^{4i-4M-8}),
\qquad
\overline\omega_i=
\begin{cases}
O(\kappa^{2M-2i-2}),&i<M,\\
O(\kappa^{-4}),&i=M.
\end{cases}
\end{equation}
Indeed, $\overline u_i=\tau_{d,i}\widetilde u_i$, while
$\overline\omega_i=\tau_i\widetilde\omega_i+
\psi'_{d,i}\widetilde u_i$.  These orders also control the initial tracking terms
in the Lyapunov function, as shown below.

Finally, for $i<M$,
\[
\widetilde\omega_i\gamma_i
=
O(\kappa^{6M-2i+2})
O(\kappa^{-4M-4+2i})
=
O(\kappa^{2M-2}),
\]
while the terminal primal contribution is smaller.  Moreover,
\[
\widetilde u_i\gamma_{d,i}
=
O(\kappa^{4i-4})O(\kappa^{2M-4i})
=
O(\kappa^{2M-4}).
\]
All preceding coefficient and weight bounds are uniform in
$\delta\in(0,\delta_{\rm adm}]$.  Applying the majorant construction from
Lemma~\ref{lem:Gamma-uniformity} to these bounds therefore gives $\bar\Gamma_{\rm adm}
=
O(\kappa^{2M-2}).$
This proves \eqref{eq:appendix-main-orders}.
\end{proof}

\paragraph{Initialization of the Lyapunov function.}
Under the initialization assumptions of
Theorem~\ref{thm:minimization-complexity},
$\mathcal E_i^0=O(1)$, $d_i^0=0$, and $s^0=0$.
Paired initialization gives
$\mathcal Z_i^0\le C_{\mathrm{path},i}^2
=O(\kappa^{2(M-i+1)})$.
Combining \eqref{eq:appendix-barred-weight-orders} with
$\alpha=c_\alpha\kappa^{-(2M+1)}$ yields
\[
\alpha\overline\omega_i\mathcal E_i^0
=
\begin{cases}
O(\kappa^{-2i-3}),&i<M,\\
O(\kappa^{-2M-5}),&i=M,
\end{cases}
\qquad
\alpha\overline u_i\mathcal Z_i^0
=O(\kappa^{2i-4M-7}),
\quad i=1,\ldots,M.
\]
These terms are uniformly bounded in $\kappa$ and $\delta$
at fixed depth. Since $\Phi(x^0)-\Phi^\star=O(1)$ and $s^0=0$,
it follows that $\mathcal V^0=O(1)$, uniformly over
$0<\delta\le\delta_{\rm adm}$.

\section{Bilevel structural simplification}
\label{app:bilevel-details}

For $M=1$, both branches use the exact input $x^k$; cross-level
forcing and input mismatch disappear. The following specialization
permits weaker contraction requirements than the general-depth analysis.

\begin{proposition}[Structural simplification for $M=1$]
\label{prop:bilevel-structural-simplification}
Suppose $M=1$ and the implicit layer is a strongly convex minimization
response.  With
$\mathcal E^k:=\mathcal E_1^k$ and $\mathcal Z^k:=\mathcal Z_1^k$,
\begin{equation}
\label{eq:bilevel-tracking-recursions}
\mathcal E^{k+1}
\le
\tau\mathcal E^k+\gamma_E\|s^k\|^2,
\qquad
\mathcal Z^{k+1}
\le
\tau_d\mathcal Z^k+\psi_d'\mathcal E^k+\gamma_Z\|s^k\|^2,
\end{equation}
where
\[
\tau=2\chi_1,
\qquad
\gamma_E=2\chi_1L_{z,1}^2,
\qquad
\tau_d=4\chi_{d,1}(1+C_{\mathrm{dec},1}\delta^2),
\]
\[
\psi_d'=8C_{\mathrm{dec},1}\chi_{d,1},
\qquad
\gamma_Z
=2\tau_dL_{d,1}^2
+8C_{\mathrm{dec},1}\chi_{d,1}L_{z,1}^2.
\]
In particular,
$C_{d,u,1}=C_{d,v,1}=\beta_{d,1}=\lambda_{d,1}=\psi_{d,1}=0$.
\end{proposition}

\begin{proof}
Here $\bar z_0^k=x^k=z_0^{\star}(x^k)$, and because level $1$ is also terminal
there is no deeper perturbation vector.  Hence $\mathcal R_0^k=0$, and the displayed
recursions follow directly from Lemma~\ref{lem:primal-tracking} and
Theorem~\ref{thm:auxiliary-tracking}.
\end{proof}

\begin{proposition}[Bilevel condition-number orders]
\label{prop:bilevel-kappa-orders}
Suppose $M=1$ in the strongly convex minimization setting and
\eqref{eq:single-kappa-regularity} holds.  Choose
$\chi_1,\chi_{d,1}\le c_1\kappa^{-2}$ for sufficiently small fixed
$c_1>0$, and shrink the admissible parameter,
independently of the target accuracy, so that
$\delta_{\rm adm}=O(\kappa^{-3})$.  Then, uniformly for
$0<\delta\le\delta_{\rm adm}$,
\begin{equation}
\label{eq:bilevel-Gamma-order}
C_{\mathrm{bias},0}=O(\kappa^3),
\qquad
L_\Phi=O(\kappa^3),
\qquad
\bar\Gamma_{\rm adm}=O(\kappa^4).
\end{equation}
Consequently
$\max\{L_\Phi,\sqrt{\bar\Gamma_{\rm adm}}\}=O(\kappa^3)$.
\end{proposition}

\begin{proof}
The single-level sensitivity orders are
$L_{z,1}=O(\kappa)$, $C_{\mathrm{path},1}=O(\kappa)$, and
$L_{d,1}=O(\kappa^3)$.  The bias recursion gives
$C_{\mathrm{bias},0}=O(\kappa^3)$, while
$L_{J_1}=O(\kappa^3)$ gives $L_\Phi=O(\kappa^3)$.  Furthermore
$C_{vZ,1}=O(1)$, $C_{vE,1}=O(\kappa^2)$, and
$C_{\mathrm{dec},1}=O(\kappa^4)$.  Proposition~\ref{prop:bilevel-structural-simplification}
therefore yields
\[
\tau,\tau_d=O(\kappa^{-2}),
\qquad
\gamma_E=O(1),
\qquad
\psi_d'=O(\kappa^2),
\qquad
\gamma_Z=O(\kappa^4).
\]
Thus $b_1^E=O(\kappa^2)$, $b_1^Z=O(1)$,
$\widetilde u_1=O(1)$, and $\widetilde\omega_1=O(\kappa^2)$ uniformly,
whence
$\bar\Gamma_{\rm adm}=O(\kappa^4)$.
Also $\overline u_1=O(\kappa^{-2})$ and
$\overline\omega_1=O(\kappa^2)$. Under paired initialization,
$\mathcal E_1^0=O(1)$ and $\mathcal Z_1^0=O(\kappa^2)$;
therefore $\mathcal V^0=O(1)$ when the initial objective gap is $O(1)$
and $\alpha=O(\kappa^{-3})$.
\end{proof}

\section{Centered recursive perturbations for minimization chains}
\label{app:centered}

This appendix concerns only the exact centered estimator for minimization
chains: $P_i=I$ and $\Psi_i=g_i$. We use $\zeta$ for its states to
distinguish them from the recursively defined one-sided states.

At the terminal level define
\begin{equation}
\label{eq:centered-terminal-states}
\begin{aligned}
\zeta_{M,\pm\delta}^{\star}(x)&:=\arg\min_z\{g_M(z,z_{M-1}^{\star}(x))\pm\delta f(x,z)\},\\
v_{M-1,\delta}^{\mathrm{cen}}(x)&:=
\frac{\nabla_2g_M(\zeta_{M,+\delta}^{\star},z_{M-1}^{\star})-
\nabla_2g_M(\zeta_{M,-\delta}^{\star},z_{M-1}^{\star})}{2\delta}.
\end{aligned}
\end{equation}
For $i=M-1,\ldots,1$, recursively set
\begin{equation}
\label{eq:centered-intermediate-states}
\begin{aligned}
\zeta_{i,\pm\delta}^{\star}(x)&:=\arg\min_z\{g_i(z,z_{i-1}^{\star}(x))
\pm\delta\langle z,v_{i,\delta}^{\mathrm{cen}}(x)\rangle\},\\
v_{i-1,\delta}^{\mathrm{cen}}(x)&:=
\frac{\nabla_2g_i(\zeta_{i,+\delta}^{\star},z_{i-1}^{\star})-
\nabla_2g_i(\zeta_{i,-\delta}^{\star},z_{i-1}^{\star})}{2\delta}.
\end{aligned}
\end{equation}
Let
$G_{\delta,\mathrm{cen}}^{\star}(x):=\nabla_1f(x,z_M^{\star}(x))+v_{0,\delta}^{\mathrm{cen}}(x)$.
The centered states are exact nonlinear responses, not generally
$z_i^\star\pm\delta d_i^\star$.

\begin{assumption}[Additional smoothness for centered perturbations]
\label{ass:centered-smoothness}
In addition to the standing strongly convex minimization assumptions, each
$g_i$ and $f$ is four times continuously differentiable. Their derivatives
of orders two through four are bounded on a common neighborhood of the
perturbed path families used below, uniformly over bounded forcing
vectors and uniformly in $x$ whenever a global bias bound is claimed.
The standing bound on $\nabla_2f$ is retained; no global bound on the
other first derivatives is imposed.
There exists $\delta_{\rm cen}>0$ such
that both terminal perturbed problems remain uniformly strongly convex
for $|s|\le\delta_{\rm cen}$; it suffices that
$\delta_{\rm cen}L_f<\mu_M$.
\end{assumption}

\begin{proposition}[Second-order centered bias]
\label{prop:centered-bias}
Under the standing strongly convex minimization assumptions and
Assumption~\ref{ass:centered-smoothness}, for
$0<\delta\le\delta_{\rm cen}$ there exist finite constants
$b_i^{\mathrm{cen}}$ such that
\begin{equation}
\label{eq:centered-recursive-bias}
\|v_{M-1,\delta}^{\mathrm{cen}}-v_{M-1}^{\star}\|\le b_M^{\mathrm{cen}}\delta^2,
\qquad
\|v_{i-1,\delta}^{\mathrm{cen}}-v_{i-1}^{\star}\|
\le\kappa_i\|v_{i,\delta}^{\mathrm{cen}}-v_i^{\star}\|+b_i^{\mathrm{cen}}\delta^2
\quad(i<M).
\end{equation}
Consequently, for every fixed $x$ the following bound holds; if the
neighborhood and derivative bounds in Assumption~\ref{ass:centered-smoothness}
are uniform in $x$, the constant below is uniform in $x$ as well:
\begin{equation}
\label{eq:centered-gradient-bias}
\|G_{\delta,\mathrm{cen}}^{\star}(x)-\nabla\Phi(x)\|
\le C_{\mathrm{bias},0}^{\mathrm{cen}}\delta^2,
\qquad
C_{\mathrm{bias},0}^{\mathrm{cen}}:=
\sum_{j=1}^M\left(\prod_{\ell=1}^{j-1}\kappa_\ell\right)b_j^{\mathrm{cen}}.
\end{equation}
\end{proposition}

\begin{proof}
Fix $x$. We use the standard centered Taylor estimate: if a vector map
$\mathcal G$ is $C^3$ on $[-\delta_{\rm cen},\delta_{\rm cen}]$ and
$\sup_s\|\mathcal G^{(3)}(s)\|\le M_3$, then
\begin{equation}
\label{eq:centered-Taylor}
\left\|\frac{\mathcal G(\delta)-\mathcal G(-\delta)}{2\delta}
-\mathcal G'(0)\right\|\le\frac{M_3}{6}\delta^2.
\end{equation}

First establish bounded forcing vectors without a Taylor argument.
The terminal stationarity equations give
$\|\zeta_{M,\pm\delta}^\star-z_M^\star\|
\le\delta C_f/\mu_M$, and hence
$\|v_{M-1,\delta}^{\mathrm{cen}}\|\le\kappa_M C_f$.
For $i<M$, strong monotonicity similarly gives
$\|\zeta_{i,+\delta}^\star-\zeta_{i,-\delta}^\star\|
\le2\delta\|v_{i,\delta}^{\mathrm{cen}}\|/\mu_i$.
Therefore
\[
\|v_{i,\delta}^{\mathrm{cen}}\|
\le C_f\prod_{j=i+1}^M\kappa_j,\qquad i=0,\ldots,M-1,
\]
uniformly in $\delta$. Thus the derivative bounds on bounded forcing
families in Assumption~\ref{ass:centered-smoothness} apply independently
of the bias estimate being proved.

For the terminal level, let $z_M(s)$ minimize
$g_M(z,z_{M-1}^{\star})+sf(x,z)$ and set
\[
\mathcal G_M(s):=\nabla_2g_M(z_M(s),z_{M-1}^{\star}).
\]
Uniform strong convexity and Assumption~\ref{ass:centered-smoothness} imply, by the
implicit-function theorem, that $\mathcal G_M$ is $C^3$ with bounded
third derivative. Differentiating the stationarity equation at zero gives
\[
\dot z_M(0)=-[H_M^{\star}]^{-1}v_M^{\star},\qquad
\mathcal G_M'(0)=-B_M^{\star}[H_M^{\star}]^{-1}v_M^{\star}=v_{M-1}^{\star}.
\]
Applying \eqref{eq:centered-Taylor} proves the terminal estimate in
\eqref{eq:centered-recursive-bias}.

For $i<M$ and a bounded perturbation vector $w$, let $z_i(s;w)$ minimize
$g_i(z,z_{i-1}^{\star})+s\langle z,w\rangle$ and set
$\mathcal G_i(s;w):=\nabla_2g_i(z_i(s;w),z_{i-1}^{\star})$. The same
implicit-function argument gives a uniform third-derivative bound on
bounded sets of $w$, and differentiation at zero yields
$\partial_s\mathcal G_i(0;w)=-B_i^{\star}[H_i^{\star}]^{-1}w$.
Taking $w=v_{i,\delta}^{\mathrm{cen}}$ in \eqref{eq:centered-Taylor}
gives
\[
\|v_{i-1,\delta}^{\mathrm{cen}}+B_i^{\star}[H_i^{\star}]^{-1}v_{i,\delta}^{\mathrm{cen}}\|
\le b_i^{\mathrm{cen}}\delta^2.
\]
Subtracting the exact adjoint identity and
using $\|B_i^{\star}[H_i^{\star}]^{-1}\|\le\kappa_i$ proves the recursive estimate.
Unrolling it gives the constant in \eqref{eq:centered-gradient-bias}; the
direct term $\nabla_1f(x,z_M^{\star})$ is common to both estimators.
\end{proof}

\begin{remark}[Scope of the centered result]
Proposition~\ref{prop:centered-bias} concerns the exact estimator only. An
implementable centered method would have to track the positive and negative
perturbation branches separately; controlling their centered difference is
not enough to control the individual branch errors. We therefore do not claim
a convergence theorem for the finite-iteration two-branch method.
\end{remark}

\section{Details of the Numerical Experiments}
\subsection{Experimental protocol}
\label{app:experimental-protocol}

\paragraph{Implementations and timing.}
Experiments use an Intel Core i5-1038NG7 CPU at 2.00~GHz, with four
intra-op threads and one inter-op thread. PDE runs use float64; MNIST
training uses float32 and reference evaluation uses float64.
Application methods use identical data and zero outer and response
initializations within each instance. Response states are warm-started
across outer iterations. Runs are serial, in fresh processes; revised PDE
configurations form separate batches.

Algorithm time includes response updates, sensitivity computation, and
outer updates. Setup, spectral calibration, saved-state output, and
independent reference evaluation are excluded. An update crossing the
budget is discarded, together with its oracle counts, and the last
accepted iterate is returned. Application statistics use seeds $6$--$10$;
seed $5$ is used for initial calibration. The final operating points are
listed in Table~\ref{tab:experimental-operating-points}.
For \ReMID{}, each layer performs $T$ unperturbed and $T$ perturbed
response updates per outer iteration.

\begin{table}[t]
\centering
\caption{Final operating points $(T,\eta_{\rm out})$.
The implicit baseline is ID-CG for PDE control and MNIST, and ID-MINRES
for the saddle-point PDE; the trajectory baseline is Sato/ITD or ITD-SG,
respectively. \ReMID{} uses $\delta=10^{-3}$.
For MNIST, the listed outer steps apply to $400\Phi$,
and the perturbation coefficient is $\delta/400$.
Reported objective values and gradient norms use $\Phi$.}
\label{tab:experimental-operating-points}
\label{tab:mnist-frozen-configurations}
\small
\begin{tabular}{@{}lcccc@{}}
\toprule
Problem & Budget (s) & \ReMID{} & Implicit & Trajectory\\
\midrule
PDE control & 300 & $(4,150)$ & $(4,250)$ & $(24,800)$\\
Saddle-point PDE & 500 & $(12,125)$ & $(12,150)$ & $(32,2400)$\\
MNIST & 150 & $(8,0.8)$ & $(4,0.8)$ & $(12,1.0)$\\
\bottomrule
\end{tabular}
\end{table}

\paragraph{Sensitivity computation.}
\ReMID{} performs a forward unperturbed refresh followed by a backward
perturbed sweep. Each unperturbed--perturbed response pair uses the same
refreshed parent, with both states warm-started across outer iterations.
The terminal outer-loss gradient is reevaluated at each perturbed update.
The implicit baselines compute adjoints recursively, using CG for
minimization responses and MINRES for saddle responses. Following the
public implementations associated with
\cite{grazzi2020iteration,ji2021bilevel},\footnote{
\url{https://github.com/prolearner/hypertorch};
\url{https://github.com/JunjieYang97/stocBiO}.}
we compute matrix-free HVPs by automatic differentiation, reusing each
level's gradient graph throughout its linear solve.
Sato/ITD differentiates the current $T$-step GD trajectories across
levels \cite{sato2021multilevel}; ITD-SG instead differentiates
signed-gradient trajectories. Warm-started response states are detached
at the start of each outer iteration, so trajectory differentiation
holds these initial states fixed.

\paragraph{Statistics and plots.}
Controlled diagnostics report medians and interquartile ranges;
applications report means and sample standard deviations over five runs.
Objective differences are computed per run before aggregation, and
gradient norms are averaged. Curves use a five-second grid: linear
interpolation for PDE and the last completed iterate for MNIST, with
terminal values held to the budget. All recorded transients are retained;
nonpositive lower standard-deviation bounds are clipped only for display
on logarithmic axes.

\paragraph{Oracle counts.}
FO denotes first-order oracle calls (gradient evaluations).
Per accepted outer update, FO counts are $2MT+2M+1$, $MT+M+1$, and
$MT+1$ for \ReMID{}, implicit, and trajectory methods in PDE;
MNIST counts only response-gradient evaluations: $2MT$, $MT$, and $MT$,
respectively. All other training computations remain timed.
HVP counts include explicit CG/MINRES operator calls and residual checks;
ITD counts $MT$ differentiated lower updates per outer update.
HVPs arising within trajectory differentiation are not counted separately.
These implementation counts differ from the theoretical
separate-block oracle convention.
\subsection{Controlled diagnostics}
\label{app:controlled-diagnostics}

The radial pseudo-Huber loss is
\begin{equation}
\label{eq:diagnostic-pseudohuber}
\phi_c(q)=c^2\left(\sqrt{1+\|q\|^2/c^2}-1\right).
\end{equation}
Its gradient and Hessian are bounded, and its Hessian is Lipschitz.
Unless stated otherwise, perturbation diagnostics use high-accuracy
response solves and independently computed implicit reference gradients.

\paragraph{Bias in nonlinear chains.}
For Figure~\ref{fig:controlled-diagnostics}(a), scalar response potentials
are
\[
g_i(z,u)=\tfrac12a_i z^2+b_i(1-\cos z)-c_iuz,
\qquad f(x,z_M)=\phi_1(z_M-\tau)+0.025x^2.
\]
Independently sample $a_i\sim\mathcal U[1.2,1.8]$,
$b_i\sim\mathcal U[0.10,0.35]$, $c_i\sim\mathcal U[0.45,0.85]$,
$\tau\sim\mathcal U[-0.8,0.8]$, and $x\sim\mathcal U[-0.9,0.9]$.
The state curvature is at least $a_i-b_i\ge0.85$.
For each $M\in\{2,3,5\}$, use 30 instances and 25 logarithmically
spaced radii in $[10^{-5},2\times10^{-1}]$.
Newton solves stop at increment tolerance $10^{-14}(1+|z_i|)$.
Log--log fits retain errors in $[2\times10^{-12},2\times10^{-2}]$;
the fitted slope ranges are $[0.994,0.997]$ for one-sided and
$[2.000,2.001]$ for centered perturbations.

\paragraph{Depth and conditioning.}
For Figure~\ref{fig:controlled-diagnostics}(b), set
$g_i(z,u)=\frac12z^\top H_i z-z^\top H_iS_i u$ in dimension 16,
with spectra of $H_i$ in $[1,\kappa]$ and singular values of $S_i$
in $[0.92,1]$. Use $\kappa\in\{2,4,8,16,32\}$,
$M\in\{2,3,5,8\}$, and 20 instances per pair.
The affine terminal loss $f(x,z_M)=a^\top z_M+0.025\|x\|^2$,
$\|a\|=1$, removes perturbation bias. At $\delta=10^{-3}$,
linear solves give the exact reference; perturbed solves start at the
exact unperturbed states and take $T$ first-order steps.
Fits of $\log e=A_M-B_M T/\kappa$ give
$B_M\in[1.143,1.177]$ and $R^2\ge0.9985$.
This experiment isolates solve error at a fixed input.

\paragraph{Saddle and mixed chains.}
For Figure~\ref{fig:controlled-diagnostics}(c), use
$\Psi_i(z,u)=\frac12z^\top H_i z-z^\top H_iS_i u$ with
contractive $S_i$. Minimization layers have $H_i\succ0$ and $P_i=I$;
saddle layers have
\[
H_i=\begin{pmatrix}A_i&C_i\\C_i^\top&-D_i\end{pmatrix},
\qquad P_i=\operatorname{diag}(I,-I),\qquad A_i,D_i\succ0.
\]
The signed operator is strongly monotone since
$\operatorname{sym}(P_iH_i)=\operatorname{diag}(A_i,D_i)\succ0$.
The terminal loss is $\phi_1(Cz_M-r)+0.025\|x\|^2$,
where $C=I+0.12E/\sqrt d$, $E$ is Gaussian, and
$r\sim\mathcal N(0,0.7^2I)$.
Compare all-minimization, all-saddle, and four-level
minimization--saddle--minimization--saddle chains over 30 instances.
The fitted one-sided slopes lie in $[0.9977,0.9998]$.

\paragraph{Finite-solve error and the forward refresh.}
For Figure~\ref{fig:finite-computation-diagnostics}(a), use the quadratic
chains above with $M=5$, $\kappa=8$, $d=16$, and terminal loss
$\frac12\|z_M-r\|^2+0.025\|x\|^2$.
Unperturbed responses are exact; each perturbed solve starts at its
unperturbed state and takes $T\in\{8,16,32,64\}$ steps.
Vary $\delta$ over $[10^{-5},10^{-1}]$ and use 30 instances.
The initial perturbed displacement scales with $\delta$, explaining
the bounded error floor after division by $\delta$.

For panel~(b), compare a forward refresh using newly updated parents
with a lagged refresh using parents from the preceding outer iteration.
Both retain numerical warm starts and use the same data, initializations,
solvers, outer steps, $\delta$, and $2MT$ inner-gradient evaluations.
Set $\kappa=8$, $T=16$, $\delta=10^{-3}$,
$M\in\{2,3,5,8\}$, and use 20 matched instances per depth.
Tracking error is measured after $k=2M$ outer iterations against the
exact perturbation estimator $G_\delta^\star(x^k)$.
These quadratic-loss tests examine local solve and tracking mechanisms;
their terminal gradients are not globally bounded.

\begin{figure}[t]
\centering
\begin{minipage}[t]{0.48\textwidth}
\centering
\includegraphics[width=\linewidth]{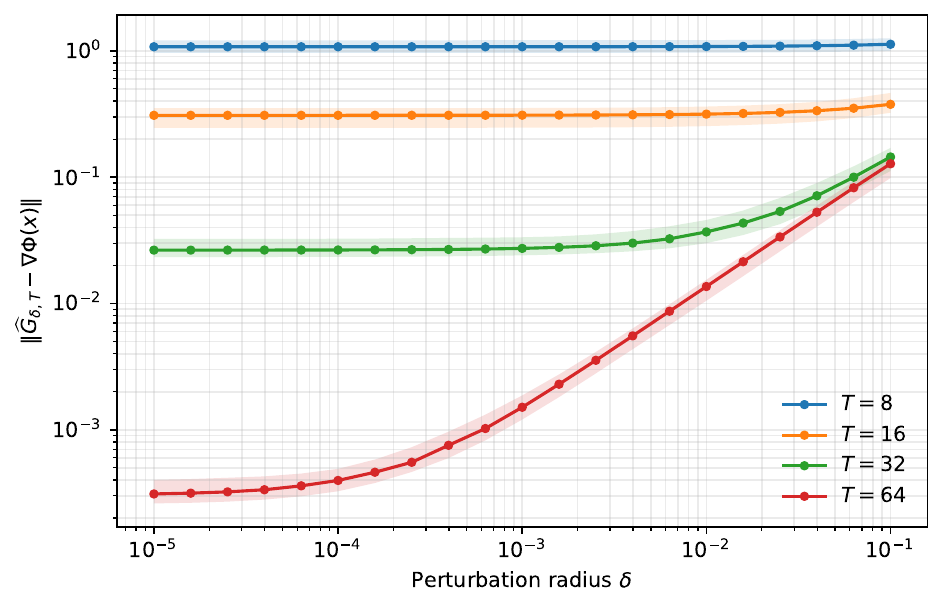}
\par\smallskip {\small (a) Perturbation and inner accuracy}
\end{minipage}\hfill
\begin{minipage}[t]{0.43\textwidth}
\centering
\includegraphics[width=\linewidth]{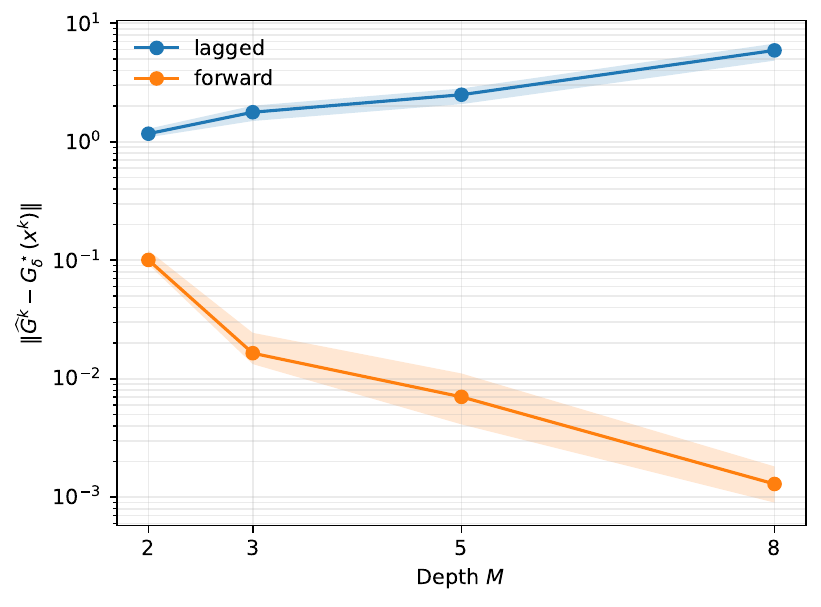}
\par\smallskip {\small (b) Updated-input ablation}
\end{minipage}
\caption{Finite-solve and scheduling diagnostics.
(a) $\|\widehat G_{\delta,T}-\nabla\Phi(x)\|$ versus $\delta$ for
different inner budgets.
(b) $\|\widehat G^k-G_\delta^\star(x^k)\|$ versus depth for updated-input
and previous-input schedules. Bands show interquartile ranges.}
\label{fig:finite-computation-diagnostics}
\end{figure}

\subsection{PDE control problems}
\label{app:pde-details}

\paragraph{Discretization and parameters.}
Both PDE problems use the heat equation on the unit square, homogeneous
Dirichlet boundary conditions, zero initial state, and a distributed
source constant in time. A five-point Laplacian on a $64\times64$
interior grid is integrated by 100 explicit-Euler steps, with
\[
\nu_i=0.02\left(0.90+0.20\frac{i-1}{5}\right),\qquad
\nu_i\Delta t_i/h^2=0.20,\qquad h=1/65.
\]
The validation domain has $\nu_{\rm val}=0.022$.
Each domain uses an independent $25\%$ observation mask, weight 20,
and Gaussian observation noise with standard deviation $0.005$.
For PDE control, $(\lambda_c,\mu,\tau)=(0.03,0.002,10^{-4})$,
with pseudo-Huber scale $c=5$.
All methods use $\eta_i=1.8/\widehat L_i$, where 15 power iterations
estimate the data-Hessian norm and $\lambda_c+\mu$ is added.
These are numerical estimates rather than certified upper bounds.

\paragraph{Parameter selection.}
Initial tuning uses seed $5$ with a 100-second algorithm-time budget
per candidate. We first compare the common inner-step scales
$\{0.9,1.5,1.8\}$ and select $1.8$. At this scale, we test
$T\in\{4,6,8\}$ for \ReMID{} and ID-CG, and
$T\in\{12,16,24\}$ for Sato/ITD.
The outer-step search covers $\{50,75,100,125,150\}$ for \ReMID{},
$\{50,75,100,125,150,200,250,300,400,500,600\}$ for ID-CG,
and $\{800,1600,2400\}$ for Sato/ITD.
We assess configurations jointly by their terminal objective values,
reference-gradient norms, and trajectory stability, rather than
by either endpoint metric alone.
Following stability checks on the evaluation runs, the outer steps
for ID-CG and Sato/ITD are reduced to 250 and 800, respectively.
The final settings are
$(T,\eta_{\rm out})=(4,150)$ for \ReMID{},
$(4,250)$ for ID-CG, and $(24,800)$ for Sato/ITD.
All methods use $\eta_i=1.8/\widehat L_i$, and these parameters
remain fixed throughout the reported 300-second runs.

\paragraph{Reference evaluations.}
For PDE control, reference responses use analytic matrix-free CG
with relative residual tolerance $10^{-10}$ and a cap of 500 iterations; reference adjoints use
$10^{-9}$ and a cap of 120. Training ID-CG instead uses autograd HVPs,
CG tolerance $10^{-6}$, and a cap of 80.
The affine responses and convex tracking loss make the exact composite
objective strongly convex. Separate spectral L-BFGS solves, with Newton
refinement where needed, provide numerical reference optima.
The spectral and original PDE maps, objectives, and gradients are
cross-checked; reference-gradient norms are below $8\times10^{-10}$.
Unrounded reference objectives are used for the reported gaps.
Reference checkpoints follow five-second algorithm-time thresholds.

\begin{table}[t]
\centering
\caption{PDE control: terminal mean $\pm$ sample standard
deviation under a 300-second algorithm-time budget.
$\Delta\Phi=\Phi-\Phi_{\rm ref}^\star$ uses numerical reference optima.
Counts are in thousands and exclude reference evaluations.}
\label{tab:pde-main}
\small\setlength{\tabcolsep}{3.5pt}
\begin{tabular}{@{}lccccc@{}}
\toprule
Method & $\Delta\Phi$ & $\|\nabla\Phi\|$ & FO & HVP & ITD\\
\midrule
\ReMID{} & $(\mathbf{4.83}\pm0.99)\times10^{-9}$ & $(\mathbf{2.01}\pm0.18)\times10^{-6}$ & $26.245\pm0.206$ & $0$ & $0$\\
ID-CG & $(5.28\pm1.69)\times10^{-6}$ & $(4.01\pm0.64)\times10^{-5}$ & $2.631\pm0.081$ & $9.714\pm0.293$ & $0$\\
Sato/ITD & $(4.73\pm0.87)\times10^{-5}$ & $(1.41\pm0.13)\times10^{-4}$ & $4.961\pm0.121$ & $0$ & $4.920\pm0.120$\\
\bottomrule
\end{tabular}
\end{table}

\subsection{Saddle-point extension}
\label{app:pde-signed}

For $z=(y,\xi)$ and $u=(u_y,u_\xi)$, the potential is
\begin{equation}
\label{eq:pde-saddle-potential}
\begin{aligned}
\Psi_i(y,\xi;u)={}&
\frac{\omega_i}{2}\|M_i(\mathcal K_i(y+\beta_\xi\xi)-d_i^{\rm tar})\|^2
+\frac{\lambda_y}{2}\|y-u_y\|^2+\frac{\mu_y}{2}\|y\|^2-\frac{\lambda_\xi}{2}\|\xi-u_\xi\|^2-\frac{\mu_\xi}{2}\|\xi\|^2.
\end{aligned}
\end{equation}
The first input is $(x,0)$; later inputs are preceding response pairs.
The outer objective uses the same tracking loss evaluated at
$\mathcal K_{\rm val}(y_M+\beta_\xi\xi_M)$, plus $\tau\|x\|^2/2$.
Set $(\lambda_y,\mu_y,\lambda_\xi,\mu_\xi,\beta_\xi)
=(0.08,0.01,0.12,0.02,0.40)$ and $\tau=10^{-4}$.

With $L_{{\rm data},i}=\|\sqrt{\omega_i}M_i\mathcal K_i\|^2$,
the condition $\lambda_\xi+\mu_\xi>\beta_\xi^2L_{{\rm data},i}$
ensures strong monotonicity of $P_i\nabla_1\Psi_i$.
The implementation uses 15-step power estimates and the common steps
\[
\begin{aligned}
\widehat\mu_i^{\rm sg}
&=\min\{\lambda_y+\mu_y,\lambda_\xi+\mu_\xi
-\beta_\xi^2\widehat L_{{\rm data},i}\},\\
\widehat L_i^{\rm sg}
&=(1+\beta_\xi^2)\widehat L_{{\rm data},i}
+\max\{\lambda_y+\mu_y,\lambda_\xi+\mu_\xi\},\qquad
\eta_i=\frac{1.8\widehat\mu_i^{\rm sg}}{(\widehat L_i^{\rm sg})^2}.
\end{aligned}
\]
The estimated concavity margin is checked positive; power iteration
alone does not certify the true margin. The update is
$z\leftarrow z-\eta_iP_i\nabla_1\Psi_i(z,u)$.

The operating points in Table~\ref{tab:experimental-operating-points}
use seed-$5$ calibration and additional sensitivity checks, with the
same evaluation seeds examined in earlier configurations.
ID-MINRES solves against the symmetric potential Hessian, not the
generally nonsymmetric signed-gradient Jacobian. Training uses tolerance
parameter $10^{-6}$ and at most 120 MINRES iterations; reference solves
use analytic Hessian actions, parameter $10^{-10}$, and a cap of 300,
with actual relative residuals recorded. Reference checkpoints follow
ten-second thresholds. Independent spectral L-BFGS/Newton refinement
provides numerical optima, checked against the original PDE implementation.
The final comparison combines separate serial batches at the reported
parameters; it is not a new interleaved three-method timing experiment.

\begin{table}[t]
\centering
\caption{Saddle-point PDE: terminal mean $\pm$ sample standard
deviation under a 500-second algorithm-time budget.
Counts are in thousands and exclude reference evaluations.}
\label{tab:pde-signed}
\label{tab:robust-pde-500s}
\small\setlength{\tabcolsep}{3.5pt}
\begin{tabular}{@{}lccccc@{}}
\toprule
Method & $\Delta\Phi$ & $\|\nabla\Phi\|$ & FO & HVP & ITD\\
\midrule
\ReMID{} & $(\mathbf{2.16}\pm0.56)\times10^{-5}$ & $(\mathbf{7.94}\pm1.17)\times10^{-5}$ & $34.12\pm1.17$ & $0$ & $0$\\
ID-MINRES & $(7.40\pm1.72)\times10^{-4}$ & $(4.71\pm0.54)\times10^{-4}$ & $8.17\pm0.24$ & $13.54\pm0.40$ & $0$\\
ITD-SG & $(4.26\pm0.66)\times10^{-2}$ & $(4.18\pm0.37)\times10^{-3}$ & $8.26\pm0.51$ & $0$ & $8.16\pm0.51$\\
\bottomrule
\end{tabular}
\end{table}

\subsection{MNIST graph learning}
\label{app:domain-adaptation}

\paragraph{Data and parameterization.}
Each instance uses 2000 class-balanced MNIST training images, scaled to
$[0,1]$, with disjoint source-label and outer-label sets
$|\mathcal A|=|\mathcal V|=400$ and 1200 remaining graph nodes.
The same image identities appear at rotations $\theta_i=60^\circ i/8$,
using bilinear interpolation. Each domain has a symmetrized 10-nearest-
neighbor graph in pixel space, with self-edges removed and each undirected
edge retained once. Weights are
$\omega_{i,pq}=\exp(-d_{i,pq}^2/(2s_i))$, where $s_i$ is the median
squared edge distance, bounded below by $10^{-12}$.

The fixed feature map has 127 Fourier coordinates
$\varphi_j(a)=\sqrt{2/127}\cos(\xi_j^\top a+b_j)$,
$\xi_j\sim\mathcal N(0,9I/784)$, $b_j\sim\mathcal U[0,2\pi]$,
and a bias coordinate. The source classifier minimizes
$\operatorname{CE}_{\mathcal A}(A_0W,y)+10^{-3}\|W\|_F^2/2$
on unrotated features $A_0$, using L-BFGS with a strong-Wolfe line search
and a 100-iteration cap, gradient tolerance $10^{-9}$, and
objective-change tolerance $10^{-12}$. It is then fixed, with first parent
$A_1(W_{\rm src}+x)$. Only labels in $\mathcal A$ and $\mathcal V$ are
used; the latter are outer-training labels, not held-out validation labels.

\paragraph{Potentials and normalization.}
For $z,u\in\mathbb R^{N\times10}$,
\begin{equation}
\label{eq:mnist-lower-objective}
\begin{aligned}
\Psi_i(z,u)={}&\mathbf1_{\{i=1\}}\operatorname{CE}_{\mathcal A}(z,y)
+\frac{\alpha}{N}\sum_{\{p,q\}\in E_i}\omega_{i,pq}\phi_r(z_p-z_q)+\frac{\lambda_c}{2N}\|z-u\|_F^2+\frac{\mu}{2N}\|z\|_F^2,
\end{aligned}
\end{equation}
where $(\alpha,r,\lambda_c,\mu)=(0.03,1,0.3,0.02)$ and $N=2000$.
The state curvature is bounded below by $(\lambda_c+\mu)/N=1.6\times10^{-4}$.
The smooth convex graph penalties have bounded Lipschitz Hessians;
the affine first-layer substitution preserves response regularity.
The mean cross-entropy has bounded terminal gradient and Lipschitz Hessian.
The reported outer objective uses $\tau=10^{-6}$ and is nonnegative,
but the nonlinear composition need not be convex.

Training uses $\Phi_{\rm sum}=400\Phi$, so the listed outer steps apply
to this summed objective. Reported gradients are divided by 400.
For \ReMID{}, $\delta_{\rm mean}=10^{-3}$ and
$\delta_{\rm sum}=\delta_{\rm mean}/400=2.5\times10^{-6}$, giving the
same terminal perturbation in either normalization.

\paragraph{Stepsizes and calibration.}
Let $\mathcal L_i$ be the weighted graph Laplacian and
$D_{\mathcal A}$ the source-node indicator matrix. Set
\[
B_i=\frac{\alpha}{N}\mathcal L_i+\frac{\lambda_c+\mu}{N}I
+\mathbf1_{\{i=1\}}\frac{D_{\mathcal A}}{2|\mathcal A|},
\qquad L_i=(1+10^{-6})\lambda_{\max}(B_i),
\]
using eigensolver tolerance $10^{-11}$.
All methods use $\eta_i=1.8/L_i$ for $i<M$ and
$\eta_M=1.8/(L_M+\delta_{\rm mean}/(2|\mathcal V|))$.
Training ID-CG uses relative residual tolerance $10^{-6}$ and at most
60 CG iterations.

Seed-$5$ calibration uses 150 seconds per candidate, testing
$T\in\{4,6,8\}$ for \ReMID{} and ID-CG and $\{6,12,24\}$ for Sato/ITD,
first at outer step 0.8, then $\{0.4,0.8,1.0\}$ at the selected budget.
The score $Q_{150}$ is defined as above, with Frobenius gradient norms;
candidates whose objective and gradient both rise from 100 to 150 seconds
are excluded. The configurations in
Table~\ref{tab:experimental-operating-points} were frozen before the
reported five-run evaluation.

\paragraph{Reference evaluation and empirical values.}
Responses are evaluated in float64 using accelerated GD, at most 5000
iterations, and stopping test
$\|\nabla_1\Psi_i(z,u)\|_F/(1+\|z\|_F)\le10^{-11}$.
Reference adjoints use autograd CG with tolerance $10^{-9}$ and at most
500 iterations. Nonconverged reference solves are rejected.
For each run, separate L-BFGS-B optimization starts from the
lowest-objective saved iterate among the three methods, using these
high-accuracy response and adjoint evaluations. It uses at most 250
outer iterations, objective-change tolerance $10^{-14}$, and requested
gradient tolerance $10^{-7}$. It stops on objective changes before
reaching that gradient target. The retained point is reevaluated with
cold response solves at tolerance $10^{-13}$ and adjoints at $10^{-9}$.
Its value $\widehat\Phi_s^\star$ is an empirical reference, not a
certificate of global optimality; its gradient norms range from
$3.74\times10^{-7}$ to $1.21\times10^{-6}$.

\begin{table}[t]
\centering
\caption{MNIST graph learning: terminal mean $\pm$ sample standard
deviation under a 150-second algorithm-time budget.
$\Delta\Phi=\Phi-\widehat\Phi^\star$ uses empirical reference values.
Counts are in thousands and exclude reference evaluations.}
\label{tab:mnist-domain-adaptation}
\small\setlength{\tabcolsep}{4pt}
\begin{tabular}{@{}lccccc@{}}
\toprule
Method & $\Delta\Phi\;(10^{-3})$ & $\|\nabla\Phi\|_F\;(10^{-4})$ & FO & HVP & ITD\\
\midrule
\ReMID{} & $\mathbf{5.05}\pm1.27$ & $\mathbf{1.65}\pm0.21$ & $107.67\pm12.06$ & $0$ & $0$\\
ID-CG & $11.16\pm1.69$ & $2.72\pm0.14$ & $16.38\pm1.36$ & $30.85\pm1.56$ & $0$\\
Sato/ITD & $18.32\pm4.38$ & $3.75\pm0.76$ & $33.39\pm3.80$ & $0$ & $33.39\pm3.80$\\
\bottomrule
\end{tabular}
\end{table}

\bibliographystyle{plain}
\bibliography{references_mp}

@inproceedings{ling2018game,
  author    = {Ling, Chun Kai and Fang, Fei and Kolter, J. Zico},
  title     = {What Game Are We Playing? End-to-End Learning
               in Normal and Extensive Form Games},
  booktitle = {Proceedings of the Twenty-Seventh International
               Joint Conference on Artificial Intelligence},
  pages     = {396--402},
  year      = {2018},
  doi       = {10.24963/ijcai.2018/55},
  url       = {https://www.ijcai.org/proceedings/2018/55}
}

@inproceedings{yao2025gap,
  author    = {Yao, Wei and Yin, Haian and Zeng, Shangzhi
               and Zhang, Jin},
  title     = {Overcoming Lower-Level Constraints in Bilevel
               Optimization: A Novel Approach with
               Regularized Gap Functions},
  booktitle = {The Thirteenth International Conference
               on Learning Representations},
  year      = {2025},
  url       = {https://openreview.net/forum?id=cyPMEXdqQ2}
}

@article{shen2026minimax,
  author  = {Shen, Yiyang and He, Yutian and Wang, Weiran
             and Lin, Qihang},
  title   = {Penalty-Based First-Order Methods for Bilevel
             Optimization with Minimax and Constrained
             Lower-Level Problems},
  journal = {arXiv preprint arXiv:2605.08006},
  year    = {2026},
  doi     = {10.48550/arXiv.2605.08006},
  url     = {https://arxiv.org/abs/2605.08006}
}

@article{ostroukhov2022tensor,
  title={Tensor methods for strongly convex strongly concave
         saddle point problems and strongly monotone
         variational inequalities},
  author={Ostroukhov, Petr and Kamalov, Rinat and
          Dvurechensky, Pavel and Gasnikov, Alexander},
  journal={Computer Research and Modeling},
  volume={14},
  number={2},
  pages={357--376},
  year={2022},
  url={https://arxiv.org/abs/2012.15595}
}

@article{paige1975solution,
  author  = {Paige, Christopher C. and Saunders, Michael A.},
  title   = {Solution of Sparse Indefinite Systems of Linear Equations},
  journal = {SIAM Journal on Numerical Analysis},
  volume  = {12},
  number  = {4},
  pages   = {617--629},
  year    = {1975}
}

@article{somayajula2022multilevel,
  author  = {Sai Ashish Somayajula and Linfeng Song and Pengtao Xie},
  title   = {A Multi-Level Optimization Framework for End-to-End Text Augmentation},
  journal = {Transactions of the Association for Computational Linguistics},
  volume  = {10},
  pages   = {343--358},
  year    = {2022},
  doi     = {10.1162/tacl_a_00464}
}

@article{fang2022endtoend,
  author  = {Hongchao Fang and Pengtao Xie},
  title   = {An End-to-End Contrastive Self-Supervised Learning Framework for Language Understanding},
  journal = {Transactions of the Association for Computational Linguistics},
  volume  = {10},
  pages   = {1324--1340},
  year    = {2022},
  doi     = {10.1162/tacl_a_00521}
}

@article{zhang2025tapweight,
  author  = {Ruiyi Zhang and Sai Ashish Somayajula and Pengtao Xie},
  title   = {{TapWeight}: Reweighting Pretraining Objectives for Task-Adaptive Pretraining},
  journal = {Transactions on Machine Learning Research},
  year    = {2025},
  url     = {https://openreview.net/forum?id=DCCw2CEVFS}
}

@article{jiang2026bizol,
  title   = {{Bi-ZOL}: Bilevel Zeroth-Order Learning with Nonsmooth Responses},
  author  = {Jiang, Zhisen and Bolognani, Saverio},
  journal = {arXiv preprint arXiv:2609.08021},
  year    = {2026}
}

@article{zheng2026bilevel,
  author  = {Zihao Zheng and Irwin King and Songtao Lu},
  title   = {Bilevel Optimization over Saddle Points of Zero-Sum Markov Games},
  journal = {arXiv preprint arXiv:2605.26654},
  year    = {2026},
  eprint  = {2605.26654},
  archivePrefix = {arXiv}
}

@article{shen2026penalty,
  author  = {Yiyang Shen and Yutian He and Weiran Wang and Qihang Lin},
  title   = {Penalty-Based First-Order Methods for Bilevel Optimization
             with Minimax and Constrained Lower-Level Problems},
  journal = {arXiv preprint arXiv:2605.08006},
  year    = {2026},
  eprint  = {2605.08006},
  archivePrefix = {arXiv}
}

@article{persak2024decision,
  author        = {Egon Per{\v{s}}ak and Miguel F. Anjos},
  title         = {Decision-Focused Forecasting: Decision Losses for
                   Multistage Optimisation},
  journal       = {arXiv preprint arXiv:2405.14719},
  year          = {2024},
  eprint        = {2405.14719},
  archiveprefix = {arXiv},
  url           = {https://arxiv.org/abs/2405.14719}
}

@article{xia2026sequentialmarkets,
  author  = {Yuxin Xia and Iacopo Savelli and Thomas Morstyn},
  title   = {Scalable Multi-Level Optimization for Sequentially Cleared
             Energy Markets with a Case Study on Gas and Carbon Aware
             Unit Commitment},
  journal = {Computers \& Industrial Engineering},
  volume  = {218},
  pages   = {112112},
  year    = {2026},
  doi     = {10.1016/j.cie.2026.112112}
}

@inproceedings{zhao2026ffolayer,
  author    = {Zihao Zhao and Kai-Chia Mo and Shing-Hei Ho and Brandon Amos and Kai Wang},
  title     = {A Fully First-Order Layer for Differentiable Optimization},
  booktitle = {Proceedings of the 43rd International Conference on Machine Learning},
  year      = {2026},
  note      = {Spotlight},
  eprint    = {2512.02494},
  archiveprefix = {arXiv},
  primaryclass  = {cs.LG}
}

@article{ren2025interlayer,
  author  = {Yong Ren and Corey Dickman and Matthew Pantano
             and Edward Reutzel and Qian Wang},
  title   = {Control of Interlayer Temperature through Laser Power
             Optimization in Powder Bed Fusion},
  journal = {Journal of Manufacturing Processes},
  volume  = {135},
  pages   = {189--203},
  year    = {2025},
  doi     = {10.1016/j.jmapro.2024.12.045}
}

@article{kavas2023layer,
  author  = {Bar{\i}\c{s} Kavas and Efe C. Balta and Michael Tucker
             and Alisa Rupenyan and John Lygeros and Markus Bambach},
  title   = {Layer-to-Layer Closed-Loop Feedback Control Application
             for Inter-Layer Temperature Stabilization in Laser Powder
             Bed Fusion},
  journal = {Additive Manufacturing},
  volume  = {78},
  pages   = {103847},
  year    = {2023},
  doi     = {10.1016/j.addma.2023.103847}
}

@article{riensche2022feedforward,
  author  = {Alex Riensche and Benjamin D. Bevans and Ziyad Smoqi
             and Reza Yavari and Ajay Krishnan and Josie Gilligan
             and Nicholas Piercy and Kevin Cole and Prahalada Rao},
  title   = {Feedforward Control of Thermal History in Laser Powder Bed
             Fusion: Toward Physics-Based Optimization of Processing
             Parameters},
  journal = {Materials \& Design},
  volume  = {224},
  pages   = {111351},
  year    = {2022},
  doi     = {10.1016/j.matdes.2022.111351}
}

@article{alam2021slmcontrol,
  author  = {Tonia-Maria Alam and Serge Nicaise and Luc Paquet},
  title   = {An Optimal Control Problem Governed by the Heat Equation
             with Nonconvex Constraints Applied to the Selective Laser
             Melting Process},
  journal = {Minimax Theory and its Applications},
  volume  = {6},
  number  = {2},
  pages   = {191--204},
  year    = {2021}
}

@inproceedings{lenga2020continual,
  author    = {Matthias Lenga and Heinrich Schulz and Axel Saalbach},
  title     = {Continual Learning for Domain Adaptation in Chest X-ray Classification},
  booktitle = {Proceedings of the Third Conference on Medical Imaging with Deep Learning},
  series    = {Proceedings of Machine Learning Research},
  volume    = {121},
  pages     = {413--423},
  year      = {2020},
  publisher = {PMLR},
  url       = {https://proceedings.mlr.press/v121/lenga20a.html}
}

@article{kirkpatrick2017overcoming,
  author  = {James Kirkpatrick and Razvan Pascanu and Neil Rabinowitz
             and Joel Veness and Guillaume Desjardins and Andrei A. Rusu
             and Kieran Milan and John Quan and Tiago Ramalho
             and Agnieszka Grabska-Barwinska and Demis Hassabis
             and Claudia Clopath and Dharshan Kumaran and Raia Hadsell},
  title   = {Overcoming Catastrophic Forgetting in Neural Networks},
  journal = {Proceedings of the National Academy of Sciences},
  volume  = {114},
  number  = {13},
  pages   = {3521--3526},
  year    = {2017},
  doi     = {10.1073/pnas.1611835114}
}

@inproceedings{xie2023classroom,
  title     = {Improving Bi-level Optimization Based Methods with Inspiration from Humans' Classroom Study Techniques},
  author    = {Xie, Pengtao},
  booktitle = {Proceedings of the 40th International Conference on Machine Learning},
  series    = {Proceedings of Machine Learning Research},
  volume    = {202},
  pages     = {38137--38186},
  year      = {2023},
  publisher = {PMLR},
  url       = {https://proceedings.mlr.press/v202/xie23a.html}
}

@article{guo2025mlomae,
  title   = {Downstream Task Guided Masking Learning in Masked Autoencoders Using Multi-Level Optimization},
  author  = {Guo, Han and Hosseini, Ramtin and Zhang, Ruiyi and Somayajula, Sai Ashish and Roy Chowdhury, Ranak and Gupta, Rajesh K. and Xie, Pengtao},
  journal = {Transactions on Machine Learning Research},
  year    = {2025},
  url     = {https://openreview.net/forum?id=cFmmaxkD5A}
}

@article{dizon2024differential,
  title         = {Differential Estimates for Fast First-Order Multilevel Nonconvex Optimisation},
  author        = {Dizon, Neil and Valkonen, Tuomo},
  journal       = {arXiv preprint arXiv:2412.01481},
  year          = {2024},
  eprint        = {2412.01481},
  archivePrefix = {arXiv},
  primaryClass  = {math.OC},
  url           = {https://arxiv.org/abs/2412.01481}
}

@inproceedings{li2020endtoendgames,
  author    = {Li, Jiayang and Yu, Jing and Nie, Yu and Wang, Zhaoran},
  title     = {End-to-End Learning and Intervention in Games},
  booktitle = {Advances in Neural Information Processing Systems},
  volume    = {33},
  year      = {2020}
}

@inproceedings{fiez2020stackelberg,
  author    = {Fiez, Tanner and Chasnov, Benjamin and Ratliff, Lillian},
  title     = {Implicit Learning Dynamics in Stackelberg Games: Equilibria Characterization, Convergence Analysis, and Empirical Study},
  booktitle = {Proceedings of the 37th International Conference on Machine Learning},
  series    = {Proceedings of Machine Learning Research},
  volume    = {119},
  pages     = {3133--3144},
  year      = {2020},
  publisher = {PMLR}
}

@article{shafiei2024trilevel,
  author  = {Shafiei, Allahkaram and Kungurtsev, Vyacheslav and Marecek, Jakub},
  title   = {Trilevel and Multilevel Optimization Using Monotone Operator Theory},
  journal = {Mathematical Methods of Operations Research},
  volume  = {99},
  number  = {1},
  pages   = {77--114},
  year    = {2024},
  doi     = {10.1007/s00186-024-00852-5}
}

@inproceedings{choe2023betty,
  author    = {Choe, Sang Keun and Neiswanger, Willie and Xie, Pengtao and Xing, Eric},
  title     = {Betty: An Automatic Differentiation Library for Multilevel Optimization},
  booktitle = {International Conference on Learning Representations},
  year      = {2023}
}

@article{gu2024differentiable,
  author        = {Gu, Yuntian and Chen, Xuzheng},
  title         = {Towards Differentiable Multilevel Optimization: A Gradient-Based Approach},
  journal       = {arXiv preprint arXiv:2410.11312},
  year          = {2024},
  eprint        = {2410.11312},
  archivePrefix = {arXiv}
}

@book{dempe2002foundations,
  author    = {Dempe, Stephan},
  title     = {Foundations of Bilevel Programming},
  series    = {Nonconvex Optimization and Its Applications},
  volume    = {61},
  publisher = {Springer},
  address   = {New York},
  year      = {2002}
}

@article{colson2007overview,
  author  = {Colson, Beno{\^\i}t and Marcotte, Patrice and Savard, Gilles},
  title   = {An Overview of Bilevel Optimization},
  journal = {Annals of Operations Research},
  volume  = {153},
  number  = {1},
  pages   = {235--256},
  year    = {2007},
  doi     = {10.1007/s10479-007-0176-2}
}

@book{luo1996mpec,
  author    = {Luo, Zhi-Quan and Pang, Jong-Shi and Ralph, Daniel},
  title     = {Mathematical Programs with Equilibrium Constraints},
  publisher = {Cambridge University Press},
  address   = {Cambridge},
  year      = {1996},
  doi       = {10.1017/CBO9780511983658}
}

@inproceedings{domke2012generic,
  author    = {Domke, Justin},
  title     = {Generic Methods for Optimization-Based Modeling},
  booktitle = {Proceedings of the 15th International Conference on Artificial Intelligence and Statistics},
  series    = {Proceedings of Machine Learning Research},
  volume    = {22},
  pages     = {318--326},
  year      = {2012},
  publisher = {PMLR}
}

@inproceedings{maclaurin2015reversible,
  author    = {Maclaurin, Dougal and Duvenaud, David and Adams, Ryan P.},
  title     = {Gradient-Based Hyperparameter Optimization through Reversible Learning},
  booktitle = {Proceedings of the 32nd International Conference on Machine Learning},
  series    = {Proceedings of Machine Learning Research},
  volume    = {37},
  pages     = {2113--2122},
  year      = {2015},
  publisher = {PMLR}
}

@inproceedings{pedregosa2016hyperparameter,
  author    = {Pedregosa, Fabian},
  title     = {Hyperparameter Optimization with Approximate Gradient},
  booktitle = {Proceedings of the 33rd International Conference on Machine Learning},
  series    = {Proceedings of Machine Learning Research},
  volume    = {48},
  pages     = {737--746},
  year      = {2016},
  publisher = {PMLR}
}

@inproceedings{franceschi2018bilevel,
  author    = {Franceschi, Luca and Frasconi, Paolo and Salzo, Saverio and Grazzi, Riccardo and Pontil, Massimiliano},
  title     = {Bilevel Programming for Hyperparameter Optimization and Meta-Learning},
  booktitle = {Proceedings of the 35th International Conference on Machine Learning},
  series    = {Proceedings of Machine Learning Research},
  volume    = {80},
  pages     = {1568--1577},
  year      = {2018},
  publisher = {PMLR}
}

@inproceedings{grazzi2020iteration,
  author    = {Grazzi, Riccardo and Franceschi, Luca and Pontil, Massimiliano and Salzo, Saverio},
  title     = {On the Iteration Complexity of Hypergradient Computation},
  booktitle = {Proceedings of the 37th International Conference on Machine Learning},
  series    = {Proceedings of Machine Learning Research},
  volume    = {119},
  pages     = {3748--3758},
  year      = {2020},
  publisher = {PMLR}
}

@inproceedings{blondel2022implicit,
  author    = {Blondel, Mathieu and Berthet, Quentin and Cuturi, Marco and Frostig, Roy and Hoyer, Stephan and Llinares-L{\'o}pez, Felipe and Pedregosa, Fabian and Vert, Jean-Philippe},
  title     = {Efficient and Modular Implicit Differentiation},
  booktitle = {Advances in Neural Information Processing Systems},
  volume    = {35},
  year      = {2022}
}

@article{ghadimi2018approximation,
  author        = {Ghadimi, Saeed and Wang, Mengdi},
  title         = {Approximation Methods for Bilevel Programming},
  journal       = {arXiv preprint arXiv:1802.02246},
  year          = {2018},
  eprint        = {1802.02246},
  archivePrefix = {arXiv}
}

@inproceedings{ji2021bilevel,
  author    = {Ji, Kaiyi and Yang, Junjie and Liang, Yingbin},
  title     = {Bilevel Optimization: Convergence Analysis and Enhanced Design},
  booktitle = {Proceedings of the 38th International Conference on Machine Learning},
  series    = {Proceedings of Machine Learning Research},
  volume    = {139},
  pages     = {4882--4892},
  year      = {2021},
  publisher = {PMLR}
}

@inproceedings{yang2021faster,
  author    = {Yang, Junjie and Ji, Kaiyi and Liang, Yingbin},
  title     = {Provably Faster Algorithms for Bilevel Optimization},
  booktitle = {Advances in Neural Information Processing Systems},
  volume    = {34},
  year      = {2021}
}

@article{hong2023twotimescale,
  author  = {Hong, Mingyi and Wai, Hoi-To and Wang, Zhaoran and Yang, Zhuoran},
  title   = {A Two-Timescale Stochastic Algorithm Framework for Bilevel Optimization: Complexity Analysis and Application to Actor-Critic},
  journal = {SIAM Journal on Optimization},
  volume  = {33},
  number  = {1},
  pages   = {147--180},
  year    = {2023},
  doi     = {10.1137/20M1387341}
}

@inproceedings{dagreou2022framework,
  author    = {Dagr{\'e}ou, Mathieu and Ablin, Pierre and Vaiter, Samuel and Moreau, Thomas},
  title     = {A Framework for Bilevel Optimization That Enables Stochastic and Global Variance Reduction Algorithms},
  booktitle = {Advances in Neural Information Processing Systems},
  volume    = {35},
  year      = {2022}
}

@inproceedings{dagreou2024lower,
  author    = {Dagr{\'e}ou, Mathieu and Moreau, Thomas and Vaiter, Samuel and Ablin, Pierre},
  title     = {A Lower Bound and a Near-Optimal Algorithm for Bilevel Empirical Risk Minimization},
  booktitle = {Proceedings of the 27th International Conference on Artificial Intelligence and Statistics},
  series    = {Proceedings of Machine Learning Research},
  volume    = {238},
  pages     = {82--90},
  year      = {2024},
  publisher = {PMLR}
}

@inproceedings{li2022fullysingle,
  author    = {Li, Junyi and Gu, Bin and Huang, Heng},
  title     = {A Fully Single Loop Algorithm for Bilevel Optimization without Hessian Inverse},
  booktitle = {Proceedings of the AAAI Conference on Artificial Intelligence},
  volume    = {36},
  pages     = {7426--7434},
  year      = {2022},
  doi       = {10.1609/aaai.v36i7.20706}
}

@inproceedings{liu2022bome,
  author    = {Liu, Bo and Ye, Mao and Wright, Stephen and Stone, Peter and Liu, Qiang},
  title     = {BOME! Bilevel Optimization Made Easy: A Simple First-Order Approach},
  booktitle = {Advances in Neural Information Processing Systems},
  volume    = {35},
  pages     = {17248--17262},
  year      = {2022}
}

@inproceedings{sow2022hessianfree,
  author    = {Sow, Daouda and Ji, Kaiyi and Liang, Yingbin},
  title     = {On the Convergence Theory for Hessian-Free Bilevel Algorithms},
  booktitle = {Advances in Neural Information Processing Systems},
  volume    = {35},
  year      = {2022}
}

@inproceedings{kwon2023fullyfirst,
  author    = {Kwon, Jeongyeol and Kwon, Dohyun and Wright, Stephen and Nowak, Robert D.},
  title     = {A Fully First-Order Method for Stochastic Bilevel Optimization},
  booktitle = {Proceedings of the 40th International Conference on Machine Learning},
  series    = {Proceedings of Machine Learning Research},
  volume    = {202},
  pages     = {18083--18113},
  year      = {2023},
  publisher = {PMLR}
}

@inproceedings{kornowski2024constrained,
  author    = {Kornowski, Guy and Padmanabhan, Swati and Wang, Kai and Zhang, Jimmy and Sra, Suvrit},
  title     = {First-Order Methods for Linearly Constrained Bilevel Optimization},
  booktitle = {Advances in Neural Information Processing Systems},
  volume    = {37},
  year      = {2024},
  doi       = {10.52202/079017-4491}
}

@article{chen2025f2ba,
  author  = {Chen, Lesi and Ma, Yaohua and Zhang, Jingzhao},
  title   = {Near-Optimal Nonconvex-Strongly-Convex Bilevel Optimization
             with Fully First-Order Oracles},
  journal = {Journal of Machine Learning Research},
  volume  = {26},
  number  = {109},
  pages   = {1--56},
  year    = {2025}
}

@inproceedings{sato2021multilevel,
  author    = {Sato, Ryo and Tanaka, Mirai and Takeda, Akiko},
  title     = {A Gradient Method for Multilevel Optimization},
  booktitle = {Advances in Neural Information Processing Systems},
  volume    = {34},
  year      = {2021}
}

@article{wang2017scgd,
  author  = {Wang, Mengdi and Fang, Ethan X. and Liu, Han},
  title   = {Stochastic Compositional Gradient Descent: Algorithms for Minimizing Compositions of Expected-Value Functions},
  journal = {Mathematical Programming},
  volume  = {161},
  pages   = {419--449},
  year    = {2017}
}

@article{yang2019multilevel,
  author  = {Yang, Shuoguang and Wang, Mengdi and Fang, Ethan X.},
  title   = {Multilevel Stochastic Gradient Methods for Nested Composition Optimization},
  journal = {SIAM Journal on Optimization},
  volume  = {29},
  number  = {1},
  pages   = {616--659},
  year    = {2019},
  doi     = {10.1137/18M1164846}
}

@article{zhang2021nested,
  author  = {Zhang, Junyu and Xiao, Lin},
  title   = {Multilevel Composite Stochastic Optimization via Nested Variance Reduction},
  journal = {SIAM Journal on Optimization},
  volume  = {31},
  number  = {2},
  pages   = {1131--1157},
  year    = {2021},
  doi     = {10.1137/19M1285457}
}

@article{balasubramanian2022multilevel,
  author  = {Balasubramanian, Krishnakumar and Ghadimi, Saeed and Nguyen, Anthony},
  title   = {Stochastic Multilevel Composition Optimization Algorithms with Level-Independent Convergence Rates},
  journal = {SIAM Journal on Optimization},
  volume  = {32},
  number  = {2},
  pages   = {519--544},
  year    = {2022},
  doi     = {10.1137/21M1406222}
}

@inproceedings{jiang2022optimal,
  author    = {Jiang, Wei and Wang, Bokun and Wang, Yibo and Zhang, Lijun and Yang, Tianbao},
  title     = {Optimal Algorithms for Stochastic Multi-Level Compositional Optimization},
  booktitle = {Proceedings of the 39th International Conference on Machine Learning},
  series    = {Proceedings of Machine Learning Research},
  volume    = {162},
  pages     = {10195--10216},
  year      = {2022},
  publisher = {PMLR}
}

@inproceedings{arbel2022amortized,
  title     = {Amortized Implicit Differentiation for Stochastic Bilevel Optimization},
  author    = {Arbel, Michael and Mairal, Julien},
  booktitle = {International Conference on Learning Representations},
  year      = {2022}
}
\end{document}